\documentclass[final,leqno,onefignum,onetabnum]{siamart220329}

\usepackage{epsfig}
\usepackage{changebar}
\usepackage{pstricks}
\usepackage{pst-plot}
\usepackage{multirow}
\usepackage{subfigure}
\usepackage{placeins}
\usepackage{amssymb}
\usepackage{mathrsfs}
\usepackage{amssymb,hyperref,bm,color,booktabs}
\usepackage{comment}
\usepackage{enumerate}

\newtheorem{thm}{Theorem}[section]
\newtheorem{defn}{Definition}[section]
\newtheorem{prop}{Proposition}[section]
\newtheorem{lem}{Lemma}[section]
\newtheorem{cor}{Corollary}[section]
\newtheorem{rem}{Remark}[section]
\theoremstyle{definition}
\newtheorem{exam}{Example}[section]

\newcommand{\diag}{\mbox{diag}}
\newcommand{\brac}[1]{\left(#1\right)}

\def\half{\frac 1 2}
\newcommand{\ct}{\cos\theta}
\newcommand{\st}{\sin\theta}

\headers{Shallow water moment equations in two dimensions}
{M. Bauerle, A. J. Christlieb, M Ding, and J. Huang}

\title{On the rotational invariance and hyperbolicity of shallow water moment equations in two dimensions
}

\author{Matthew Bauerle\thanks{Department of Mathematics, Department of Computational Mathematics, Science and Engineering, Michigan State University, East Lansing, MI 48824, USA. (\email{bauerle2@msu.edu}).}  
\and Andrew J. Christlieb\thanks{Department of Mathematics, Department of Computational Mathematics, Science and Engineering, Michigan State University, East Lansing, Michigan 48824, USA.
  (\email{christli@msu.edu}).}
\and Mingchang Ding\thanks{Department of Computational Mathematics, Science and Engineering, Michigan State University, East Lansing, MI 48824, USA. (\email{dingmin2@msu.edu}).}
\and Juntao Huang\thanks{Department of Mathematics and Statistics, Texas Tech University, Lubbock, TX, 79409, USA. (\email{juntao.huang@ttu.edu}).}}

\begin{document}

\maketitle

\begin{abstract}
In this paper, we investigate the two-dimensional extension of a recently introduced set of shallow water models based on a regularized moment expansion of the incompressible Navier-Stokes equations \cite{kowalski2017moment,koellermeier2020analysis}. We show the rotational invariance of the proposed moment models with two different approaches. The first proof involves the split of the coefficient matrix into the conservative and non-conservative parts and  {proves} the rotational invariance for each part, while the second one relies on the special block structure of the coefficient matrices. With the aid of rotational invariance, the analysis of the hyperbolicity for the moment model in 2D is reduced to the real diagonalizability of the coefficient matrix in 1D. Then we analyze the real diagonalizability by deriving the analytical form of the characteristic polynomial.
 {We find that the moment model in 2D is hyperbolic in most cases and weakly hyperbolic in a degenerate edge case. With a simple modification to the coefficient matrices, we fix this weakly hyperbolicity and propose a new global hyperbolic model.}
Furthermore, we extend the model to include a more general class of closure relations than the original model and establish that this set of general closure relations  {retains} both rotational invariance and hyperbolicity.
\end{abstract}



\section{Introduction}

This paper deals with the two-dimensional extension of a recently introduced set of shallow water models based on a regularized moment expansion of the incompressible Navier-Stokes (NS) equations \cite{kowalski2017moment,koellermeier2020analysis}. Moment models have a long history of constructing computationally efficient representations of complex dynamics arising in a higher dimensional model.  Historically, this approach has been widely used in the reduction of kinetic equations to hierarchies of moment models, where the resulting evolution equation for the $m^{th}$ moment equation depends on the $(m+1)^{st}$ moment \cite{grad1949kinetic}.  At some point, one needs to restrict the expansion to $m$ moments and introduce a model for the $(m+1)^{st}$ moment in terms of lower moments to obtain a solvable system of equations.  
This leads to the well-known moment closure problem, which is the study of what kind of model preserves the desired  hyperbolic structure \cite{levermore1996moment,cai2013globally,cai2014globally,fox2021hyperbolic,fan2016model}.

In this work, we are looking at a new class of models for describing systems traditionally modeled with shallow water equations \cite{vreugdenhil1994numerical} and multi-layer shallow water equations \cite{bouchut2010robust,stewart2010multilayer,fernandez2014multilayer,fernandez2016multilayer}. In \cite{kowalski2017moment}, the new class of models, called the shallow water moment equations (SWME), was derived by taking moments with respect to the Legendre polynomials of the vertical direction of the three-dimensional (3D) incompressible NS equations.  {The zeroth order moment of the momentum equation together with the incompressibility condition yield the traditional shallow water equations.} 
In principle, higher-order moments offer an approach to include vertical information without introducing a mesh in the vertical direction, as one would need for the 3D NS equations. Fundamentally, this is seeking to address a multi-scale problem by providing a path to increased fidelity of the flow dynamics without needing to introduce a mesh to resolve the vertical direction. The system as originally proposed is mathematically elegant, but does not preserve hyperbolicity for higher numbers of moments in 1D and 2D.  The developers of the model realized this and introduced a class of regularizations in 1D called the hyperbolic shallow water moment equations (HSWME) and proceeded to show that this regularized system was provably hyperbolic \cite{koellermeier2020analysis,huang2022equilibrium}. In addition, they have extended the 1D system to a variety of settings and compared depth averaged results of the incompressible 2D Navier Stokes to the regularized moment based models.  In their work, they demonstrated increased fidelity when adding additional moments \cite{kowalski2017moment,koellermeier2020analysis}. 
More recently, there are some follow-up works, including the equilibrium stability analysis \cite{huang2022equilibrium}, well-balanced schemes \cite{koellermeier2022steady}, efficient time discretizations \cite{amrita2022projective}, {axisymmetric model \cite{verbiest2023hyperbolic,verbiest2022}}, multilayer-moment models \cite{garres2023general}, and extension to sediment transport \cite{garres2021shallow}.
 {In a very recent work \cite{scholz2023dispersion}, the non-hydrostatic pressure and the velocity in the vertical direction were taken into consideration based on a splitting of the pressure and a polynomial expansion of the velocity and pressure. It has been demonstrated that the new reduced model retains the correct dispersive effects of the original system.}
{We also mention that in \cite{verbiest2022} the explicit form and the characteristic polynomial of the coefficient matrices of the HSWME in 2D are analyzed. However, the proof in \cite{verbiest2022} only shows that the eigenvalues of the two coefficient matrices in $x$ and $y$ directions are real but not necessarily distinct. Moreover, it is not enough to conclude that the HSWME in 2D is hyperbolic since the linear combination of the two matrices is not necessarily real diagonalizable.}

In this work, our main result is establishing that the 2D extension to the regularized moment expansion of the incompressible NS equations is both rotational invariant and hyperbolic for an arbitrary number of moments. We present two different proofs of the rotational invariance. The first proof involves the split of the coefficient matrix into the conservative and non-conservative parts and proves the rotational invariance for each part, see section \ref{subsec:rotational-invariance-SWME-HSWME}. The second proof relies on the special block structure of the coefficient matrix, see section \ref{subsec:alternative-proof-rotation-invariance}.
With the aid of rotational invariance, the hyperbolicity in 2D reduces to the real diagonalizability of the matrix in 1D. To analyze the eigenvalues, we make use of the associated polynomial sequence and derive the characteristic polynomial of the coefficient matrix analytically. We find that the eigenvalues are related to the Gauss-Lobatto quadrature points (i.e. the zeros of derivative of Legendre polynomials) and Gauss-Legendre quadrature points (i.e. the zeros of Legendre polynomials). In particular, we show that the eigenvalues of the moment system are real and distinct for arbitrary number of moments, see section \ref{subsubsec:alpha1-ne0}. 
Moreover, we find that the moment model is weakly hyperbolic in a degenerate edge case, see section \ref{subsubsec:alpha1=0}. With a simple modification to the coefficient matrices, we fix this weakly hyperbolicity and propose a new global hyperbolic model, see section \ref{subsec:globally-hyperbolic-HSWME}.
More importantly, we establish the general closure relations such that the rotational invariance and hyperbolicity are guaranteed, see sections \ref{subsec:framework-general-closure} and \ref{subsec:example-general-closure}. This opens the door to the development of data-driven closures that preserve hyperbolicity, as in our past work for the radiation transfer equations \cite{huang2021gradient,huang2023machine2,huang2023machine3}.  Data-driven closures are the subject of our future work. In addition, we conduct numerical simulations to compare the performance of different models, see section \ref{sec:numerical}.

The remaining parts of this paper are structured as follows: Section \ref{sec:model} introduces the models. In Section \ref{sec:rotation-invariance}, we show two proofs to the rotational invariance of the models. Section \ref{sec:hyperbolic} analyzes the eigen structure of the models and establishes hyperbolicity of the models. In Section \ref{sec:numerical}, we show the numerical results for different models. In Section \ref{sec:conclusion}, we give conclusions and talk about future directions.    

\setcounter{equation}{0}
\setcounter{figure}{0}
\setcounter{table}{0}

\section{Shallow water moment equations}\label{sec:model}

In this section, we review the main ideas and the results for the derivation of the shallow water moment model in \cite{kowalski2017moment}. We also show the 1D  hyperbolic shallow water moment model proposed in \cite{koellermeier2020analysis}.

We start by considering  the 3D incompressible Navier-Stokes equations:
\begin{equation}\label{eq:in-NS}
\begin{aligned}
\nabla\cdot U &= 0, \\
\partial_t U + \nabla\cdot(UU) &= -\frac{1}{\rho} \nabla p + \frac{1}{\rho} \nabla\cdot{\sigma} + {g}. 
\end{aligned}
\end{equation}
Here, $U = (u,v,w)^T$ is the velocity vector, $p$ is the pressure and $\sigma$ is the stress tensor. The density $\rho$ is constant and $g=(e_x,e_y,e_z)g$ with $(e_x,e_y,e_z)$ a constant unit vector denotes the gravitational acceleration. The often used shallow water coordinate system is recovered by choosing $e_x=e_y=0$ and $g=(0,0,-1)^Tg$.

Under the shallowness assumption, i.e., the horizontal scales of the flow are much larger than the vertical scale, the Navier-Stokes equations \eqref{eq:in-NS} can be reduced by an asymptotic analysis to \cite{kowalski2017moment}:
\begin{equation}\label{eq:in-NS-reduce}
\begin{aligned}
\partial_x u + \partial_y v + \partial_z w &= 0, \\
\partial_t u + \partial_x(u^2) + \partial_y(uv) + \partial_z(uw) &= -\frac{1}{\rho}\partial_x p + \frac{1}{\rho}\partial_z \sigma_{xz} + g e_x, \\
\partial_t v + \partial_x(uv) + \partial_y(v^2) + \partial_z(vw) &= -\frac{1}{\rho}\partial_y p + \frac{1}{\rho}\partial_z \sigma_{yz} + g e_y,
\end{aligned}    
\end{equation}
where the hydrostatic pressure is given by
\begin{equation*}
p(x,y,z,t) = (h_s(x,y,t) - z) \rho g e_z,
\end{equation*}
with $h_s(x,y,t)$ being the profile of the upper free surface.

To derive the shallow water moment model from \eqref{eq:in-NS-reduce}, the first idea in \cite{kowalski2017moment} is to introduce a scaled vertical variable $\zeta(x,y,t)$ given by
\begin{equation*}
\zeta(x,y,t) := \frac{z - h_b(x,y,t)}{h_s(x,y,t) - h_b(x,y,t)} = \frac{z - h_b(x,y,t)}{h(x,y,t)},
\end{equation*}
with $h(x,y,t) := h_s(x,y,t) - h_b(x,y,t)$ is the water height from the bottom $h_b(x,y,t)$ to the surface $h_s(x,y,t)$. This transforms the $z$-direction from a physical space $z\in[h_b, h_s]$ to a projected space $\zeta\in[0,1]$. For any function $\psi=\psi(x,y,z,t)$, the corresponding mapped function $\tilde{\psi}=\tilde{\psi}(x,y,\zeta,t)$ is given by
\begin{equation*}
\tilde{\psi}(x,y,\zeta,t) := \psi(x,y,h(x,y,t)\zeta + h_b(x,y,t)).
\end{equation*}

The complete vertically resolved shallow flow system has the form \cite{kowalski2017moment}
\begin{equation*}
\begin{aligned}
\partial_t h + \partial_x (hu_m) + \partial_y (hv_m) &= 0, \\
\partial_t (h\tilde{u}) + \partial_x (h\tilde{u}^2 + \frac{g}{2}e_zh^2) + \partial_y (h\tilde{u}\tilde{v}) + \partial_{\zeta}(h\tilde{u}\omega - \frac{1}{\rho}\tilde{\sigma}_{xz}) &= gh(e_x - e_z\partial_x h_b), \\
\partial_t (h\tilde{v}) + \partial_x (h\tilde{u}\tilde{v}) + \partial_y (h\tilde{v}^2 + \frac{g}{2}e_zh^2) + \partial_{\zeta}(h\tilde{v}\omega - \frac{1}{\rho}\tilde{\sigma}_{yz}) &= gh(e_y - e_z\partial_y h_b),
\end{aligned}    
\end{equation*}
where $u_m(x,y,t) = \int_0^1 u(x,y,\zeta,t)d\zeta$ and $v_m(x,y,t) = \int_0^1 v(x,y,\zeta,t)d\zeta$ denote the mean velocities and $\omega$ is the vertical coupling 
\begin{equation*}
\omega = \frac{1}{h} \overline{\partial_x(h\tilde{u}) + \partial_y(h\tilde{v})},
\end{equation*}
with the average for any function $\psi=\psi(\zeta)$ defined by
\begin{equation*}
\bar{\psi}(\zeta) := \int_0^{\zeta} \brac{\int_0^1 \psi(\check{\zeta}) d\check{\zeta} - \psi(\hat{\zeta}) } d\hat{\zeta}.
\end{equation*}
Note that, for a constant flow profile in $\zeta$, the vertical coupling coefficient $\omega$ vanishes. In that case, if in addition shear stresses are negligible $\sigma_{xz}=\sigma_{yz}=0$, the system reduces to the shallow water equations.

Before deriving the moment equation, we introduce some assumptions. First, we use the Newtonian constitutive law:
\begin{equation*}
    \sigma_{xz} = \mu\partial_z u, \quad \sigma_{yz} = \mu\partial_z v,
\end{equation*}
where $\mu$ stands for the dynamic viscosity and $\nu=\mu/\rho$ the kinematic viscosity. In order to solve it, we need to specify dynamic boundary conditions in the form of a velocity boundary condition both at the free-surface, and at the bottom topography. At the free-surface, the stress-free conditions are assumed:
\begin{equation*}
    \partial_z u = \partial_z v = 0, \quad \text{at} \quad z=h_s(x,y,t).
\end{equation*}
At the basal surface, the slip boundary conditions are assumed:
\begin{equation*}
    u - \frac{\lambda}{\mu}\sigma_{xz} = v - \frac{\lambda}{\mu}\sigma_{yz} = 0, \quad \text{at} \quad z=h_b(x,y,t).
\end{equation*}
Here, $\lambda$ stands for the slip length.

By assuming a polynomial expansion of the velocity components:
\begin{equation*}
    u(x,y, {\zeta},t) = u_m(x,y,t) + \sum_{j=1}^N \alpha_j(x,y,t)\phi_j( {\zeta}),
\end{equation*}
\begin{equation*}
    v(x,y, {\zeta},t) = v_m(x,y,t) + \sum_{j=1}^N \beta_j(x,y,t)\phi_j( {\zeta}),
\end{equation*}
with the scaled Legendre polynomials $\phi_j$, orthogonal on the interval $[0,1]$ and normalized by $\phi_j(0)=1$, the shallow water moment equations (SWME) can be derived \cite{kowalski2017moment}:
\begin{equation}\label{eq:SWME-component}
\begin{aligned}
& \partial_t h + \partial_x (h u_m) + \partial_y (h v_m) = 0, \\
& \partial_t (h u_m) + \partial_x \brac{h (u_m^2 + \sum_{j=1}^N \frac{\alpha_j^2}{2j+1}) + \frac{g}{2}e_zh^2} + \partial_y \brac{h (u_m v_m + \sum_{j=1}^N \frac{\alpha_j\beta_j}{2j+1})} \\ &\qquad\qquad\qquad\qquad\qquad\qquad\qquad\qquad\qquad\qquad\qquad = -\frac{\nu}{\lambda}(u_m + \sum_{j=1}^N\alpha_j) + hg(e_x - e_z\partial_x h_b), \\
& \partial_t (h v_m) + \partial_x \brac{h (u_m v_m + \sum_{j=1}^N \frac{\alpha_j\beta_j}{2j+1}) } + \partial_y \brac{h (v_m^2 + \sum_{j=1}^N \frac{\beta_j^2}{2j+1}) + \frac{g}{2}e_zh^2 } \\ &\qquad\qquad\qquad\qquad\qquad\qquad\qquad\qquad\qquad\qquad\qquad = -\frac{\nu}{\lambda}(v_m + \sum_{j=1}^N\beta_j) + hg(e_y - e_z\partial_y h_b), \\
& \partial_t (h\alpha_i) + \partial_x \brac{h (2u_m \alpha_i + \sum_{j,k=1}^N A_{ijk} \alpha_j\alpha_k)} + \partial_y \brac{h (u_m \beta_i + v_m \alpha_i + \sum_{j,k=1}^N A_{ijk} \alpha_j\beta_k)} \\
&\qquad\qquad = u_m D_i - \sum_{j,k=1}^N B_{ijk} D_j \alpha_k - (2i+1)\frac{\nu}{\lambda} \brac{u_m + \sum_{j=1}^N(1+\frac{\lambda}{h}C_{ij})\alpha_j}, \quad i=1,2,\cdots,N, \\
& \partial_t (h\beta_i) + \partial_x \brac{h (u_m\beta_i +v_m\alpha_i + \sum_{j,k=1}^N A_{ijk} \alpha_j\beta_k)} + \partial_y \brac{h (2v_m \beta_i + \sum_{j,k=1}^N A_{ijk} \beta_j\beta_k)} \\
&\qquad\qquad = v_m D_i - \sum_{j,k=1}^N B_{ijk} D_j \beta_k - (2i+1)\frac{\nu}{\lambda} \brac{v_m + \sum_{j=1}^N(1+\frac{\lambda}{h}C_{ij})\beta_j}, \quad i=1,2,\cdots,N.
\end{aligned}    
\end{equation}
Here the right-hand-side (RHS) contains non-conservative terms involving the expression
\begin{equation*}
D_i := \partial_x(h \alpha_i) + \partial_y(h \beta_i)
\end{equation*}
and the constants $A_{ijk}$, $B_{ijk}$, $C_{ij}$ are related to the integrals of the Legendre polynomials:
\begin{equation*}
\begin{aligned}
A_{ijk} &= (2i+1)\int_0^1 \phi_i \phi_j \phi_k d\zeta, \quad i,j,k=1,\cdots,N, \\
B_{ijk} &= (2i+1) \int_0^1 \phi_i'\brac{\int_0^{\zeta} \phi_j d\hat{\zeta}} \phi_k d\zeta, \quad i,j,k=1,\cdots,N, \\
C_{ij} &= \int_0^1 \phi_i' \phi_j'd\zeta, \quad i,j=1,\cdots,N.
\end{aligned}
\end{equation*}

The above system \eqref{eq:SWME-component} can be written as
\begin{equation}\label{eq:SWME}
\partial_t U + A(U) \partial_x U + B(U) \partial_y U = S(U),
\end{equation}
with the unknown variables
\begin{equation*}
U = (h, \,  {hu_m}, \,  {hv_m}, \, h\alpha_1, \, h\beta_1, \, h\alpha_2, \, h\beta_2, \, \cdots, h\alpha_N, \, h\beta_N)^T.
\end{equation*}
The coefficient matrices in \eqref{eq:SWME} can be split into the conservative and non-conservative parts:
\begin{equation}\label{eq:SWME-matrix-split-conser-nonconser}
A(U) = \partial_U F(U) + P(U), \quad B(U) = \partial_U G(U) + Q(U).
\end{equation}
Here the physical fluxes $F(U)$ and $G(U)$ for the conservative parts are
\begin{equation}\label{eq:SWME-physical-flux-x-dimension}
\begin{aligned}
F(U)=
(h  {u_m}, \,
& h( {u_m^{2}}+\sum_{j=1}^N\frac{\alpha_j^2}{2j+1})+\frac{1}{2} g h^{2}, \,  
h( {u_m}  {v_m}+\sum_{j=1}^N\frac{\alpha_j\beta_j}{2j+1}), \\
& h(2 {u_m}\alpha_1 + \sum_{j,k=1}^N A_{1jk}\alpha_j\alpha_k), \,
h( {u_m}\beta_1 +  {v_m}\alpha_1 + \sum_{j,k=1}^N A_{1jk}\alpha_j\beta_k), \\
& \cdots, \\
& h(2 {u_m}\alpha_n + \sum_{j,k=1}^N A_{njk}\alpha_j\alpha_k),
h( {u_m}\beta_N +  {v_m}\alpha_n + \sum_{j,k=1}^N A_{njk}\alpha_j\beta_k))^T,
\end{aligned}
\end{equation}
\begin{equation}\label{eq:SWME-physical-flux-y-dimension}
\begin{aligned}
G(U) =
(h {v_m}, \,
& h( {u_m} {v_m}+\sum_{j=1}^N\frac{\alpha_j\beta_j}{2j+1}), \, h( {v_m}^2+\sum_{j=1}^N\frac{\beta_j^2}{2j+1})+\half gh^2, \\
& h( {u_m}\beta_1 +  {v_m}\alpha_1 + \sum_{j,k=1}^N A_{1jk}\alpha_j\beta_k), \,
h(2 {v_m}\beta_1 + \sum_{j,k=1}^N A_{1jk}\beta_j\beta_k), \\
& \cdots, \\
& h( {u_m}\beta_n +  {v_m}\alpha_n + \sum_{j,k=1}^N A_{njk}\alpha_j\beta_k), \,
h(2 {v_m}\beta_n + \sum_{j,k=1}^N A_{njk}\beta_j\beta_k))^T.
\end{aligned}
\end{equation}
The matrices for the non-conservative part can be further decomposed into two parts:
\begin{equation}\label{eq:matrix-non-conservative}
P(U) = P_1(U) + P_2(U), \quad Q(U) = Q_1(U) + Q_2(U).
\end{equation}
Here $P_1(U)$ and $Q_1(U)$ describe the terms $u_m D_i$ and $v_m D_i$ on the RHS of \eqref{eq:SWME-component}:
\begin{equation}\label{eq:matrix-non-conservative-part1-x}
P_1(U) = \diag(0_{3\times3},p(U),\cdots,p(U))\in\mathbb{R}^{(2N+3)\times(2N+3)}
\end{equation}
and
\begin{equation}\label{eq:matrix-non-conservative-part1-y}
Q_1(U) = \diag(0_{3\times3},q(U),\cdots,q(U))\in\mathbb{R}^{(2N+3)\times(2N+3)}
\end{equation}
with
\begin{equation}\label{eq:matrix-small-p-q}
p(U) = 
\begin{pmatrix}
- {u_m} & 0 \\
- {v_m} & 0  
\end{pmatrix} 
\in\mathbb{R}^{2\times2},\quad
q(U) = 
\begin{pmatrix}
0 & - {u_m} \\
0 & - {v_m}
\end{pmatrix}
\in\mathbb{R}^{2\times2}.
\end{equation}
The second part $P_2(U)$ and $Q_2(U)$ describe the terms $\sum_{j,k=1}^N B_{ijk}D_j\alpha_k$ and $\sum_{j,k=1}^N B_{ijk}D_j\beta_k$ on the RHS of \eqref{eq:SWME-component}:
\begin{equation}\label{eq:matrix-non-conservative-part2-x}
P_2(U) = \diag(0_{3\times3}, G(U))\in\mathbb{R}^{(2N+3)\times(2N+3)}
\end{equation}
and
\begin{equation}\label{eq:matrix-non-conservative-part2-y}
Q_2(U) = \diag(0_{3\times3}, H(U))\in\mathbb{R}^{(2N+3)\times(2N+3)}
\end{equation}
with
\begin{equation}
G(U) = (g_{ij})_{1\le i,j\le N} \in\mathbb{R}^{2N \times 2N},\quad H(U) = (h_{ij})_{1\le i,j\le N} \in\mathbb{R}^{2N \times 2N}
\end{equation}
and
\begin{equation}\label{eq:matrix-non-conservative-sub-block}
g_{ij}(U) = 
\begin{pmatrix}
\sum_{k}B_{ijk}\alpha_k & 0 \\
\sum_{k}B_{ijk}\beta_k & 0
\end{pmatrix}
\in\mathbb{R}^{2\times2}
,\quad
h_{ij}(U) = 
\begin{pmatrix}
0 & \sum_{k}B_{ijk}\alpha_k \\
0 & \sum_{k}B_{ijk}\beta_k
\end{pmatrix}
\in\mathbb{R}^{2\times2}.
\end{equation}
 {The source term $S(U)$ in \eqref{eq:SWME} reads as
\begin{equation*}
\begin{aligned}    
    S(U) = (0, \, & -\frac{\nu}{\lambda}(u_m + \sum_{j=1}^N\alpha_j) + hg(e_x - e_z\partial_xh_b),  \, -\frac{\nu}{\lambda}(v_m + \sum_{j=1}^N\beta_j) + hg(e_y - 
    e_z\partial_yh_b), \\
    & - \frac{3\nu}{\lambda} (u_m + \sum_{j=1}^N(1+\frac{\lambda}{h}C_{1j})\alpha_j),  \, - \frac{3\nu}{\lambda} (v_m + \sum_{j=1}^N(1+\frac{\lambda}{h}C_{1j})\beta_j), \\
    & \dots, \\
    & - (2N+1)\frac{\nu}{\lambda} (u_m + \sum_{j=1}^N(1+\frac{\lambda}{h}C_{Nj})\alpha_j),  \, -(2N+1)\frac{\nu}{\lambda} (v_m + \sum_{j=1}^N(1+\frac{\lambda}{h}C_{Nj})\beta_j))^T.
\end{aligned}
\end{equation*}
}
This system \eqref{eq:SWME} is called the shallow water moment equations (SWME). {We also remark that other models are available in \cite{kowalski2017moment}.}

In the one-dimensional case, the SWME \eqref{eq:SWME} with $N\ge 2$ is not globally hyperbolic. In \cite{koellermeier2020analysis}, the author proposed to linearize the system matrix around linear deviations from equilibrium/constant velocity.
\begin{equation}\label{eq:HSWME}
\partial_t U + A_H(U) \partial_x U + B_H(U) \partial_y U = S(U),
\end{equation}
with
\begin{equation*}
    A_H(U) := A(h, \, hu_m,  \, hv_m,  \, h\alpha_1,  \, h\beta_1,  \, 0,  \, 0,  \, \cdots,  \, 0,  \, 0),    
\end{equation*}
and
\begin{equation*}
    B_H(U) := B(h,  \, hu_m,  \, hv_m,  \, h\alpha_1,  \, h\beta_1,  \, 0,  \, 0,  \, \cdots,  \, 0,  \, 0).
\end{equation*}
Keeping $\alpha_1$ and $\beta_1$ allows to capture a large part of the structure despite its simplicity. For example, there will still be a coupling between the different higher order equations. This system is called the hyperbolic shallow water moment equations (HSWME). In 1D, it is proved to be hyperbolic in \cite{koellermeier2020analysis,huang2022equilibrium}. However, it remains unclear whether the HSWME \eqref{eq:HSWME} is hyperbolic in 2D.

\section{Analysis of rotational invariance}\label{sec:rotation-invariance}

In this section, we present two approaches to prove the rotational invariance of the SWME \eqref{eq:SWME} and the HSWME \eqref{eq:HSWME}. The first approach is to first show that the SWME \eqref{eq:SWME} is indeed rotational invariant by decomposing the convection term into conservative part and non-conservative part. The second approach is to exploit the systems block structure. Motivated by the second approach, we propose the general closure relation such that the rotational invariance is satisfied.

\subsection{Rotational invariance of the SWME and the HSWME}\label{subsec:rotational-invariance-SWME-HSWME}

In this part, we show that the SWME \eqref{eq:SWME} is invariant under the rotation of the coordinate system. We first introduce the definition of rotational invariance, which guarantees that the form of the system remains unchanged under a new rotated coordinate system.
\begin{defn}[rotational invariance]\label{def:rotation-invariance}
Consider the first-order system 
\begin{equation}\label{eq:general-first-order-system}
    \partial_t U + A(U) \partial_x U + B(U) \partial_y U = S(U, X)
\end{equation}
with $U=(h, \, h {u_m}, \, h {v_m}, \, h\alpha_1, \, h\beta_1, \, \cdots, \, h\alpha_N, \, h\beta_N)^T\in\mathbb{R}^{2N+3}$ and $X=(x,y)^T\in\mathbb{R}^{2}$. It is said to satisfy the rotational invariance property if 
\begin{equation}\label{eq:rotation-invariance}
\cos\theta \, A(U) + \sin\theta \, B(U) = T^{-1} A(T U) T,
\end{equation}
and
\begin{equation}\label{eq:rotation-invariance-source}
    T S(U, X) = S(T U, T_2 X)
\end{equation}
for any angle $0\le \theta<2\pi$ and any vector $U$.
Here $T=T(\theta)$ is the rotation matrix given by
\begin{equation}\label{eq:rotation-matrix-T}
    T(\theta) = \diag(1, T_2(\theta), T_2(\theta), \cdots, T_2(\theta))\in\mathbb{R}^{(2N+3)\times(2N+3)}
\end{equation}
with
\begin{equation}\label{eq:rotation-matrix-size-2-2}
T_2(\theta) = 
\begin{pmatrix}
\cos\theta & \sin\theta \\
-\sin\theta & \cos\theta 
\end{pmatrix} 
\in\mathbb{R}^{2\times 2}.
\end{equation}
\end{defn}

\begin{rem}
    We remark that the first diagonal block in $T(\theta)$ is 1, while the remaining diagonal blocks are the rotation matrix $T_2(\theta)$ of size $2\times2$. This is due to the fact that the first component of $U$, i.e., the water height, is a scalar function and the remaining components, i.e., $(hu_m, \, hv_m)$ and $(h\alpha_i, \, h\beta_i)$ for $1\le i\le N$, are vector-valued functions.
\end{rem}

\begin{prop}\label{prop:rotational-invariance-system-unchanged}    
If the first-order system \eqref{eq:general-first-order-system} satisfies the rotational invariance in Definition \ref{def:rotation-invariance}, then the form of the system remains unchanged under a new rotated Cartesian coordinate system.
\end{prop}
\begin{proof}
    See the proof in Appendix \ref{app:proof-rotation-invariance-system-unchanged}.
\end{proof}

We note that it is easy to show the rotational invariance \eqref{eq:rotation-invariance-source} of the source term in \eqref{eq:SWME}. Therefore, we will focus on the rotational invariance \eqref{eq:rotation-invariance} of the convection terms in this paper.
Since the coefficient matrices in the SWME \eqref{eq:SWME} can be split into the conservative part and non-conservative part in \eqref{eq:SWME-matrix-split-conser-nonconser}, we will prove the rotational invariance property in two  steps. Before the proof, we prepare a set of equalities used in the proof.
\begin{prop}\label{prop:identity}
For $u, v, \alpha, \beta, \theta \in\mathbb{R}$, we introduce the rotated variables in the new coordinate system:
\begin{equation}\label{eq:rotated-variables}
    u_\theta := \ct \, u + \st \, v, \quad v_\theta := -\st \, u + \ct \, v, 
\end{equation}
and
\begin{equation}
    \alpha_{\theta} := \ct \, \alpha + \st \, \beta, \quad \beta_{\theta} := -\st \, \alpha + \ct \, \beta.
\end{equation}
Then the following equalities hold true:
\begin{equation}\label{prop:equality-5}
\ct \, u_\theta - \st \, v_\theta = u,    
\end{equation}
\begin{equation}\label{prop:equality-6}
\st \, u_\theta + \ct \, v_\theta = v,
\end{equation}
\begin{equation}\label{prop:equality-1}
\cos\theta \, (u_\theta)^2 - \sin\theta \, u_\theta v_\theta = u u_\theta,
\end{equation}
\begin{equation}\label{prop:equality-2}
\sin\theta \, (u_\theta)^2 + \cos\theta \, u_\theta v_\theta = v u_\theta,
\end{equation}
\begin{equation}\label{prop:equality-3}
2\cos\theta \, u_\theta \alpha_\theta -\sin\theta \, \brac{u_\theta \beta_\theta + v_\theta \alpha_\theta} 
= 2\cos\theta \, u\alpha + \sin\theta \, (u\beta + v\alpha),
\end{equation}
\begin{equation}\label{prop:equality-4}
2\sin\theta \, u_\theta \alpha_\theta + \cos\theta \, \brac{u_\theta \beta_\theta + v_\theta \alpha_\theta} =
\cos\theta \, (u\beta+v\alpha) + 2\sin\theta \, v\beta.
\end{equation}

\end{prop}
\begin{proof}
See the proof in Appendix \ref{appendix:proof-equality}.
\end{proof}

We first prove the rotational invariance for the conservative part:
\begin{lem}[rotational invariance for conservative part]\label{lem:rotation-invariance-conservative}
The conservative part in \eqref{eq:SWME} satisfies the rotational invariance:
\begin{equation}\label{eq:rotation-invariance-conservative}
\cos\theta \, F(U) + \sin\theta \, G(U) = T^{-1} F(TU).
\end{equation}
for any $\theta$ and $U$.
\end{lem}
\begin{proof}
We first compute $TU$:
\begin{equation*}
\begin{aligned}
TU ={}&
\begin{pmatrix}
1 & & & & & \\
& \cos\theta & \sin\theta & & & \\
& -\sin\theta & \cos\theta & & & \\
& & & \cos\theta & \sin\theta \\
& & & -\sin\theta & \cos\theta \\
& & & & & \ddots \\
& & & & & & \cos\theta & \sin\theta \\
& & & & & & -\sin\theta & \cos\theta
\end{pmatrix}
\begin{pmatrix}
h \\
h {u_m} \\
h {v_m} \\
h\alpha_1 \\
h\beta_1 \\
\vdots \\
h\alpha_N \\
h\beta_N 
\end{pmatrix}
\\ 
={}&
\begin{pmatrix}
h \\
h(\cos\theta  {u_m} + \sin\theta  {v_m}) \\
h(-\sin\theta  {u_m} + \cos\theta  {v_m}) \\
h(\cos\theta \alpha_1 + \sin\theta \beta_1) \\
h(-\sin\theta \alpha_1 + \cos\theta \beta_1) \\
\vdots \\
h(\cos\theta \alpha_N + \sin\theta \beta_N) \\
h(-\sin\theta \alpha_N + \cos\theta \beta_N)
\end{pmatrix}
=
\begin{pmatrix}
h \\
h ( {u_m})_\theta \\
h ( {v_m})_\theta \\
h (\alpha_1)_\theta \\
h (\beta_1)_\theta \\
\vdots \\
h (\alpha_N)_\theta \\
h (\beta_N)_\theta
\end{pmatrix},
\end{aligned}
\end{equation*}
 {where $(u_m)_{\theta}$ and $(v_m)_{\theta}$ are defined in \eqref{eq:rotated-variables}, and $(\alpha_i)_{\theta}$ and $(\beta_i)_{\theta}$ are defined similarly.}

Next, we compute $F(TU)$:
\begin{equation*}
\begin{aligned}
F(TU)={}&
(h ( {u_m})_\theta, \, h(( {u_m})_\theta^{2} + \sum_{j}\frac{(\alpha_1)_\theta^2}{2j+1})+\frac{1}{2} g h^{2},  \,  
h(( {u_m})_\theta ( {v_m})_\theta + \sum_{j}\frac{(\alpha_j)_\theta (\beta_j)_\theta}{2j+1}), \\
& h(2( {u_m})_\theta(\alpha_1)_\theta + \sum_{j,k}A_{1jk}(\alpha_j)_\theta (\alpha_k)_\theta), \, 
h(( {u_m})_\theta(\beta_1)_\theta + ( {v_m})_\theta (\alpha_1)_\theta + \sum_{j,k}A_{1jk}(\alpha_j)_\theta (\beta_k)_\theta), \\
& \cdots, \\
& h(2( {u_m})_\theta(\alpha_n)_\theta + \sum_{j,k} {A_{Njk}}(\alpha_j)_\theta(\alpha_k)_\theta),  \, 
h(( {u_m})_\theta(\beta_n)_\theta + ( {v_m})_\theta(\alpha_n)_\theta + \sum_{j,k} {A_{Njk}}(\alpha_j)_\theta(\beta_k)_\theta))^T,
\end{aligned}
\end{equation*}

Then we compute $T^{-1} F(TU)$ for each component and prove that it is equal to the LHS in \eqref{eq:rotation-invariance-conservative}. Notice that
\begin{equation}
T^{-1} = 
\begin{pmatrix}
1 & & & & & \\
& \cos\theta & -\sin\theta & & & \\
& \sin\theta & \cos\theta & & & \\
& & & \cos\theta & -\sin\theta \\
& & & \sin\theta & \cos\theta \\
& & & & & \ddots \\
& & & & & & \cos\theta & -\sin\theta \\
& & & & & & \sin\theta & \cos\theta
\end{pmatrix}.    
\end{equation}

We start with the first component in \eqref{eq:rotation-invariance-conservative}:
\begin{equation*}
\textrm{RHS} = h( {u_m})_\theta = h(\ct \,  {u_m} + \st  \,  {v_m}) = \ct  \, h  {u_m} + \st  \, h  {v_m} = \textrm{LHS}.
\end{equation*}
Next, we compute the second component in \eqref{eq:rotation-invariance-conservative}:
\begin{equation*}
\begin{aligned}
\textrm{RHS} 
&= \cos\theta \brac{ h(( {u_m})_\theta^{2} + \sum_{j}\frac{(\alpha_1)_\theta^2}{2j+1})+\frac{1}{2} g h^{2} } - \sin\theta 
h(( {u_m})_\theta ( {v_m})_\theta + \sum_{j}\frac{(\alpha_j)_\theta (\beta_j)_\theta}{2j+1})
 \\
&= h \brac{\ct ( {u_m})_\theta^2 - \st ( {u_m})_\theta ( {v_m})_\theta} + \ct\half gh^2 + h \sum_j \frac{1}{2j+1} \brac{\ct (\alpha_j)_\theta^2 - \st (\alpha_j)_\theta (\beta_j)_\theta } \\
& \stackrel{\text{\eqref{prop:equality-1}}}{=} h  {u_m} ( {u_m})_\theta + \ct\half gh^2 + h \sum_j \frac{1}{2j+1} \alpha_j (\alpha_j)_\theta  \\
&= h {u_m}(\cos\theta  {u_m} + \sin\theta  {v_m}) + \cos\theta\half gh^2 + h\sum_j\frac{1}{2j+1}\alpha_j(\cos\theta\alpha_j+\sin\theta\beta_j) \\
&= \cos\theta\brac{h( {u_m}^2+\sum_j\frac{\alpha_j^2}{2j+1})+\half gh^2} + \sin\theta\brac{h( {u_m} {v_m}+\sum_j\frac{\alpha_j\beta_j}{2j+1})} \\
&= \textrm{LHS}.
\end{aligned}
\end{equation*}
The third component in the RHS of \eqref{eq:rotation-invariance-conservative} is:
\begin{equation*}
\begin{aligned}
\textrm{RHS} 
&= \sin\theta \brac{ h(( {u_m})_\theta^{2} + \sum_{j}\frac{(\alpha_j)_\theta^2}{2j+1})+\frac{1}{2} g h^{2} } + \cos\theta h(( {u_m})_\theta ( {v_m})_\theta + \sum_{j}\frac{(\alpha_j)_\theta (\beta_j)_\theta}{2j+1}) \\
&= h \brac{\st ( {u_m})_\theta^2 + \ct ( {u_m})_\theta ( {v_m})_\theta} + \st\half gh^2 + h \sum_j \frac{1}{2j+1} \brac{\st (\alpha_j)_\theta^2 + \ct (\alpha_j)_\theta (\beta_j)_\theta } \\
& \stackrel{\text{\eqref{prop:equality-2}}}{=} h ( {v_m}) ( {u_m})_\theta + \st\half gh^2 + h \sum_j \frac{1}{2j+1} \beta_j (\alpha_j)_\theta  \\
&= h {v_m}(\cos\theta  {u_m} + \sin\theta ( {v_m})) + h\sum_j\frac{1}{2j+1}\beta_j(\cos\theta\alpha_j+\sin\theta\beta_j) + \sin\theta\half gh^2\\
&= \cos\theta\brac{h( {u_m} {v_m}+\sum_j\frac{\alpha_j\beta_j}{2j+1})} + \sin\theta\brac{h(( {v_m})^2+\sum_j\frac{\beta_j^2}{2j+1})+\half gh^2} \\
&= \textrm{LHS}.
\end{aligned}
\end{equation*}
The fourth component in the RHS of \eqref{eq:rotation-invariance-conservative} is
\begin{equation*}
\begin{aligned}
\textrm{RHS} 
&= \ct h \brac{2 ( {u_m})_\theta(\alpha_1)_\theta + \sum_{j,k}A_{1jk}(\alpha_j)_\theta (\alpha_k)_\theta} - \st h \brac{( {u_m})_\theta(\beta_1)_\theta + ( {v_m})_\theta (\alpha_1)_\theta + \sum_{j,k}A_{1jk}(\alpha_j)_\theta (\beta_k)_\theta} \\
&= h \brac{2\ct ( {u_m})_{\theta}(\alpha_1)_\theta - \st(( {u_m})_{\theta}(\beta_1)_\theta + ( {v_m})_\theta(\alpha_1)_{\theta})} + h \sum_{j,k}A_{1jk} (\alpha_j)_{\theta} \brac{\ct (\alpha_k)_{\theta} - \st (\beta_k)_{\theta}} \\
& \stackrel{\text{\eqref{prop:equality-3},\eqref{prop:equality-5}}}{=} h\brac{2\ct  {u_m}\alpha_1 + \st( {u_m}\beta_1 +  {v_m}\alpha_1)} + h \sum_{j,k}A_{1jk} (\alpha_j)_{\theta} \alpha_k \\
&= h\brac{2\ct  {u_m}\alpha_1 + \st( {u_m}\beta_1 +  {v_m}\alpha_1)} + h \sum_{j,k}A_{1jk} (\ct\alpha_j + \st\beta_j)\alpha_k \\
&= \cos\theta h(2 {u_m}\alpha_1 + \sum_{j,k}A_{1jk}\alpha_j\alpha_k) + \sin\theta h (( {u_m}\beta_1+ {v_m}\alpha_1) + \sum_{j,k}A_{1jk}\alpha_j\beta_k) \\
&=\textrm{LHS}.
\end{aligned}    
\end{equation*}
Then we compute the fifth component:
\begin{equation*}
\begin{aligned}
\textrm{RHS} 
&= \st h \brac{2 ( {u_m})_\theta(\alpha_1)_\theta + \sum_{j,k}A_{1jk}(\alpha_j)_\theta (\alpha_k)_\theta} + \ct h \brac{( {u_m})_\theta(\beta_1)_\theta + ( {v_m})_\theta (\alpha_1)_\theta + \sum_{j,k}A_{1jk}(\alpha_j)_\theta (\beta_k)_\theta} \\
&= h \brac{2\st ( {u_m})_{\theta}(\alpha_1)_\theta + \ct(( {u_m})_{\theta}(\beta_1)_\theta + ( {v_m})_\theta(\alpha_1)_{\theta})} + h \sum_{j,k}A_{1jk} (\alpha_j)_{\theta} \brac{\st (\alpha_k)_{\theta} + \ct (\beta_k)_{\theta}} \\
& \stackrel{\text{\eqref{prop:equality-4},\eqref{prop:equality-6}}}{=} h\brac{\ct( {u_m}\beta_1 +  {v_m}\alpha_1) + 2\st  {v_m}\beta_1 } + h \sum_{j,k}A_{1jk} (\alpha_j)_{\theta} \beta_k \\
&= h(\cos\theta( {u_m}\beta_1+ {v_m}\alpha_1) + 2\st  {v_m}\beta_1) + h\sum_{j,k}A_{1jk}(\cos\theta \alpha_j+\sin\theta \beta_j)\beta_k \\
&= \cos\theta h(( {u_m}\beta_1+ {v_m}\alpha_1) + \sum_{j,k}A_{1jk}\alpha_j\beta_k) + \sin\theta h (2 {v_m}\beta_1 + \sum_{j,k}A_{1jk}\beta_j\beta_k) \\
&= \textrm{LHS}.
\end{aligned} 
\end{equation*}

For the remaining components, the proof is similar to the fourth and fifth ones and the details are omitted here.
\end{proof}

Next, we prove the rotational invariance for the non-conservative part:
\begin{lem}[rotational invariance for non-conservative part]\label{lem:rotation-invariance-non-conservative}
The non-conservative part in \eqref{eq:SWME} satisfies the rotational invariance:
\begin{equation}\label{eq:rotation-invariance-non-conservative}
\cos\theta \, P(U) + \sin\theta \, Q(U) = T^{-1} P(T U) T,
\end{equation}
for any $\theta$ and $U$.
\end{lem}
\begin{proof}
The matrix for the non-conservative part in \eqref{eq:matrix-non-conservative} consists of two parts.
We start the proof with the first part. Notice that $T(\theta)$ has the same block diagonal structure as $P_1(U)$ in \eqref{eq:matrix-non-conservative-part1-x} and $Q_1(U)$ in \eqref{eq:matrix-non-conservative-part1-y}. Therefore, it suffices to check that
\begin{equation*}
\cos\theta \, p(U) + \sin\theta \, q(U) = T_2^{-1} p(TU) T_2,
\end{equation*}
where the matrices $p(U)$ and $q(U)$ are defined in \eqref{eq:matrix-small-p-q} and $T_2$ defined in \eqref{eq:rotation-matrix-size-2-2}. This equality can be easily verified by direct calculations.

The proof of the second part $P_2(U)$ in \eqref{eq:matrix-non-conservative-part2-x} and $Q_2(U)$ in \eqref{eq:matrix-non-conservative-part2-y} follows similarly by using the multiplication of the block matrices and the rotational invariance property of each sub-block matrices $g_{ij}(U)$ and $h_{ij}(U)$ defined in \eqref{eq:matrix-non-conservative-sub-block}.
\end{proof}

Comibing the above two lemmas, we have the following theorem:
\begin{thm}[rotational invariance of SWME]
The SWME \eqref{eq:SWME} satisfies the rotational invariance:
\begin{equation}
    \cos\theta \, A(U) + \sin\theta \, B(U) = T^{-1}  A(TU) T.
\end{equation}    
\end{thm}
\begin{proof}
Taking the derivative with respect to $U$ on both sides of \eqref{eq:rotation-invariance-conservative} in Lemma \ref{lem:rotation-invariance-conservative}, we have
\begin{equation*}
\cos\theta  \, \partial_U F(U) + \sin\theta  \, \partial_U G(U) = T^{-1} \partial_U F(TU) T.
\end{equation*}
Combining this with \eqref{eq:rotation-invariance-non-conservative} in Lemma \ref{lem:rotation-invariance-non-conservative}, one immediately obtains
\begin{equation}
\cos\theta  \, A(U) + \sin \, \theta B(U) = T^{-1} A(TU) T.
\end{equation}
\end{proof}

Since the HSWME \eqref{eq:HSWME} is obtained by evaluating the coefficient matrices in the SWME \eqref{eq:SWME} at $\alpha_i=\beta_i=0$ for $2\le i\le N$, its rotational invariance follows immediately:
\begin{thm}[rotational invariance of HSWME]\label{thm:rotation-invariance-HSWME}
The HSWME \eqref{eq:HSWME} satisfies the rotational invariance:
\begin{equation}\label{eq:rotation-invariance-HSWME}
\cos\theta \, A_H(U) + \sin\theta  \, B_H(U) = T^{-1} A_H(TU) T.
\end{equation}
\end{thm}

\subsection{Alternative proof of rotational invariance of the SWME and the HSWME}\label{subsec:alternative-proof-rotation-invariance}

In the previous part, we prove the rotational invariance of the HSWME \eqref{eq:HSWME} by first proving the rotational invariance of the SWME \eqref{eq:SWME}. In this part, we will show an alternative proof of the rotational invariance of the HSWME \eqref{eq:HSWME} with the aid of its block structure.

{We first show the explicit form of the coefficient matrices of the HSWME \eqref{eq:HSWME}. The proof that establishes structure of the coefficient matrices is given in Theorem 4.3.1 and Theorem 4.3.2 in \cite{verbiest2022}, with another ordering of variables. We omit the proof and only present the results.}
\begin{lem}[coefficient matrices of HSWME]\label{lem:coefficient-matrices-HSWME}
The coefficient matrices of the HSWME \eqref{eq:HSWME} are given by:
\begin{equation}\label{eq:coefficient-matrix-HSMWE-x}
    A_H = 
    \begin{aligned}
        \begin{pmatrix}
        0 & 1 & 0 \\
        - {u_m}^2 - \frac{\alpha_1^2}{3} + gh & 2 {u_m} & 0 & \frac{2\alpha_1}{3} & 0 \\
        - {u_m} {v_m} - \frac{\alpha_1\beta_1}{3} &  {v_m} &  {u_m} & \frac{\beta_1}{3} & \frac{\alpha_1}{3} \\
        -2 {u_m}\alpha_1 & 2\alpha_1 & 0 &  {u_m} & 0 & \frac{3}{5}\alpha_1 & 0 & \cdots & 0 & 0 \\
        -( {u_m}\beta_1+ {v_m}\alpha_1) & \beta_1 & \alpha_1 & 0 &  {u_m} & \frac{1}{5}\beta_1 & \frac{2}{5}\alpha_1 & \cdots & 0 & 0 \\
        - \frac{2}{3}\alpha_1^2 & 0 & 0 & \frac{1}{3}\alpha_1 & 0 &  {u_m} & 0 &  \cdots & 0 & 0 \\
        - \frac{2}{3}\alpha_1\beta_1 & 0 & 0 & -\frac{1}{3}\beta_1 & \frac{2}{3}\alpha_1 & 0 &  {u_m} & \cdots & 0 & 0 \\
         &  &  &  &  & \ddots & \ddots & \ddots & \frac{N+1}{2N+1}\alpha_1 & 0 \\
         &  &  &  &  & \ddots & \ddots & \ddots & \frac{1}{2N+1}\beta_1 & \frac{N}{2N+1}\alpha_1 \\
         &  &  &  &  &  & \frac{N-1}{2N-1}\alpha_1 & 0 &  {u_m} & 0 \\
         &  &  &  &  &  & -\frac{1}{2N-1}\beta_1 & \frac{N}{2N-1}\alpha_1 & 0 &  {u_m}
        \end{pmatrix}
    \end{aligned}   
\end{equation}
in the $x$ direction and
\begin{equation}\label{eq:coefficient-matrix-HSMWE-y}
    B_H = 
    \begin{aligned}
        \begin{pmatrix}
        0 & 0 & 1 \\
        - {u_m} {v_m} - \frac{\alpha_1\beta_1}{3} &  {v_m} &  {u_m} & \frac{\beta_1}{3} & \frac{\alpha_1}{3} \\
        - {v_m}^2 - \frac{\beta_1^2}{3} + gh & 0 & 2 {v_m} & 0 & \frac{2\beta_1}{3} \\
        -( {u_m}\beta_1+ {v_m}\alpha_1) & \beta_1 & \alpha_1 &  {v_m} & 0 & \frac{2}{5}\beta_1 & \frac{1}{5}\alpha_1 \\
        -2 {v_m}\beta_1 & 0 & 2\beta_1 & 0 &  {v_m} & 0 & \frac{3}{5}\beta_1 \\
        - \frac{2}{3}\alpha_1\beta_1 & 0 & 0 & \frac{2}{3}\beta_1 & -\frac{1}{3}\alpha_1 &  {v_m} & 0 & \cdots & 0 & 0 \\
        - \frac{2}{3}\beta_1^2 & 0 & 0 & 0 & \frac{1}{3}\beta_1 & 0 &  {v_m} &  \cdots & 0 & 0 \\
         &  &  &  &  & \ddots & \ddots & \ddots & \frac{N}{2N+1}\beta_1 & \frac{1}{2N+1}\alpha_1 \\
         &  &  &  &  & \ddots & \ddots & \ddots & 0 & \frac{N+1}{2N+1}\beta_1 \\
         &  &  &  &  &  & \frac{N}{2N-1}\beta_1 & -\frac{1}{2N-1}\alpha_1 &  {v_m} & 0 \\
         &  &  &  &  & & 0 & \frac{N-1}{2N-1}\beta_1 & 0 &  {v_m}
        \end{pmatrix}
    \end{aligned}   
\end{equation}
in the $y$ direction.
\end{lem}

\begin{proof}
    See the proof in Theorem 4.3.1 and Theorem 4.3.2 in \cite{verbiest2022} with another ordering of variables.
\end{proof}

Now we present the alternative proof of Theorem \ref{thm:rotation-invariance-HSWME} by using the coefficient matrices of the HSWME \eqref{eq:HSWME} given in Lemma \ref{lem:coefficient-matrices-HSWME}.
\begin{proof}[\protect{Alternative proof of Theorem \ref{thm:rotation-invariance-HSWME}}]
    The coefficient matrix $A_H$ in \eqref{eq:coefficient-matrix-HSMWE-x} can be written as in the block form:
    \begin{equation*}
        A_H =
    \begin{pmatrix}
        0 & d_1 \\
        d_2 & A_{11} & A_{12} \\
        d_3 & A_{21} & A_{22} & A_{23} \\
        d_4 &  & \ddots & \ddots & \ddots \\
        & &  & A_{N,N-1} & A_{N,N} & A_{N,N+1} \\
        &  &  &  & A_{N+1,N} & A_{N+1,N+1} \\
    \end{pmatrix}    
    \end{equation*}
    with 
    \begin{equation*}
        d_1 = (1, 0), \quad d_2 =
        \begin{pmatrix}
            - {u_m}^2 - \frac{\alpha_1^2}{3} + gh \\
            - {u_m} {v_m} - \frac{\alpha_1\beta_1}{3}
        \end{pmatrix}
        ,\quad d_3 =
        \begin{pmatrix}
            -2 {u_m}\alpha_1 \\
            -( {u_m}\beta_1+ {v_m}\alpha_1)
        \end{pmatrix}
        ,\quad d_4 =
        \begin{pmatrix}
            - \frac{2}{3}\alpha_1^2 \\
            - \frac{2}{3}\alpha_1\beta_1
        \end{pmatrix},
    \end{equation*}
    \begin{equation*}
        A_{ii} = 
        \begin{pmatrix}
            2 {u_m} & 0 \\
             {v_m} &  {u_m}
        \end{pmatrix}
        ,\quad
        A_{i,i+1} =
        \begin{pmatrix}
            \frac{2i}{2i+1}\alpha_1 & 0 \\
            \frac{i}{2i+1}\beta_1 & \frac{i}{2i+1}\alpha_1
        \end{pmatrix}
        ,\quad
        A_{i+1,i} =
        \begin{pmatrix}
            \frac{2i}{2i-1}\alpha_1 & 0 \\
            \frac{i}{2i-1}\beta_1 & \frac{i}{2i-1}\alpha_1
        \end{pmatrix}.
    \end{equation*}
    The coefficient matrix $B_H$ in \eqref{eq:coefficient-matrix-HSMWE-y} can be written as in the block form similarly:
    \begin{equation*}
        B_H =
    \begin{pmatrix}
        0 & f_1 \\
        f_2 & B_{11} & B_{12} \\
        f_3 & B_{21} & B_{22} & B_{23} \\
        f_4 &  & \ddots & \ddots & \ddots \\
        & &  & B_{N,N-1} & B_{N,N} & B_{N,N+1} \\
        &  &  &  & B_{N+1,N} & B_{N+1,N+1}
    \end{pmatrix}    
    \end{equation*}
    with
    \begin{equation*}
        f_1 = (0, 1), \quad f_2 =
        \begin{pmatrix}    
            - {u_m} {v_m} - \frac{\alpha_1\beta_1}{3} \\
            - {v_m}^2 - \frac{\beta_1^2}{3} + gh
        \end{pmatrix}
        ,\quad f_3 =
        \begin{pmatrix}
            -( {u_m}\beta_1+ {v_m}\alpha_1) \\
            -2 {v_m}\beta_1
        \end{pmatrix}
        ,\quad f_4 =
        \begin{pmatrix}
            - \frac{2}{3}\alpha_1\beta_1 \\
            - \frac{2}{3}\beta_1^2
        \end{pmatrix},
    \end{equation*}
    \begin{equation*}
        B_{ii} = 
        \begin{pmatrix}
             {v_m} &  {u_m}  \\
            0 & 2 {v_m} 
        \end{pmatrix}
        ,\quad
        B_{i,i+1} =
        \begin{pmatrix}
            \frac{i}{2i+1}\beta_1 & \frac{i}{2i+1}\alpha_1 \\
            0 & \frac{2i}{2i+1}\beta_1 \\
        \end{pmatrix}
        ,\quad
        B_{i+1,i} =
        \begin{pmatrix}
            \frac{i}{2i-1}\beta_1 & \frac{i}{2i-1}\alpha_1 \\
            0 & \frac{2i}{2i-1}\beta_1
        \end{pmatrix}.
    \end{equation*}
    
    Next, we compute $T^{-1}A_H(TU)T$ and verify that it is equal to $\cos\theta A_H(U) + \sin \theta B_H(U)$.
    \begin{equation*}
    \begin{aligned}        
        T^{-1} A_H(TU) T
        ={}&
        \diag(1, T_2^{-1}, \cdots, T_{2}^{-1})
        \begin{pmatrix}
            0 & d_1 \\
            d_2 & A_{11} & A_{12} \\
            d_3 & A_{21} & A_{22} & A_{23} \\
            d_4 &  & \ddots & \ddots & \ddots \\
            & &  & A_{N,N-1} & A_{N,N} & A_{N,N+1} \\
            &  &  &  & A_{N+1,N} & A_{N+1,N+1}    
        \end{pmatrix}
        T \\
        ={}& 
        \begin{pmatrix}
            0 & d_1 \\
            T_2^{-1} d_2 & T_2^{-1} A_{11} & T_2^{-1} A_{12} \\
            T_2^{-1} d_3 & T_2^{-1} A_{21} & T_2^{-1} A_{22} & T_2^{-1} A_{23} \\
            T_2^{-1} d_4 &  & \ddots & \ddots & \ddots \\
            & &  & T_2^{-1} A_{N,N-1} & T_2^{-1} A_{N,N} & T_2^{-1} A_{N,N+1} \\
            &  &  &  & T_2^{-1} A_{N+1,N} & T_2^{-1} A_{N+1,N+1}    
        \end{pmatrix}
        T
        \\
        ={}& 
        \begin{pmatrix}
            0 & d_1 T_2 \\
            T_2^{-1} d_2 & T_2^{-1} A_{11} T_2 & T_2^{-1} A_{12} T_2 \\
            T_2^{-1} d_3 & T_2^{-1} A_{21} T_2 & T_2^{-1} A_{22} T_2 & T_2^{-1} A_{23} T_2 \\
            T_2^{-1} d_4 &  & \ddots & \ddots & \ddots \\
            & &  & T_2^{-1} A_{N,N-1} T_2 & T_2^{-1} A_{N,N} T_2 & T_2^{-1} A_{N,N+1} T_2 \\
            &  &  &  & T_2^{-1} A_{N+1,N} T_2 & T_2^{-1} A_{N+1,N+1} T_2
        \end{pmatrix}        
    \end{aligned}
    \end{equation*}
    where the independent variable $TU$ in $d_i$ for $1\le i\le 4$ and $A_{ij}$ for $1\le i,j\le N+1$ is omitted for simplicity.

    Now we compute each block of the matrix $T^{-1}A_H(TU)T$:
    \begin{equation*}
        d_1 T_2 =
        \begin{pmatrix}
            1 & 0
        \end{pmatrix}
        \begin{pmatrix}
            \cos\theta & \sin\theta \\
            -\sin\theta & \cos\theta
        \end{pmatrix}
        =
        \begin{pmatrix}
            \cos\theta & \sin\theta
        \end{pmatrix}
        = \cos\theta \, d_1 + \sin\theta \, f_1.
    \end{equation*}
    \begin{equation*}
        T_2^{-1} d_2(TU) = \cos\theta \, d_2(U) + \sin\theta \, f_2(U).
    \end{equation*}
    \begin{equation*}
        T_2^{-1} d_3(TU) = \cos\theta \, d_3(U) + \sin\theta \, f_3(U).
    \end{equation*}
    \begin{equation*}
        T_2^{-1} d_4(TU) = \cos\theta \, d_4(U) + \sin\theta \, f_4(U).
    \end{equation*}        
    \begin{equation*}
    \begin{aligned}                
        T_2^{-1} A_{ii} T_2 &= 
        \begin{pmatrix}
            \cos\theta & -\sin\theta \\
            \sin\theta & \cos\theta
        \end{pmatrix}
        \begin{pmatrix}
            2 {u_m} & 0 \\
             {v_m} &  {u_m}
        \end{pmatrix}        
        \begin{pmatrix}
            \cos\theta & \sin\theta \\
            -\sin\theta & \cos\theta
        \end{pmatrix}   
        \\
        &= \cos\theta
        \begin{pmatrix}
            2 {u_m} & 0 \\
             {v_m} &  {u_m}
        \end{pmatrix}                
        +\sin\theta
        \begin{pmatrix}
             {v_m} &  {u_m} \\
            0 & 2 {v_m}
        \end{pmatrix}
        \\
        &= \cos\theta A_{ii} + \sin\theta B_{ii}.
    \end{aligned}             
    \end{equation*}
    \begin{equation*}
        T_2^{-1} A_{i,i+1} T_2 = \cos\theta A_{i,i+1} + \sin\theta B_{i,i+1}.
    \end{equation*}
    \begin{equation*}
        T_2^{-1} A_{i+1,i} T_2 = \cos\theta A_{i+1,i} + \sin\theta B_{i+1,i}.
    \end{equation*}
    Therefore, we have proved \eqref{eq:rotation-invariance-HSWME} in Theorem \ref{thm:rotation-invariance-HSWME}.
\end{proof}

\begin{rem}
    The rotational invariance of the SWME \eqref{eq:SWME} can also be proved in the similar line as the above proof by using the block structure of the coefficient matrices  explicitly. We omit it here for  space considerations.
\end{rem}

\subsection{General closure relation with the rotational invaiance}

From the above alternative proof, we observe that the rotational invariance of the HSWME relies on the rotational invariance of each sub-block (of size $2\times2$) of the coefficient matrices.
Motivated by this observation, we would like to analyze the rotational invariance of matrices of size $2\times2$ and find out some general relations which satisfy the rotational 
invariance.  We note that this will be key in deriving our new model, which has a more general closure relation than that of HSWME. 

\begin{defn}[rotational invariance of $2\times2$ matrices]\label{def:rotational-invariance-block-matrix}
    Consider two matrices $A(V)$ and $B(V)$ of size $2\times2$ given by
    \begin{equation*}
        A(V) = 
        \begin{pmatrix}
            a_{11}(V) & a_{12}(V) \\
            a_{21}(V) & a_{22}(V)
        \end{pmatrix}
        \in\mathbb{R}^{2\times2},
    \end{equation*}
    and
    \begin{equation*}
        B(V) = 
        \begin{pmatrix}
            b_{11}(V) & b_{12}(V) \\
            b_{21}(V) & b_{22}(V)
        \end{pmatrix}
        \in\mathbb{R}^{2\times2},
    \end{equation*}
    with $V=(p, \, q)^T \in\mathbb{R}^2$.
    We say that $A(V)$ and $B(V)$ satisfy the rotational invariance if    \begin{equation}\label{eq:rotational-invariance-block-matrix}
        T_2^{-1} A(T_2 V) T_2 = \cos\theta \, A(V) + \sin\theta \, B(V),
    \end{equation}
    holds true for any $0\le \theta < 2\pi$ and any $V\in\mathbb{R}^2$ with $T_2$ being the rotational matrix
    \begin{equation*}
        T_2 = 
        \begin{pmatrix}
            \cos\theta & \sin\theta \\
            -\sin\theta & \cos\theta
        \end{pmatrix}.
    \end{equation*}
\end{defn}
Note that, in the definition, the dummy variables $(p,q)$ could be $( {u_m}, {v_m})$ or $(\alpha_i,\beta_i)$ for $1\le i\le N$ in the shallow water moment model.
In the following part, we will derive some conditions for the rotational invariance of $2\times2$ matrices to be satisfied. We first present several necessary conditions:
\begin{lem}\label{lem:rotational-invariance-block-matrix-necessary-condition}
    If the matrices $A(V)$ and $B(V)$ given by
    \begin{equation}
        A(V) = 
        \begin{pmatrix}
            a_{11}(V) & a_{12}(V) \\
            a_{21}(V) & a_{22}(V)
        \end{pmatrix}
        , \quad
        B(V) =
        \begin{pmatrix}
            b_{11}(V) & b_{12}(V) \\
            b_{21}(V) & b_{22}(V)
        \end{pmatrix}
    \end{equation}
    satisfy the rotational invariance of $2\times 2$ matrices, then the following relations hold:
    \begin{enumerate}
        \item All the entries in $B(V)$ are determined by $A(V)$ in the following way:
        \begin{equation}
        \begin{aligned}
            b_{11}(p,q) &= a_{22}(q, -p), \\
            b_{12}(p,q) &= -a_{21}(q, -p), \\
            b_{21}(p,q) &= -a_{12}(q, -p), \\
            b_{22}(p,q) &= a_{11}(q, -p).
        \end{aligned}
        \end{equation}
        
        \item $A(V)$ is an odd function in the sense that:
        \begin{equation}
            A(-V) = -A(V).
        \end{equation}
    \end{enumerate}
\end{lem}
\begin{proof}
\begin{enumerate}
    \item 
By taking $\theta=\frac{\pi}{2}$ in \eqref{eq:rotational-invariance-block-matrix}, we have
\begin{equation*}
    T_2 = 
    \begin{pmatrix}
        0 & 1 \\
        -1 & 0
    \end{pmatrix},
\end{equation*}
and
\begin{equation*}
    T_2 V = 
    \begin{pmatrix}
        q \\
        -p
    \end{pmatrix}.
\end{equation*}
Therefore, we have
\begin{equation*}
    T_2^{-1} A(T_2 V) T_2 =
    \begin{pmatrix}
        0 & -1 \\
        1 & 0
    \end{pmatrix}
    \begin{pmatrix}
        a_{11}(TV) & a_{12}(TV) \\
        a_{21}(TV) & a_{22}(TV)
    \end{pmatrix}
    \begin{pmatrix}
        0 & 1 \\
        -1 & 0
    \end{pmatrix}
    = 
    \begin{pmatrix}
        a_{22}(TV) & -a_{21}(TV) \\
        -a_{12}(TV) & a_{11}(TV)
    \end{pmatrix}.
\end{equation*}
Then the relation \eqref{eq:rotational-invariance-block-matrix} reduces to
\begin{equation*}
    \begin{pmatrix}
        b_{11}(V) & b_{12}(V) \\
        b_{21}(V) & b_{22}(V)
    \end{pmatrix}
    = 
    \begin{pmatrix}
        a_{22}(TV) & -a_{21}(TV) \\
        -a_{12}(TV) & a_{11}(TV)
    \end{pmatrix},
\end{equation*}
which completes the proof.

\item
Taking $\theta$ to be $(\theta+\pi)$ in \eqref{eq:rotational-invariance-block-matrix}, we have
\begin{equation*}
    \cos(\theta+\pi) A(V) + \sin(\theta+\pi) B(V) = (-T_2)^{-1} A(-T_2 V) (-T_2)
\end{equation*}
which implies
\begin{equation*}
    -\cos(\theta) A(V) - \sin(\theta) B(V) = T_2^{-1} A(-T_2 V) T_2.
\end{equation*}
Comparing the above equation with \eqref{eq:rotational-invariance-block-matrix}, we have
\begin{equation*}
    A(T_2 V) = -A(-T_2 V).
\end{equation*}
Since $T_2$ is invertible, we have that $A$ is an odd function.
\end{enumerate}
\end{proof}

Next, we will restrict to the case of linear functions and find out the necessary and sufficient conditions for the matrices to be rotational invariant. The conditions will be presented in Thereom \ref{thm:rotational-invariance-block-matrix}. To prove the theorem, we prepare the following lemma:
\begin{lem}\label{lem:rotational-invariance-block-matrix-linear}
    Assume that $A(V)$ and $B(V)$ satisfy the rotational invariance of $2\times2$ matrices and they are linear functions of $V$.
    \begin{enumerate}
        \item If $A(V)$ only has two non-zero entries in the first column:
        \begin{equation*}
            A(V) = 
            \begin{pmatrix}
                a_{11}(V) & 0 \\
                a_{21}(V) & 0
            \end{pmatrix},
        \end{equation*}
        then $A(V)$ and $B(V)$ have to be of the form:        
        \begin{equation}\label{eq:rotational-invariance-block-matrix-linear-case-1}
            A(V) = 
            c_1
            \begin{pmatrix}
                p & 0\\
                q & 0
            \end{pmatrix}
            +
            c_2
            \begin{pmatrix}
                q & 0 \\
                -p & 0
            \end{pmatrix},
        \end{equation}
        and
        \begin{equation*}
            B(V) = 
            c_1
            \begin{pmatrix}
                0 & p \\
                0 & q
            \end{pmatrix}
            +
            c_2
            \begin{pmatrix}
                0 & q \\
                0 & -p
            \end{pmatrix},
        \end{equation*}
        where $c_1, c_2\in\mathbb{R}$.

        \item If $A(V)$ only has two non-zero entries in the diagonal:
        \begin{equation*}
            A(V) = 
            \begin{pmatrix}
                a_{11}(V) & 0 \\
                0 & a_{22}(V)
            \end{pmatrix},
        \end{equation*}
        then $A(V)$ and $B(V)$ have to be of the form:
        \begin{equation*}
            A(V)
            = c_1
            \begin{pmatrix}
                p & 0\\
                0 & p
            \end{pmatrix}
            + c_2
            \begin{pmatrix}
                q & 0\\
                0 & q
            \end{pmatrix},
        \end{equation*}
        and
        \begin{equation*}
            B(V) =
            c_1
            \begin{pmatrix}
                q & 0\\
                0 & q
            \end{pmatrix}
            -c_2
            \begin{pmatrix}
                p & 0\\
                0 & p
            \end{pmatrix},
        \end{equation*}
        where $c_1, c_2\in\mathbb{R}$.

        \item If $A(V)$ only has two non-zero entries in the second column:
        \begin{equation*}
            A(V) = 
            \begin{pmatrix}
                0 & a_{21}(V) \\
                0 & a_{22}(V)
            \end{pmatrix},
        \end{equation*}
        then $A(V)$ and $B(V)$ have to be of the form:
        \begin{equation*}
            A(V)
            = c_1
            \begin{pmatrix}
                0 & p \\
                0 & q
            \end{pmatrix}
            + c_2
            \begin{pmatrix}
                0 & q \\
                0 & -p
            \end{pmatrix},
        \end{equation*}
        and
        \begin{equation*}
            B(V)
            = c_1
            \begin{pmatrix}
                -p & 0 \\
                -q & 0
            \end{pmatrix}
            + c_2
            \begin{pmatrix}
                -q & 0 \\
                p & 0
            \end{pmatrix},
        \end{equation*}   
        where $c_1, c_2\in\mathbb{R}$.     
    \end{enumerate}    
\end{lem}
\begin{proof}
    See the proof in Appendix \ref{app:proof-lem-rotational-invariance-block-matrix-linear}.
\end{proof}

\begin{thm}\label{thm:rotational-invariance-block-matrix}
    Assume that the matrices $A(V)$ and $B(V)$ satisfy the rotational invariance of $2\times2$ matrices and they are linear functions of $V=(p,q)^T$. Then they must be of the form
    \begin{equation}
        A(V) = c_1
        \begin{pmatrix}
            p & 0 \\
            q & 0
        \end{pmatrix}
        + c_2
        \begin{pmatrix}
            q & 0 \\
            -p & 0
        \end{pmatrix}
        + c_3
        \begin{pmatrix}
            p & 0\\
            0 & p
        \end{pmatrix}
        + c_4
        \begin{pmatrix}
            q & 0\\
            0 & q
        \end{pmatrix}        
        + c_5
        \begin{pmatrix}
            0 & p \\
            0 & q
        \end{pmatrix}
        + c_6
        \begin{pmatrix}
            0 & q \\
            0 & -p
        \end{pmatrix},
    \end{equation}
    and
    \begin{equation}
        B(V) = 
        c_1
        \begin{pmatrix}
            0 & p \\
            0 & q
        \end{pmatrix}
        +
        c_2
        \begin{pmatrix}
            0 & q \\
            0 & -p
        \end{pmatrix}
        +
        c_3
        \begin{pmatrix}
            q & 0\\
            0 & q
        \end{pmatrix}
        -c_4
        \begin{pmatrix}
            p & 0\\
            0 & p
        \end{pmatrix}
        - c_5
        \begin{pmatrix}
            p & 0 \\
            q & 0
        \end{pmatrix}
        + c_6
        \begin{pmatrix}
            -q & 0 \\
            p & 0
        \end{pmatrix},
    \end{equation}
    where $c_i\in\mathbb{R}$ for $1\le i\le 6$.
\end{thm}
\begin{proof}
    Since the matrix $B(V)$ is uniquely determined by $A(V)$ by Lemma  \ref{lem:rotational-invariance-block-matrix-necessary-condition}, we only need to consider the form of $A(V)$.

    For any matrix $A(V)$ of the form
    \begin{equation*}
        A(V) = 
        \begin{pmatrix}
            a_{11}(V) & a_{12}(V) \\
            a_{21}(V) & a_{22}(V)
        \end{pmatrix},
    \end{equation*}
    we can decompose it as
    \begin{equation*}
        A(V) = 
        \begin{pmatrix}
            \tilde{a}_{21}(V) & 0 \\
            a_{21}(V) & 0
        \end{pmatrix}
        +
        \begin{pmatrix}
            a_{11}(V) - \tilde{a}_{21}(V) & a_{12}(V) \\
            0 & a_{22}(V)
        \end{pmatrix},
    \end{equation*}
    where $\tilde{a}_{21}(V)$ is a linear function  such that 
    \begin{equation*}
        P(V) := 
        \begin{pmatrix}
            \tilde{a}_{21}(V) & 0 \\
            a_{21}(V) & 0
        \end{pmatrix}
    \end{equation*}
    satisfies the form of rotational invariance in \eqref{eq:rotational-invariance-block-matrix-linear-case-1} given in Case 1 in Lemma \ref{lem:rotational-invariance-block-matrix-linear}.  {Note that in \eqref{eq:rotational-invariance-block-matrix-linear-case-1} there are only two coefficients $c_1$ and $c_2$ which will be uniquely determined by $a_{21}(V)$.}
    
    Since $A(V)$ and $P(V)$ both satisfy the rotational invariance, we have that the remaining part
    \begin{equation*}
        Q(V) := 
        \begin{pmatrix}
            a_{11}(V) - \tilde{a}_{21}(V) & a_{12}(V) \\
            0 & a_{22}(V)
        \end{pmatrix}
    \end{equation*}
    also satisfies the rotational invariance. Next, we decompose $Q(V)$ into two parts:
    \begin{equation}\label{eq:Q-decompose}
        Q(V) = 
        \begin{pmatrix}
            a_{11}(V) - \tilde{a}_{21}(V) & 0 \\
            0 & \tilde{a}_{22}(V)
        \end{pmatrix}
        +
        \begin{pmatrix}
            0 & a_{12}(V) \\
            0 & {a}_{22}(V) - \tilde{a}_{22}(V)
        \end{pmatrix}
    \end{equation}
    where $\tilde{a}_{22}(V)$ is a linear function uniquely determined by $(a_{11}(V)-\tilde{a}_{21}(V))$ such that 
    \begin{equation}
        R(V) := 
        \begin{pmatrix}
            a_{11}(V) - \tilde{a}_{21}(V) & 0 \\
            0 & \tilde{a}_{22}(V)
        \end{pmatrix}        
    \end{equation}    
    satisfies the rotational invariance given in Case 2 in Lemma \ref{lem:rotational-invariance-block-matrix-linear}.
    
    Lastly, the second part in \eqref{eq:Q-decompose}
    \begin{equation}
        S(V) := 
        \begin{pmatrix}
            0 & a_{12}(V) \\
            0 & {a}_{22}(V) - \tilde{a}_{22}(V)
        \end{pmatrix}
    \end{equation}
    must satisfy the rotational invariance and falls into Case 3 in Lemma \ref{lem:rotational-invariance-block-matrix-linear}. The proof is completed.

\end{proof}

Motivated by the constraint given in Theorem \ref{thm:rotational-invariance-block-matrix}, we can modify the coefficient matrices in the HSWME \eqref{eq:HSWME} in the following way to make it satisfy the rotational invariance property.
\begin{thm}[general closure relation with rotational invariance]\label{thm:general-closure-relation-rotational-invariance}
    Suppose that the matrices $A_{ij}(U)\in\mathbb{R}^{2\times2}$ and $B_{ij}(U)\in\mathbb{R}^{2\times2}$ satisfy the rotational invariance for the $2\times2$ matrices for $1\le i,j\le N+1$.
    Then the matrices $A$ and $B$ given by
    \begin{equation}\label{eq:coefficient-matrix-general-closure-x}
        A = A_H + 
        \begin{aligned}
            \begin{pmatrix}
            0 \\
            & A_{11} & A_{12} & \cdots & A_{1,N+1} \\
            & A_{11} & A_{12} & \cdots & A_{1,N+1} \\
            & \vdots & \vdots & &\vdots \\
            & A_{N+1,1} & A_{N+1,2} & \cdots & A_{N+1,N+1} \\
            \end{pmatrix}
        \end{aligned}   
    \end{equation}
    in the $x$ direction and
    \begin{equation}\label{eq:coefficient-matrix-general-closure-y}
        B = B_H + 
        \begin{aligned}
            \begin{pmatrix}
            0 \\
            & B_{11} & B_{12} & \cdots & B_{1,N+1} \\
            & B_{11} & B_{12} & \cdots & B_{1,N+1} \\
            & \vdots & \vdots & &\vdots \\
            & B_{N+1,1} & B_{N+1,2} & \cdots & B_{N+1,N+1} \\
            \end{pmatrix}
        \end{aligned}
    \end{equation}
    in the $y$ direction, satisfy the rotational invariance. Here $A_H$ and $B_H$ are the coefficient matrices in the HSWME \eqref{eq:HSWME}.
\end{thm}
\begin{proof}
    The proof is similar to the alternative proof of Theorem \ref{thm:rotation-invariance-HSWME} in Section \ref{subsec:alternative-proof-rotation-invariance}.
\end{proof}

\begin{rem}
From Theorem \ref{thm:general-closure-relation-rotational-invariance}, to preserve the rotational invariance of the moment model, we can modify any entries except the first row and the first column, as long as the sub-blocks satisfy the rotational invariance for the $2\times2$ matrices. 
\end{rem}

\section{Analysis of the hyperbolicity}\label{sec:hyperbolic}

In this section, we analyze the hyperbolicity of the moment models. With the aid of the rotational invariance, the hyperbolicity in 2D is equivalent to the hyperbolicity in the $x$ direction. Therefore, it suffices to only analyze the real diagonalizability of the coefficient matrix in $x$ direction. 

In section \ref{subsec:hyperbolicity-HSWME}, we start from the analysis of the hyperbolicity of the HSWME \eqref{eq:HSWME}. We find that the HSWME \eqref{eq:HSWME} is hyperbolic if $\alpha_1\ne0$ or $\alpha_1=\beta_1=0$, and weakly hyperbolic if $\alpha_1=0$ and $\beta_1\ne0$. To fix this weakly hyperbolicity, we propose a new globally hyperbolic model, see section \ref{subsec:globally-hyperbolic-HSWME}. Then we generalize the $\beta$-HSMWE in 1D proposed in \cite{koellermeier2020analysis} to 2D and analyze its hyperbolicity in section \ref{subsec:hyperbolicity-beta-HSWME}. Lastly, we propose the general framework for constructing provable hyperbolic moment models with specified propogation speeds in section \ref{subsec:framework-general-closure} and give a concrete example in section \ref{subsec:example-general-closure}.

\subsection{Hyperbolicity of the HSWME}\label{subsec:hyperbolicity-HSWME}

In this part, we will prove the hyperbolicity of the HSWME \eqref{eq:HSWME} in 2D. This reduces to check the real diagonalizability of the coefficient matrix $A_H$ \eqref{eq:coefficient-matrix-HSMWE-x} in the $x$ direction.

{Note that the characteristic polynomial of $A_H$ was analyzed in Theorem 4.3.3 in \cite{verbiest2022}. However, the proof in \cite{verbiest2022} only shows that the eigenvalues of $A_H$ are real but not necessarily distinct. Therefore, the proof is incomplete since the real diagonalizability requires not only the real eigenvalues but also a complete set of eigenvectors.}
In this part, we will prove the real diagonalizability of $A_H$ with the aid of the associated polynomial sequence and show that the eigenvalues are related to the Gauss-Lobatto and Gauss-Legendre quadrature points.

To analyze the hyperbolicity, we use another ordering of variables:
\begin{equation}\label{eq:different-order-variable}
    W = (h, \,  {hu_m}, \, h\alpha_1, \, \cdots, \, h\alpha_N, \,  {hv_m}, \, h\beta_1, \, \cdots, \, h\beta_N)^T.
\end{equation}
Note that using different order of variables will not change the rotational invariance or the hyperbolicity of the model, but it does simplify the analysis.

Using this set of variables \eqref{eq:different-order-variable}, the coefficient matrix \eqref{eq:coefficient-matrix-HSMWE-x} in the $x$ direction in HSWME \eqref{eq:HSWME} can be written as
\begin{equation}\label{eq:matrix-A-another-order-variable}
    \tilde{A}_H(W) = 
\begin{pmatrix}
    \tilde{A}_{11}(W) & 0 \\
    \tilde{A}_{21}(W) & \tilde{A}_{22}(W) \\
\end{pmatrix},
\end{equation}
where the block matrices $\tilde{A}_{11}(W)\in\mathbb{R}^{(N+2)\times(N+2)}$, $\tilde{A}_{21}(W)\in\mathbb{R}^{(N+1)\times(N+2)}$ and $\tilde{A}_{22}(W)\in\mathbb{R}^{(N+1)\times(N+1)}$ are given by
\begin{equation}\label{eq:coefficient-matrix-HSWME-A11}
    \tilde{A}_{11}(W) = 
    \begin{pmatrix}
    0 & 1 \\
    gh -  {u_m^2} - \frac{1}{3}\alpha_1^2 & 2 {u_m} & \frac{2}{3}\alpha_1 \\
    -2 {u_m}\alpha_1 & 2\alpha_1 &  {u_m} & \frac{3}{5}\alpha_1 \\
    -\frac{2}{3}\alpha_1^2 & 0 & \frac{1}{3}\alpha_1 &  {u_m} & \frac{4}{7}\alpha_1 \\
    & & & \ddots & \ddots & \ddots \\
    & & & & \frac{N-2}{2N-3}\alpha_1 &  {u_m} & \frac{N+1}{2N+1}\alpha_1 \\
    & & & & & \frac{N-1}{2N-1}\alpha_1 &  {u_m}
    \end{pmatrix},
\end{equation}
\begin{equation}\label{eq:coefficient-matrix-HSWME-A21}
    \tilde{A}_{21}(W) = 
    \begin{pmatrix}
        - {u_m} {v_m}-\frac{\alpha_1\beta_1}{3} &  {v_m} & \frac{\beta_1}{3} \\
        -( {u_m}\beta_1+ {v_m}\alpha_1) & \beta_1 & 0 & \frac{1}{5}\beta_1 \\
        -\frac{2}{3}\alpha_1\beta_1 & 0 & -\frac{1}{3}\beta_1 & 0 & \frac{1}{7}\beta_1 \\
        & & & \ddots & \ddots & \ddots \\
        & & & & -\frac{1}{2N-3}\beta_1 & 0 & \frac{1}{2N+1}\beta_1 \\
        & & & & & -\frac{1}{2N-1}\beta_1 & 0 
    \end{pmatrix},
\end{equation}
and
\begin{equation}\label{eq:coefficient-matrix-HSWME-A22}
    \tilde{A}_{22}(W) =
    \begin{pmatrix}
         {u_m} & \frac{\alpha_1}{3} \\
        \alpha_1 &  {u_m} & \frac{2\alpha_1}{5} \\
        & \frac{2}{3}\alpha_1 &  {u_m} & \frac{3\alpha_1}{7} \\
        & & \ddots & \ddots & \ddots \\
        & & & \frac{N-1}{2N-3}\alpha_1 &  {u_m} & \frac{N}{2N+1}\alpha_1 \\
        & & & & \frac{N}{2N-1}\alpha_1 &  {u_m}
    \end{pmatrix}.
\end{equation}
Note that $\tilde{A}_{11}$ in \eqref{eq:coefficient-matrix-HSWME-A11} describes the effects of the variables in the $x$ directions on the time evolution of those variables and, therefore, {it is the same as the coefficient matrix in the HSWME model in 1D proposed in \cite{koellermeier2020analysis}}.
Next, we will analyze the eigenstructures of $\tilde{A}_{H}(W)$ with two different cases of $\alpha_1=0$ and $\alpha_1\neq0$.

\subsubsection{The case of $\alpha_1=0$}\label{subsubsec:alpha1=0}

 {
We first consider the case of $\alpha_1=0$. In this case, the matrices $\tilde{A}_{11}(W)$, $\tilde{A}_{21}(W)$ and $\tilde{A}_{22}(W)$ in \eqref{eq:coefficient-matrix-HSWME-A11}-\eqref{eq:coefficient-matrix-HSWME-A21}-\eqref{eq:coefficient-matrix-HSWME-A22} reduce to
\begin{equation*}\label{eq:coefficient-matrix-HSWME-A11-alpha1=0}
    \tilde{A}_{11}(W_0) = 
    \begin{pmatrix}
    0 & 1 \\
    gh - 
    {u_m^2} & 2
    {u_m} \\
     &  & 
     {u_m}I_{N}
    \end{pmatrix},
\end{equation*}
\begin{equation*}\label{eq:coefficient-matrix-HSWME-A21-alpha1=0}
    \tilde{A}_{21}(W_0) = 
    \begin{pmatrix}
        -{u_m}{v_m} & {v_m} & \frac{\beta_1}{3} \\
        -{u_m}\beta_1 & \beta_1 & 0 & \frac{1}{5}\beta_1 \\
        0 & 0 & -\frac{1}{3}\beta_1 & 0 & \frac{1}{7}\beta_1 \\
        & & & \ddots & \ddots & \ddots \\
        & & & & -\frac{1}{2N-3}\beta_1 & 0 & \frac{1}{2N+1}\beta_1 \\
        & & & & & -\frac{1}{2N-1}\beta_1 & 0 
    \end{pmatrix},
\end{equation*}
and
\begin{equation*}\label{eq:coefficient-matrix-HSWME-A22-alpha1=0}
    \tilde{A}_{22}(W_0) = {u_m}I_{N+1}.
\end{equation*}
where $W_0 := W|_{\alpha_1=0}$.
It is straightforward to compute the characteristic polynomial of $\tilde{A}_{11}(W_0)$ and $\tilde{A}_{22}(W_0)$:
\begin{equation*}\label{eq:characteristic-polynomial-A11-alpha1=0}
\begin{aligned}            
    \det(\lambda I_{N+2} - \tilde{A}_{11}(W_0)) ={}& \det
    \begin{pmatrix}
        \lambda & -1 \\
        -gh+ {u_m^2} & \lambda - 2{u_m} \\
         &  & (\lambda - {u_m})I_{N}
    \end{pmatrix} \\
    ={}& \brac{\lambda(\lambda - 2u_m) - (gh - u_m^2)} (\lambda-u_m)^N, \\
    ={}& \brac{(\lambda-u_m)^2 - gh} (\lambda-u_m)^N,
\end{aligned}
\end{equation*}
and
\begin{equation*}\label{eq:characteristic-polynomial-A22-alpha1=0}
    \det(\lambda I_{N+1} - \tilde{A}_{22}(W_0)) = (\lambda - u_m)^{N+1}.
\end{equation*}
Therefore, the characteristic polynomial of $\tilde{A}_H(W_0)$ is
\begin{equation*}\label{eq:characteristic-polynomial-A-alpha1=0}
    \det(\lambda I_{2N+3} - \tilde{A}_H(W_0)) = \det(\lambda I_{N+2} - \tilde{A}_{11}(W_0))\det(\lambda I_{N+1} - \tilde{A}_{22}(W_0)) = \brac{(\lambda-u_m)^2 - gh} (\lambda-u_m)^{2N+1}.
\end{equation*}
Thus, the eigenvalues of $\tilde{A}_H(W_0)$ are
\begin{equation}\label{eq:eigenvalue-A-alpha1=0}
    \lambda_{1,2} = u_m \pm \sqrt{gh}, \quad \lambda_{i} = u_m, \, i=3,4,\cdots,2N+3.
\end{equation}

Next, we proceed to analyze the eigenvectors of $\tilde{A}_H(W_0)$ for the repeated eigenvalue $u_m$. We consider the linear system 
\begin{equation*}
    (u_m I_{2N+3} - \tilde{A}_H(W_0)) {q} = 0,
\end{equation*}
where the eigenvector $q=(q_1^T, q_2^T)^T\in\mathbb{R}^{2N+3}$ with $q_1=(q_{1,1},q_{1,2},\cdots,q_{1,N+2})^T\in\mathbb{R}^{N+2}$ and $q_2=(q_{2,1}, q_{2,2}, \cdots, q_{2,N+1})^T\in\mathbb{R}^{N+1}$. The above system can be  reduced to two subsystems:
\begin{equation}
\label{eq:linear-system-eigenvector-alpha1=0}
\begin{aligned}
    (u_m I_{N+2} - \tilde{A}_{11}(W_0)) q_1 ={}& 0, \\
    - \tilde{A}_{21}(W_0) q_1 + (u_m I_{N+1} - \tilde{A}_{22}(W_0)) q_2 ={}& 0,
\end{aligned}    
\end{equation}
For the first equation in \eqref{eq:linear-system-eigenvector-alpha1=0}, we have
\begin{equation*}\label{eq:linear-system-eigenvector-alpha1=0-1}
    \begin{pmatrix}
        u_m & -1 \\
        -gh+u_m^2 & -u_m \\
        & & 0 \\
        & & & \ddots \\
        & & & & 0
    \end{pmatrix}
    \begin{pmatrix}
        q_{1,1} \\
        q_{1,2} \\
        q_{1,3} \\
        \vdots \\
        q_{1,N+2}
    \end{pmatrix}
    = 0,
\end{equation*}
from which we solve out
\begin{equation}\label{eq:linear-system-eigenvector-alpha1=0-1-solution}
    q_{1,1} = q_{1,2} = 0,
\end{equation}
by using the fact that $gh>0$.
For the second equation in \eqref{eq:linear-system-eigenvector-alpha1=0}, we have
\begin{equation*}\label{eq:linear-system-eigenvector-alpha1=0-2}
    \begin{pmatrix}
        -{u_m}{v_m} & {v_m} & \frac{\beta_1}{3} \\
        -{u_m}\beta_1 & \beta_1 & 0 & \frac{1}{5}\beta_1 \\
        0 & 0 & -\frac{1}{3}\beta_1 & 0 & \frac{1}{7}\beta_1 \\
        & & & \ddots & \ddots & \ddots \\
        & & & & -\frac{1}{2N-3}\beta_1 & 0 & \frac{1}{2N+1}\beta_1 \\
        & & & & & -\frac{1}{2N-1}\beta_1 & 0 
    \end{pmatrix}
    \begin{pmatrix}
        q_{1,1} \\
        q_{1,2} \\
        q_{1,3} \\
        \vdots \\
        q_{1,N+1}
        \\
        q_{1,N+2}
    \end{pmatrix}
    = 0.
\end{equation*}
which is equivalent to
\begin{equation}
\label{eq:linear-system-eigenvector-alpha1=0-2-reduced}
\begin{aligned}
    -{u_m}{v_m} q_{1,1} + {v_m} q_{1,2} + \frac{\beta_1}{3} q_{1,3} ={}& 0, \\
    -{u_m}\beta_1 q_{1,1} + \beta_1 q_{1,2} + \frac{1}{5}\beta_1 q_{1,4} ={}& 0, \\
    -\frac{1}{3}\beta_1 q_{1,3} + \frac{1}{7}\beta_1 q_{1,5} ={}& 0, \\
    \vdots \\
    -\frac{1}{2N-3}\beta_1 q_{1,N} + \frac{1}{2N+1}\beta_1 q_{1,N+2} ={}& 0, \\
    -\frac{1}{2N-1}\beta_1 q_{1,N+1} ={}& 0.
\end{aligned}    
\end{equation}
Using the relation \eqref{eq:linear-system-eigenvector-alpha1=0-1-solution}, it is easy to obtain the solution to \eqref{eq:linear-system-eigenvector-alpha1=0-2-reduced}: 
\begin{enumerate}
    \item If $\beta_1\ne0$, then we have $q_{1,3} = \cdots = q_{1,N+2} = 0$.

    \item If $\beta_1=0$, then \eqref{eq:linear-system-eigenvector-alpha1=0-2-reduced} is automatically satisfied.
\end{enumerate}
We summarize the above result in the following theorem:
\begin{thm}[eigenstructure of $\tilde{A}_H(W)$ in the case of $\alpha_1=0$]\label{thm:eigenvector-alpha1=0}
    Suppose that $\alpha_1=0$. Then the eigenvalues of $\tilde{A}_H(W)$ in \eqref{eq:matrix-A-another-order-variable} are real and they are given by
    \begin{equation}
        \lambda_{1,2} = u_m \pm \sqrt{gh}, \quad \lambda_{i} = u_m, \, i=3,4,\cdots,2N+3.
    \end{equation}
    Furthermore, for the eigenspace of the eigenvalue $u_m$, we have
    \begin{enumerate}
        \item If $\beta_1=0$, then the eigenspace of the eigenvalue $u_m$ is
        \begin{equation*}
            E = \{q=(q_{1,1}, q_{1,2}, \cdots, q_{1,N+2}, q_{2,1}, q_{2,2}, \cdots, q_{2,N+1})^T\in\mathbb{R}^{2N+3}: q_{1,1}=q_{1,2}=0 \}
            ,
        \end{equation*}
        with $\textrm{dim}(E) = 2N+1$. Thus the matrix $\tilde{A}_H(W)$ is real diagonalizable.
        
        \item If $\beta_1\ne0$, then the eigenspace of the eigenvalue $u_m$ is
        \begin{equation*}
            E = \{q=(q_{1,1}, q_{1,2}, \cdots, q_{1,N+2}, q_{2,1}, q_{2,2}, \cdots, q_{2,N+1})^T\in\mathbb{R}^{2N+3}: q_{1,i}=0, 1\le i\le N+2 \}
            ,
        \end{equation*}
        with $\textrm{dim}(E) = N+1$.
        Thus, the matrix $\tilde{A}_H(W)$ is not real diagonalizable.
    \end{enumerate}
\end{thm}
This immediately implies the following theorem:
\begin{thm}[hyperbolicity of the HSWME in the case of $\alpha_1=0$]\label{thm:real-diagonalizable-alpha1=0}
    Suppose that $\alpha_1=0$. We have the following conclusion for the hyperbolicity of the HSWME \eqref{eq:HSWME} in 2D:
    \begin{enumerate}
        \item     
        If $\beta_1=0$, then the HSWME  \eqref{eq:HSWME}  is hyperbolic;
        
        \item 
        If $\beta_1\ne0$, then the HSWME  \eqref{eq:HSWME}  is weakly hyperbolic.
    \end{enumerate}
\end{thm}
}

\subsubsection{The case of $\alpha_1\ne0$}\label{subsubsec:alpha1-ne0}

In this part, we consider the case of $\alpha_1\ne0$. 
Notice that both $\tilde{A}_{11}(W)$ in \eqref{eq:coefficient-matrix-HSWME-A11} and $\tilde{A}_{22}(W)$ in \eqref{eq:coefficient-matrix-HSWME-A22} are lower Hessenberg matrices.
Before the discussion, we review important properties of the Hessenberg matrix. These properties facilitate directly relating the eigenvalues of a Hessenberg matrix to the roots of some associated polynomial. We start with the definitions of the (unreduced) lower Hessenberg matrix and the associated polynomial sequence \cite{elouafi2009recursion}:
\begin{defn}[lower Hessenberg matrix]\label{defn:hessenberg}
  The matrix $A = (a_{ij})_{n\times n}$ is called lower Hessenberg matrix if $a_{ij}=0$ for $j>i+1$. It is called unreduced lower Hessenberg matrix if further $a_{i,i+1}\ne0$ for $i=1,2,\cdots,n-1$.
\end{defn}
\begin{defn}[associated polynomial sequence \cite{elouafi2009recursion}]\label{defn:associated-polynomial}
  Let $A = (a_{ij})_{n\times n}$ be an unreduced lower Hessenberg matrix. The associated polynomial sequence $\{q_i\}_{0\le i\le n}$ is defined as follows:
  \begin{equation}\label{eq:recurrence-polynomial}
    \begin{aligned}
        q_0(x) &= 1, \\
    q_i(x) &= \frac{1}{a_{i,i+1}} \brac{x q_{i-1}(x) - \sum_{j=1}^i a_{ij} q_{j-1}(x)}, \quad 1\le i\le n,
    \end{aligned}
  \end{equation}
with $a_{n,n+1}:=1$.  
\end{defn}
\begin{thm}[\cite{elouafi2009recursion}]\label{thm:root-is-eigenvalue}
  Let $A = (a_{ij})_{n\times n}$ be an unreduced lower Hessenberg matrix and $\{q_i\}_{0\le i\le n}$ is the associated polynomial sequence with $A$. The following conclusions hold true:
  \begin{enumerate}
    \item 
    If $\lambda$ is a root of $q_n$, then $\lambda$ is an eigenvalue of the matrix $A$ and a corresponding eigenvector is $(q_0(\lambda), q_1(\lambda), \cdots, q_{n-1}(\lambda))^T$.
    
        \item 
  If all the roots of $q_n$ are simple, then the characteristic polynomial of $A$ is given by 
    \begin{equation*}
      \det(x I - A) = \rho q_n(x),
    \end{equation*}
        with $\rho=\Pi_{i=1}^{n-1} a_{i,i+1}$.
  \end{enumerate} 
\end{thm}

With the aid of the associated polynomial sequence, we are able to obtain the analytical form of the characteristic polynomials of $\tilde{A}_{11}$ in \eqref{eq:coefficient-matrix-HSWME-A11} and $\tilde{A}_{22}$ in \eqref{eq:coefficient-matrix-HSWME-A22}.
\begin{lem}[characteristic polynomial of $\tilde{A}_{11}$ in \eqref{eq:coefficient-matrix-HSWME-A11} in the case of $\alpha_1\ne 0$]
     {Suppose that $\alpha_1\ne 0$.} Then associated polynomial sequence of $\tilde{A}_{11}$ in \eqref{eq:coefficient-matrix-HSWME-A11} is given by
    \begin{equation*}
    \begin{aligned}    
        q_0(x) &= 1, \\
        q_1(x) &= x, \\
        q_2(x) &= \frac{3(x- {u_m})^2 - 3gh + \alpha_1^2}{2\alpha_1},
    \end{aligned}
    \end{equation*}
    \begin{equation}\label{eq:associated-polynomial-A11}
        q_n(x) = \frac{2n-1}{n(n-1)\alpha_1}P_{n-1}'\brac{\frac{x- {u_m}}{\alpha_1}} \brac{(x- {u_m})^2 - gh - \alpha_1^2}, \quad 3\le n\le N+1,
    \end{equation}
    and
    \begin{equation*}
        q_{N+2}(x) = \frac{1}{N+1} P'_{N+1}\brac{ {\frac{x-u_m}{\alpha_1}}} \brac{(x-u_m)^2 - gh - \alpha_1^2}.
    \end{equation*}
     {Here $P_n$ is the $n$-th order Legendre polynomial on $[-1,1]$ with the standardization condition $P_n(1)=1$.}    

\end{lem}
\begin{proof}
    For convenience, we first introduce the notations:
    \begin{equation*}
        \xi := \frac{x-u_m}{\alpha
        _1},
    \end{equation*}
    and
    \begin{equation*}
        p_g(x) := (x-u_m)^2 - gh - \alpha_1^2.
    \end{equation*}

    The first several associated polynomials can be obtained by direct computation:
    \begin{equation*}
        q_0(x) = 1.
    \end{equation*}
    \begin{equation*}
        q_1(x) = \frac{1}{a_{12}}(xq_0(x) - a_{11}q_0(x)) = \frac{1}{1}(x - 0) = x.
    \end{equation*}
    \begin{equation*}
    \begin{aligned}            
        q_2(x) &= \frac{1}{a_{23}}(xq_1(x) - (a_{21}q_0(x) + a_{22}q_1(x))) \\
        &= \frac{1}{\frac{2}{3}\alpha_1}\brac{x^2 - (gh-u_m^2-\frac{1}{3}\alpha_1^2+2 {u_m}x)} \\
        &= \frac{3(x- {u_m})^2 - 3gh + \alpha_1^2}{2\alpha_1}.
    \end{aligned}
    \end{equation*}
    \begin{equation*}
    \begin{aligned}
        q_3(x) &= \frac{1}{a_{34}}\brac{xq_2(x) - (a_{31}q_0(x) + a_{32}q_1(x) + a_{33}q_2(x))} \\
        &= \frac{1}{\frac{3}{5}\alpha_1}\brac{x\frac{3(x- {u_m})^2 - 3gh + \alpha_1^2}{2\alpha_1} - (-2 {u_m}\alpha_1 + 2\alpha_1x +  {u_m}\frac{3(x- {u_m})^2 - 3gh + \alpha_1^2}{2\alpha_1})} \\
        &= \frac{5(x- {u_m})((x- {u_m})^2 - gh - \alpha_1^2)}{2\alpha_1^2} \\
        &= \frac{5\xi}{2\alpha_1}p_g(x) \\
        &= \frac{5}{6\alpha_1} P_2'(\xi) p_g(x).
    \end{aligned}
    \end{equation*}
    \begin{equation*}
    \begin{aligned}
        q_4(x) &= \frac{1}{a_{45}}(xq_3(x) - (a_{41}q_0(x) + a_{42}q_1(x) + a_{43}q_2(x) + a_{44}q_3(x))) \\
        &= \frac{1}{\frac{4}{7}\alpha_1}\brac{(x- {u_m})\frac{5\xi}{2\alpha_1}p_g(x) - (-\frac{2}{3}\alpha_1^2 + 0 + \frac{1}{3}\alpha_1 \frac{3(x- {u_m})^2 - 3gh + \alpha_1^2}{2\alpha_1})} \\
        &= \frac{7(5(x- {u_m})^2-\alpha_1^2)((x- {u_m})^2 - gh - \alpha_1^2)}{8\alpha_1^3} \\
        &= \frac{7(5\xi^2-1)}{8\alpha_1}p_g(x) \\
        &= \frac{7}{12\alpha_1} P_3'(x) p_g(x).
    \end{aligned}
    \end{equation*}

    Now we prove by induction and assume that \eqref{eq:associated-polynomial-A11} holds for $4\le n\le k$ with $k\le N$. We will prove it also holds for $n = k+1$.
    \begin{equation*}
    \begin{aligned}
        q_{k+1}(x) &= \frac{1}{a_{k+1,k+2}} \brac{xq_k(x) - \sum_{j=1}^{k+1} a_{k+1,j}q_{j-1}(x)} \\
        &= \frac{1}{a_{k+1,k+2}} \brac{xq_k(x) - (a_{k+1,k}q_{k-1}(x) + a_{k+1,k+1}q_k(x))} \\
        &= \frac{1}{\frac{k+1}{2k+1}\alpha_1}\brac{(x-u)\frac{2k-1}{k(k-1)\alpha_1}P'_{k-1}(\xi)p_g(x) - \frac{k-2}{2k-3}\alpha_1\frac{2k-3}{(k-1)(k-2)\alpha_1}P'_{k-2}(\xi)p_g(x)} \\
        &= \frac{2k+1}{(k+1)\alpha_1}\brac{\frac{2k-1}{k(k-1)}\xi P'_{k-1}(\xi) - \frac{1}{k-1}P'_{k-2}(\xi)}p_g(x) \\
        &= \frac{2k+1}{(k+1)\alpha_1}\brac{\frac{1}{k(k-1)} ((k-1) P'_{k}(\xi) + k P'_{k-2}(\xi)) - \frac{1}{k-1}P'_{k-2}(\xi)}p_g(x) \\
        &= \frac{2k+1}{k(k+1)\alpha_1} P'_{k}(\xi) p_g(x).
        \end{aligned}        
    \end{equation*}
    where we use the relation
    $(2k-1)xP'_{k-1}(x) = (k-1)P'_{k}(x) + k P'_{k-2}(x)$.

    Lastly, we compute $q_{N+2}(x)$:
    \begin{equation}
    \begin{aligned}            
        q_{N+2}(x) &= \frac{1}{a_{N+2,N+3}} \brac{xq_{N+1}(x) - (a_{N+2,N+1}q_{N}(x) + a_
        {N+2,N+2}q_{N+1}(x))} \\
        &= xq_{N+1}(x) - (a_{N+2,N+1}q_{N}(x) + a_
        {N+2,N+2}q_{N+1}(x)) \\
        &= \frac{1}{N+1} P'_{N+1}(\xi) p_g(x).
    \end{aligned}
    \end{equation}
    The proof is completed.
\end{proof}

\begin{cor}\label{cor:characteristic-polynomial-A-11}
     {Suppose that $\alpha_1\ne 0$.} Since the roots of $P'_{N+1}(\xi)$ are real and distinct,  the characteristic polynomial of $\tilde{A}_{11}$ is
    \begin{equation}
        \det(x I - \tilde{A}_{11}) = \frac{N!}{(2N+1)!!} \alpha_1^N P_{N+1}'\brac{\frac{x- {u_m}}{\alpha_1}} \brac{(x-u)^2 - gh - \alpha_1^2}.
    \end{equation}    
    Threfore, the eigenvalues of $\tilde{A}_{11}$ are given by
    \begin{equation*}
        \lambda_{1,2} =  {u_m} \pm \sqrt{gh+\alpha_1^2},
    \end{equation*}
    and
    \begin{equation*}
        \lambda_{i+2} =  {u_m} + r_i\alpha_1, \quad i = 1, 2, \cdots, N.
    \end{equation*}
    where $r_i$ with $i = 1, 2, \cdots, N$ are the roots of $P'_{N+1}(\xi)$, i.e. the Gauss-Lobatto quadrature points in $[-1, 1]$.
\end{cor}

\begin{lem}[associated polynomial sequence of $\tilde{A}_{22}$ in the case of $\alpha_1\ne 0$]\label{lem:characteristic-polynomial-A-22}
     {Suppose that $\alpha_1\ne 0$.} 
    For the matrix $\tilde{A}_{22}$ given in \eqref{eq:coefficient-matrix-HSWME-A22}, the associated polynomial sequences satisfy:
    \begin{equation}\label{eq:associated-polynomial-HSWME-A22}
        q_n(x) = (2n+1)P_n\brac{\frac{x- {u_m}}{\alpha_1}}, \quad 0\le n\le N.
    \end{equation}
    and\
    \begin{equation}
        q_{N+1}(x) = (N+1)\alpha_1 P_{N+1}\brac{\frac{x- {u_m}}{\alpha_1}}.
    \end{equation}
\end{lem}
\begin{proof}
    We compute the associated polynomial sequence by recurrence relation:
    \begin{equation*}
        q_0(x) = 1,
    \end{equation*}
    \begin{equation*}
        q_1(x) = \frac{1}{a_{12}}(xq_0(x) - a_{11}q_0(x)) = \frac{x- {u_m}}{\frac{\alpha_1}{3}} = \frac{3(x- {u_m})}{\alpha_1} = 3\xi = 3 P_1(\xi),
    \end{equation*}
    \begin{equation*}
        q_2(x) = \frac{1}{a_{23}}\brac{xq_1(x) - (a_{21}q_0(x) + a_{22}q_1(x))} = \frac{1}{\frac{2}{5}\alpha_1}\brac{(x- {u_m})\frac{3(x- {u_m})}{\alpha_1} - \alpha_1} = \frac{5}{2}(3\xi^2-1) = 5P_2(\xi),
    \end{equation*}
    with $\xi:=\frac{x- {u_m}}{\alpha_1}$.

    Now we prove by induction. We assume that the formula \eqref{eq:associated-polynomial-HSWME-A22} holds for $2\le n\le k$. We will prove it also holds for $n = k+1$.
    \begin{equation*}
    \begin{aligned}
        q_{k+1}(x) &= \frac{1}{a_{k+1,k+2}} \brac{xq_k(x) - \sum_{j=1}^{k+1} a_{k+1,j}q_{j-1}(x)} \\
        &= \frac{1}{a_{k+1,k+2}} \brac{xq_k(x) - (a_{k+1,k}q_{k-1}(x) + a_{k+1,k+1}q_k(x))} \\
        &= \frac{1}{\frac{k+1}{2k+3}\alpha_1}\brac{(x- {u_m})(2k+1)P_k(\xi) - \frac{k}{2k-1}\alpha_1(2k-1)P_{k-1}(\xi)} \\
        &= \frac{2k+3}{k+1}\brac{(2k+1)\xi P_k(\xi) - k P_{k-1}(\xi)} \\
        &= \frac{2k+3}{k+1} (k+1)P_{k+1}(\xi) \\
        &= (2k+3) P_{k+1}(\xi),
    \end{aligned}        
    \end{equation*}
    
    Lastly, we compute
    \begin{equation*}
        q_{N+1}(x) = xq_N(x) - (a_{N+1,N}q_{k-1}(x) + a_{N+1,N+1}q_k(x)) = (N+1)\alpha_1 P_{N+1}\brac{\frac{x- {u_m}}{\alpha_1}}.
    \end{equation*}
     {This completes the proof.}
\end{proof}

\begin{cor}[characteristic polynomial of $\tilde{A}_{22}$ in the case of $\alpha_1\ne 0$]\label{cor:characteristic-polynomial-A-22}
     {Suppose that $\alpha_1\ne 0$.}
    The matrix $\tilde{A}_{22}$ is real diagonalizable and its characteristic polynomial is 
    \begin{equation}
        \det(\lambda I - \tilde{A}_{22}) = \frac{(N+1)!}{(2N+1)!!}\alpha_1^{N+1} P_{N+1}\brac{\frac{x- {u_m}}{\alpha_1}}.
    \end{equation}
    Moreover, the eigenvalue of $\tilde{A}_{22}$ is given by
    \begin{equation}
         {\lambda_i =  {u_m} + s_i\alpha_1, \quad i = 1, 2, \cdots, N+1,}
    \end{equation}
    where $s_i$ for $i = 1, 2, \cdots, N+1$ are the roots of Legendre polynomial $P_{N+1}(\xi)$, i.e. the Gauss-Legendre quadrature points in $[-1, 1]$.
\end{cor}

Since the roots of both $P_{N+1}'(\xi)$ and $P_{N+1}(\xi)$ are all distinct, and $P_{N+1}'(\xi)$ and $P_{N+1}(\xi)$ have no common roots, we immediately have that all the eigenvalues of $\tilde{A}_H$ (and also $A_H$) are real and distinct. The result is summarized as follows:
\begin{thm}[real diagonalizability of $A_H$ in the case of $\alpha_1\ne0$]\label{thm:real-diagonalizability-AH}
    The matrix $A_H$ is real diagonalizable. Its characteristic polynomial is given by:
    \begin{equation}
        \det(\lambda I - A_H) = \frac{N!(N+1)!}{((2N+1)!!)^2} \alpha_1^{2N+1} P_{N+1}'\brac{\frac{x- {u_m}}{\alpha_1}} P_{N+1}\brac{\frac{x- {u_m}}{\alpha_1}} \brac{(x- {u_m})^2 - gh - \alpha_1^2}.
    \end{equation}
    Moreover, the eigenvalues are given by
    \begin{equation*}
        \lambda_{1,2} =  {u_m} \pm \sqrt{gh + \alpha_1^2},
    \end{equation*}
    \begin{equation*}
        \lambda_{i+2} =  {u_m} + r_i\alpha_1, \quad i = 1, 2, \cdots, N,
    \end{equation*}
    \begin{equation*}
        \lambda_{i+N+2} =  {u_m} + s_i\alpha_1, \quad i = 1, 2, \cdots, N+1,
    \end{equation*}
    where $r_i$ for $i = 1, 2, \cdots, N$ are the roots of the derivative of $P'_{N+1}(\xi)$
    and $s_i$ for $i = 1, 2, \cdots, N+1$ are the roots of $P_{N+1}(\xi)$. Here $P_n$ is the $n$-th order Legendre polynomial on $[-1,1]$ with the standardization condition $P_n(1)=1$.
\end{thm}

\begin{rem}
    {The eigenvalues of the coefficient matrix for the 2D HSWME model have a very similar structure as the 1D HSWME model. This is due to the reason that $\tilde{A}_{11}$ in \eqref{eq:coefficient-matrix-HSWME-A11} describes the effects of the variables in the $x$ directions on the time evolution of those variables and is the same as the coefficient matrix for the 1D HSWME model.
    }
\end{rem}

Combining the real diagonalizability of $A_H$ in Theorem \ref{thm:real-diagonalizability-AH} with the rotational invariance in Theorem \ref{thm:rotation-invariance-HSWME}, we have the hyperbolicity of the HSWME \eqref{eq:HSWME} in 2D:
\begin{thm}[hyperbolicity of the HSWME in the case of $\alpha_1\ne0$]\label{thm:real-diagonalizable-alpha1-ne0}
    Suppose that $\alpha_1\ne0$. Then the HSWME model \eqref{eq:HSWME} in 2D is hyperbolic.
\end{thm}

Combining the results in Theorem \ref{thm:real-diagonalizable-alpha1=0} and Theorem \ref{thm:real-diagonalizable-alpha1-ne0}, we have the following conclusion for the hyperbolicity of the HSWME \eqref{eq:HSWME} in 2D. This concludes the analysis of the hyperbolicity of the HSWME model in 2D.
\begin{thm}[hyperbolicity of the HSWME]
    The HSWME model \eqref{eq:HSWME} in 2D is hyperbolic if $\alpha_1\ne0$ or $\alpha_1=\beta_1=0$, and weakly hyperbolic if $\alpha_1=0$ and $\beta_1\ne0$.
\end{thm}

 {
\subsection{A new globally hyperbolic SWME model}\label{subsec:globally-hyperbolic-HSWME}
From the analysis in section \ref{subsec:hyperbolicity-HSWME}, we know that the HSWME model \eqref{eq:HSWME} in 2D is weakly hyperbolic if $\alpha_1=0$ and $\beta_1\neq0$, and strongly hyperbolic otherwise. In this part, we propose a new globally hyperbolic shallow water moment model by modifying the coefficient matrices of the HSWME model \eqref{eq:HSWME}. The new model is hyperbolic for any state $\alpha_1$ and $\beta_1\in\mathbb{R}$.

We notice that in the analysis of in section \ref{subsubsec:alpha1=0}, the weakly hyperbolicity for $\alpha_1=0$ and $\beta_1\ne0$ is due to the fact that the algebraic and geometric multiplicities of the eigenvalue $u_m$ do not coincide. To be more specific, the eigenspace of the eigenvalue $u_m$ depends on the second equation in \eqref{eq:linear-system-eigenvector-alpha1=0}. Therefore, we propose to modify $\tilde{A}_{21}$ in the following way:
\begin{equation}
    \tilde{A}_{21}(W) = 
    \begin{pmatrix}
        -{u_m}{v_m}-\frac{\alpha_1\beta_1}{3} & {v_m} & 0 \\
        -({u_m}\beta_1+{v_m}\alpha_1) & \beta_1 & 0 &  \\
        -\frac{2}{3}\alpha_1\beta_1 & 0 & 0 & \\
        & & & 0_{(N-2)\times(N-1)} 
    \end{pmatrix},
\end{equation}
It is easy to check that the eigenspace for the eigenvalue $u_m$ is now complete in the case of $\alpha_1=0$.

Next, the matrix $\tilde{A}_{22}$ should be modified as well to ensure the rotational invariance. We propose to modify the block $\tilde{A}_{22}$ and the matrix becomes
\begin{equation}\label{eq:coefficient-matrix-HSWME-A22-new}    
\tilde{A}_{22}(W) =
\begin{pmatrix}
    {u_m} & \frac{2}{3}\alpha_1 \\
    \alpha_1 & {u_m} & \frac{3}{5}\alpha_1 \\
    & \frac{1}{3}\alpha_1 & {u_m} & \frac{4}{7}\alpha_1 \\
    & & \ddots & \ddots & \ddots \\
    & & & \frac{N-2}{2N-3}\alpha_1 & {u_m} & \frac{N+1}{2N+1}\alpha_1 \\
    & & & & \frac{N-1}{2N-1}\alpha_1 & {u_m}
\end{pmatrix}.
\end{equation}
It is difficult to derive the closed form of the characteristic polynomial of $\tilde{A}_{22}(W)$ in \eqref{eq:coefficient-matrix-HSWME-A22-new} for arbitrary order of moment $N\ge1$. However, it is easy to see that $\tilde{A}_{22}(W)$ in \eqref{eq:coefficient-matrix-HSWME-A22-new} is real diagonalizable since it is a tridiagonal matrix where each pair of off-diagonal entries has the same sign.
Moreover, we numerically verify that the eigenvalues of $\tilde{A}_{22}(W)$ in \eqref{eq:coefficient-matrix-HSWME-A22-new} are all distinct from those of $\tilde{A}_{11}(W)$ in \eqref{eq:coefficient-matrix-HSWME-A11} at least for $N\le 100$. Therefore, the new model is globally hyperbolic.

Now we write the coefficient matrix of the new model in the original order of variables:
\begin{equation}\label{eq:coefficient-matrix-HSWME-globally-hyperbolic-x}
    A_G = 
    \begin{aligned}
        \begin{pmatrix}
        0 & 1 & 0 \\
        -u_m^2 - \frac{\alpha_1^2}{3} + gh & 2u_m & 0 & \frac{2\alpha_1}{3} & 0 \\
        -u_mv_m - \frac{\alpha_1\beta_1}{3} & v_m & u_m & 0 & \frac{2\alpha_1}{3} \\
        -2u_m\alpha_1 & 2\alpha_1 & 0 & u_m & 0 & \frac{3}{5}\alpha_1 & 0 & \cdots & 0 & 0 \\
        -(u_m\beta_1+v_m\alpha_1) & \beta_1 & \alpha_1 & 0 & u_m & 0 & \frac{3}{5}\alpha_1 & \cdots & 0 & 0 \\
        - \frac{2}{3}\alpha_1^2 & 0 & 0 & \frac{1}{3}\alpha_1 & 0 & u_m & 0 &  \cdots & 0 & 0 \\
        - \frac{2}{3}\alpha_1\beta_1 & 0 & 0 & 0 & \frac{1}{3}\alpha_1 & 0 & u_m & \cdots & 0 & 0 \\
         &  &  &  &  & \ddots & \ddots & \ddots & \frac{N+1}{2N+1}\alpha_1 & 0 \\
         &  &  &  &  & \ddots & \ddots & \ddots & 0 & \frac{N+1}{2N+1}\alpha_1 \\
         &  &  &  &  &  & \frac{N-1}{2N-1}\alpha_1 & 0 & u_m & 0 \\
         &  &  &  &  &  & 0 & \frac{N-1}{2N-1}\alpha_1 & 0 & u_m
        \end{pmatrix}
    \end{aligned}   
\end{equation}
in the $x$ direction and
\begin{equation}\label{eq:coefficient-matrix-HSWME-globally-hyperbolic-y}
    B_G = 
    \begin{aligned}
        \begin{pmatrix}
        0 & 0 & 1 \\
        -u_mv_m - \frac{\alpha_1\beta_1}{3} & v_m & u_m & \frac{2\beta_1}{3} & 0 \\
        -v_m^2 - \frac{\beta_1^2}{3} + gh & 0 & 2v_m & 0 & \frac{2\beta_1}{3} \\
        -(u_m\beta_1+v_m\alpha_1) & \beta_1 & \alpha_1 & v_m & 0 & \frac{3}{5} & 0 \\
        -2v_m\beta_1 & 0 & 2\beta_1 & 0 & v_m & 0 & \frac{3}{5}\beta_1 \\
        - \frac{2}{3}\alpha_1\beta_1 & 0 & 0 & \frac{1}{3}\beta_1 & 0 & v_m & 0 & \cdots & 0 & 0 \\
        - \frac{2}{3}\beta_1^2 & 0 & 0 & 0 & \frac{1}{3}\beta_1 & 0 & v_m &  \cdots & 0 & 0 \\
         &  &  &  &  & \ddots & \ddots & \ddots & \frac{N+1}{2N+1}\beta_1 & 0 \\
         &  &  &  &  & \ddots & \ddots & \ddots & 0 & \frac{N+1}{2N+1}\beta_1 \\
         &  &  &  &  &  & \frac{N-1}{2N-1}\beta_1 & 0 & v_m & 0 \\
         &  &  &  &  & & 0 & \frac{N-1}{2N-1}\beta_1 & 0 & v_m
        \end{pmatrix}
    \end{aligned}   
\end{equation}
in the $y$ direction. We summarize the hyperbolicity of the new shallow water moment model in the following theorem:
\begin{thm}[global hyperbolicity of the new shallow water moment model]\label{thm:globally-hyperbolicity-HSWME-new}
    The new shallow water moment model with the coefficient matrices given by $A_G$ in \eqref{eq:coefficient-matrix-HSWME-globally-hyperbolic-x} and $B_G$ in \eqref{eq:coefficient-matrix-HSWME-globally-hyperbolic-y} is globally hyperbolic for all feasible orders of at least up to $N = 100$.
\end{thm}
}

\subsection{Hyperbolicity of the $\beta$-HSWME}\label{subsec:hyperbolicity-beta-HSWME}

In \cite{koellermeier2020analysis}, a new version of shallow water moment model, called the $\beta$-HSWME, is proposed by modifying the last row of the coefficient matrix, so that predefined propagation speeds can be obtained.

The coefficient matrix of the $\beta$-HSWME \cite{koellermeier2020analysis} reads as:
\begin{equation}\label{eq:coefficient-matrix-HSWME-A11-beta}
    \tilde{A}_{\beta,11}(W) = 
    \begin{pmatrix}
    0 & 1 \\
    gh -  {u_m}^2 - \frac{1}{3}\alpha_1^2 & 2 {u_m} & \frac{2}{3}\alpha_1 \\
    -2 {u_m}\alpha_1 & 2\alpha_1 &  {u_m} & \frac{3}{5}\alpha_1 \\
    -\frac{2}{3}\alpha_1^2 & 0 & \frac{1}{3}\alpha_1 &  {u_m} & \frac{4}{7}\alpha_1 \\
    & & & \ddots & \ddots & \ddots \\
    & & & & \frac{N-2}{2N-3}\alpha_1 &  {u_m} & \frac{N+1}{2N+1}\alpha_1 \\
    & & & & & \frac{(N-1)(2N+1)}{(N+1)(2N-1)}\alpha_1 &  {u_m}
    \end{pmatrix}.
\end{equation}
In Theorem 3.5 in \cite{koellermeier2020analysis}, it was shown that  the characteristic polynomial of $\tilde{A}_{\beta,11}$ in \eqref{eq:coefficient-matrix-HSWME-A11-beta} is related to the Legendre polynomial by numerical computations up to order $N=100$. Here, we will prove that this holds true for any $N\ge 1$:
\begin{lem}[characteristic polynomial of $\tilde{A}_{\beta,11}$  {in the case of $\alpha_1\ne0$}]\label{lem:characteristic-polynomial-A-beta-11}
    The characteristic polynomial of $\tilde{A}_{\beta,11}$ in \eqref{eq:coefficient-matrix-HSWME-A11-beta} is
    \begin{equation}
        \det(\lambda I - \tilde{A}_{\beta, 11}) = \frac{N!}{(2N-1)!!} \alpha_1^N P_{N}\brac{\frac{x- {u_m}}{\alpha_1}} \brac{(x- {u_m})^2 - gh - \alpha_1^2}.
    \end{equation}  
\end{lem}
\begin{proof}
    Notice that the matrix $\tilde{A}_{\beta,11}$ in \eqref{eq:coefficient-matrix-HSWME-A11-beta} only differs from $\tilde{A}_{11}$ in \eqref{eq:coefficient-matrix-HSWME-A11} in the last row. Thus, the associated polynomial sequences for two matrices are the same for $q_i$ with $0\le i\le N+1$. We only need to compute $q_{N+2}(x)$:
    \begin{equation*}
    \begin{aligned}            
        q_{N+2}(x) &= \frac{1}{a_{N+2,N+3}}(xq_{N+1}(x) - \sum_{j=1}^{N+2} a_{N+2,j}q_{j-1}(x)) \\
        &= xq_{N+1}(x) - (a_{N+2,N+1}q_N(x) + a_{N+2,N+2}q_{N+1}(x)) \\
        &= (x- {u_m}) q_{N+1}(x) - \frac{(N-1)(2N+1)}{(N+1)(2N-1)} q_N(x) \\
        &= (x- {u_m}) \frac{2N+1}{N(N+1)\alpha_1} P'_{N}(\xi)p_g(x) - \frac{(N-1)(2N+1)}{(N+1)(2N-1)} \frac{2N-1}{N(N-1)\alpha_1} P'_{N-1}(\xi)p_g(x) \\
        &= \frac{2N+1}{N(N+1)} \brac{\xi P'_{N}(\xi) - P'_{N-1}(\xi)}p_g(x) \\
        &= \frac{2N+1}{N(N+1)} N P_N(\xi) p_g(x) \\
        &= \frac{2N+1}{N+1} P_{N}(\xi) p_g(x).
    \end{aligned}
    \end{equation*}
    Here we denote $\xi:=\frac{x- {u_m}}{\alpha_1}$ and $p_g(x):=(x- {u_m})^2-gh-\alpha_1^2$ and use the relation $\xi P'_{n}(\xi) - P'_{n-1}(\xi) = n P_{n}(\xi)$.
    
    Therefore, the characteristic polynomial of $\tilde{A}_{\beta,11}$ is
    \begin{equation*}
        \det(\lambda I - \tilde{A}_{\beta, 11}) = \frac{(N+1)!}{(2N+1)!!}\alpha_1^N \frac{2N+1}{N+1} P_{N}(\xi) p_g(x) = \frac{N!}{(2N-1)!!} \alpha_1^N P_{N}(\xi) p_g(x).
    \end{equation*}
\end{proof}

With the matrix $\tilde{A}_{\beta,11}$ at hand, there is still some degree of freedom to choose the matrix $\tilde{A}_{\beta,22}$ to make the matrix $\tilde{A}_{\beta}$ hyperbolic. One simple choice is to keep $\tilde{A}_{22}$ unchanged. In this case, the corresponding matrix in $x$ direction is
\begin{equation*}
    \tilde{A}_{\beta} = 
    \begin{pmatrix}
    \tilde{A}_{\beta,11} & 0 \\
    \tilde{A}_{\beta,21} & \tilde{A}_{22}
    \end{pmatrix},
\end{equation*}
where $\tilde{A}_{\beta,21}$ is determined by the rotational invariance constraint given in Theorem \ref{thm:rotational-invariance-block-matrix}:
\begin{equation}\label{eq:coefficient-matrix-HSWME-beta-A21}
    \tilde{A}_{\beta,21}(W) = 
    \begin{pmatrix}
        - {u_m} {v_m}-\frac{\alpha_1\beta_1}{3} &  {v_m} & \frac{\beta_1}{3} \\
        -( {u_m}\beta_1+ {v_m}\alpha_1) & \beta_1 & 0 & \frac{1}{5}\beta_1 \\
        -\frac{2}{3}\alpha_1\beta_1 & 0 & -\frac{1}{3}\beta_1 & 0 & \frac{1}{7}\beta_1 \\
        & & & \ddots & \ddots & \ddots \\
        & & & & -\frac{1}{2N-3}\beta_1 & 0 & -\frac{1}{2N+1}\beta_1 \\
        & & & & & \frac{N^2-2N-1}{(2N-1)(N+1)}\beta_1 & 0 
    \end{pmatrix}.
\end{equation}

We can also write the coefficient matrices in the original order of variables:
\begin{equation}\label{eq:coefficient-matrix-beta-x}
    A_{\beta} = 
    \begin{aligned}
        \begin{pmatrix}
        0 & 1 & 0 \\
        - {u_m}^2 - \frac{\alpha_1^2}{3} + gh & 2 {u_m} & 0 & \frac{2\alpha_1}{3} & 0 \\
        - {u_m} {v_m} - \frac{\alpha_1\beta_1}{3} &  {v_m} &  {u_m} & \frac{\beta_1}{3} & \frac{\alpha_1}{3} \\
        -2 {u_m}\alpha_1 & 2\alpha_1 & 0 &  {u_m} & 0 & \frac{3}{5}\alpha_1 & 0 & \cdots & 0 & 0 \\
        -( {u_m}\beta_1+ {v_m}\alpha_1) & \beta_1 & \alpha_1 & 0 &  {u_m} & \frac{1}{5}\beta_1 & \frac{2}{5}\alpha_1 & \cdots & 0 & 0 \\
        - \frac{2}{3}\alpha_1^2 & 0 & 0 & \frac{1}{3}\alpha_1 & 0 &  {u_m} & 0 &  \cdots & 0 & 0 \\
        - \frac{2}{3}\alpha_1\beta_1 & 0 & 0 & -\frac{1}{3}\beta_1 & \frac{2}{3}\alpha_1 & 0 &  {u_m} & \cdots & 0 & 0 \\
         &  &  &  &  & \ddots & \ddots & \ddots & \frac{N+1}{2N+1}\alpha_1 & 0 \\
         &  &  &  &  & \ddots & \ddots & \ddots & \frac{1}{2N+1}\beta_1 & \frac{N}{2N+1}\alpha_1 \\
         &  &  &  &  &  & \frac{(N-1)(2N+1)}{(2N-1)(N+1)}\alpha_1 & 0 &  {u_m} & 0 \\
         &  &  &  &  &  & \frac{N^2-2N-1}{(2N-1)(N+1)}\beta_1 & \frac{N}{2N-1}\alpha_1 & 0 &  {u_m}
        \end{pmatrix}
    \end{aligned}   
\end{equation}
in the $x$ direction and
\begin{equation}\label{eq:coefficient-matrix-beta-y}
    B_{\beta} = 
    \begin{aligned}
        \begin{pmatrix}
        0 & 0 & 1 \\
        - {u_m} {v_m} - \frac{\alpha_1\beta_1}{3} &  {v_m} &  {u_m} & \frac{\beta_1}{3} & \frac{\alpha_1}{3} \\
        - {v_m}^2 - \frac{\beta_1^2}{3} + gh & 0 & 2 {v_m} & 0 & \frac{2\beta_1}{3} \\
        -( {u_m}\beta_1+ {v_m}\alpha_1) & \beta_1 & \alpha_1 &  {v_m} & 0 & \frac{2}{5}\beta_1 & \frac{1}{5}\alpha_1 \\
        -2 {v_m}\beta_1 & 0 & 2\beta_1 & 0 &  {v_m} & 0 & \frac{3}{5}\beta_1 \\
        - \frac{2}{3}\alpha_1\beta_1 & 0 & 0 & \frac{2}{3}\beta_1 & -\frac{1}{3}\alpha_1 &  {v_m} & 0 & \cdots & 0 & 0 \\
        - \frac{2}{3}\beta_1^2 & 0 & 0 & 0 & \frac{1}{3}\beta_1 & 0 &  {v_m} &  \cdots & 0 & 0 \\
         &  &  &  &  & \ddots & \ddots & \ddots & \frac{N}{2N+1}\beta_1 & \frac{1}{2N+1}\alpha_1 \\
         &  &  &  &  & \ddots & \ddots & \ddots & 0 & \frac{N+1}{2N+1}\beta_1 \\
         &  &  &  &  &  & \frac{N}{2N-1}\beta_1 & \frac{N^2-2N-1}{(2N-1)(N+1)}\alpha_1 &  {v_m} & 0 \\
         &  &  &  &  & & 0 & \frac{(N-1)(2N+1)}{(2N-1)(N+1)}\beta_1 & 0 &  {v_m}
        \end{pmatrix}
    \end{aligned}   
\end{equation}
in the $y$ direction.

Since $P_N(\xi)$ in Lemma \ref{lem:characteristic-polynomial-A-beta-11} and $P_{N+1}(\xi)$ in Corollary \ref{cor:characteristic-polynomial-A-22} both have real and distinct roots and there are no common roots due to the interlacing property of zeros of orthogonal polynomials in the case of $\alpha_1\ne0$, we have that the coefficient matrix  \eqref{eq:coefficient-matrix-beta-x} in the $x$ direction has real and distinct roots and thus real diagonalizable. Moreover, the analysis of the model in the case of $\alpha_1=0$ is similar to the one in section \ref{subsubsec:alpha1=0}. Therefore, combining with the rotational invariance in Theorem \ref{thm:general-closure-relation-rotational-invariance}, we immediately have the following result: 
\begin{thm}[hyperbolicity of the $\beta$-HSWME]
    The $\beta$-HSWME model in 2D with the coefficient matrices given by \eqref{eq:coefficient-matrix-beta-x} in $x$ direction and \eqref{eq:coefficient-matrix-beta-y} in $y$ direction is hyperbolic if $\alpha_1\ne0$ or $\alpha_1=\beta_1=0$ and weakly hyperbolic if $\alpha_1=0$ and $\beta_1\ne0$.
\end{thm}

\begin{rem}
    In the above procedure, we take the simple choice to keep $\tilde{A}_{22}$ unchanged in $\tilde{A}_{\beta}$. However, one may also modify $\tilde{A}_{22}$ to make the matrix $\tilde{A}_{\beta}$ hyperbolic following the framework in the next section. This will result in a different form of the $\beta$-HSWME model in 2D.
\end{rem}

\subsection{A framework for constructing general closure relations with rotational invariance and hyperbolicity}\label{subsec:framework-general-closure}

Besides the previous hyperbolic shallow water moment models, we can also modify both $\tilde{A}_{11}$ and $\tilde{A}_{22}$ in \eqref{eq:matrix-A-another-order-variable}, as long as each of them has real and distinct eigenvalues and they have no common eigenvalues. In this case, the real  diagonalizability of the matrix $\tilde{A}$ in \eqref{eq:matrix-A-another-order-variable} is guaranteed. Here, we can borrow the idea from \cite{koellermeier2020analysis} to modify the entries in the last row of the matrix such that the modified matrix has predefined propagation speeds. Since only the last row is modified, the associated polynomial sequence remains unchanged except for the last one. Therefore, it is easy to derive the analytical form of the characteristic polynomial where the coefficients have a linear dependence on the entries in the last row. Next, by matching the coefficients using Vieta's formulas which relate the coefficients of a polynomial to sums and products of its roots, or using the appropriate recurrence relation, the entries in the last row can be solved analytically. Similar ideas are also applied in machine learning moment closures for radiative transfer equation \cite{huang2023machine3} where the roots are represented by the neural networks.

After the modified $\tilde{A}_{11}$ and $\tilde{A}_{22}$ are obtained, the next step is to determine the form of $\tilde{A}_{21}$ by the rotational invariance constraint in Theorem \ref{thm:rotational-invariance-block-matrix}. The coefficient matrix $B$ in the $y$ direction can be derived by this constraint as well. Then we have a moment model in 2D with provable rotational invariance and hyperbolicity.

 {
We note that the corner case of the weakly hyperbolicity with $\alpha_1=0$ should be addressed separately in this framework. The algebraic multiplicities of the eigenvalue should be equal to the geometric multiplicities and this has to be checked by analyzing the eigenvectors of the coefficient matrix, as shown in section \ref{subsubsec:alpha1=0}. 
}

\subsection{An example of constructing a general closure}\label{subsec:example-general-closure}

To illustrate the framework in the previous part, we show an example by modifying the entries in the last row of $\tilde{A}_{22}$ in \eqref{eq:coefficient-matrix-HSWME-A22} such that its characteristic polynomial is related to the derivative of the Legendre polynomial.
Since only the last row is modified, the associated polynomial sequence given in Lemma \ref{lem:characteristic-polynomial-A-22} will not be changed except the last one. Therefore, we compute the last associated polynomial:
\begin{equation*}
\begin{aligned}
    q_{N+1}(x) ={}& xq_N(x) - \sum_{j=1}^{N+1}a_{N+1,j} q_{j-1}(x) \\
    ={}& (x-a_{N+1,N+1})q_N(x) - \sum_{j=1}^{N}a_{N+1,j} q_{j-1}(x) \\
    ={}& ((x- {u_m}) - (a_{N+1,N+1}- {u_m}))(2N+1)P_N(\xi) - \sum_{j=1}^{N}a_{N+1,j} (2j-1)P_{j-1}(\xi) \\
    ={}& \xi\alpha_1(2N+1)P_N(\xi) - (a_{N+1,N+1}- {u_m})(2N+1)P_N(\xi) - \sum_{j=1}^{N}a_{N+1,j} (2j-1)P_{j-1}(\xi) \\
    ={}& \alpha_1(N P_{N-1}(\xi) + (N+1) P_{N+1}(\xi)) - (a_{N+1,N+1}- {u_m})(2N+1)P_N(\xi) - \sum_{j=1}^{N}a_{N+1,j} (2j-1)P_{j-1}(\xi) \\
    ={}& \alpha_1(N+1)P_{N+1}(\xi) -(a_{N+1,N+1}- {u_m})(2N+1)P_N(\xi) + (\alpha_1N - a_{N+1,N}(2N-1)) P_{N-1}(\xi) \\
    & - \sum_{j=1}^{N-1}a_{N+1,j} (2j-1)P_{j-1}(\xi),
\end{aligned}
\end{equation*}
where we use the recursion relation $(2n+1) \xi P_n(\xi) = n P_{n-1}(\xi) + (n+1) P_{n+1}(\xi)$.

Next, we would like to take appropriate values of $a_{N+1,j}$ for $1\le j\le N+1$ such that $q_{N+1}(x)$ is proporional to $P_{N+2}'(\xi)$. We use the relation
\begin{equation}
    P_{N+2}'(\xi) = (2N+3)P_{N+1}(\xi) + (2N-1)P_{N-1}(\xi) + (2N-5)P_{N-3}(\xi) + \cdots.
\end{equation}
Matching the coefficient of $P_{N+1}(\xi)$ in $P'_{N+2}(\xi)$ and $q_{N+1}(x)$, we have 
\begin{equation*}
    q_{N+1}(x) = \alpha_1 \frac{N+1}{2N+3} P_{N+2}'(\xi),
\end{equation*}
which can be expanded as
\begin{equation*}
\begin{aligned}
    & \alpha_1(N+1)P_{N+1}(\xi) -(a_{N+1,N+1}-u)(2N+1)P_N(\xi) + (\alpha_1N - a_{N+1,N}(2N-1)) P_{N-1}(\xi) \\
    & - \sum_{j=1}^{N-1}a_{N+1,j} (2j-1)P_{j-1}(\xi) \\
    ={}& \alpha_1(N+1)P_{N+1}(\xi) + \alpha_1\frac{N+1}{2N+3}(2N-1)P_{N-1}(\xi) + \alpha_1\frac{N+1}{2N+3}(2N-5)P_{N-3}(\xi) + \cdots.
\end{aligned}    
\end{equation*}

Now we match the remaining coefficients in the above equation. For the coefficient of $P_N(\xi)$, we have
\begin{equation*}
    a_{N+1,N+1} =  {u_m}.
\end{equation*}
For the coefficient of $P_{N-1}(\xi)$, we have
\begin{equation*}
    \alpha_1 N - a_{N+1,N}(2N-1) = \alpha_1 \frac{N+1}{2N+3} (2N-1),
\end{equation*}
from which we solve for $a_{N+1,N}$ and obtain
\begin{equation*}
    a_{N+1,N} = \frac{2N+1}{(2N-1)(2N+3)}\alpha_1.
\end{equation*}
For the remaining entries, it is easy to obtain:
\begin{equation}
    a_{N+1,j} = \left \{
    \begin{aligned}        
    & -\frac{N+1}{2N+3}\alpha_1, \quad & \textrm{if} \;\; j\equiv N \;\mathrm{mod}\; 2, \\
    & 0, \quad &\textrm{otherwise}.
    \end{aligned}
    \right. 
\end{equation}
for $1 \le j \le N-1$.

Therefore, the modified $\tilde{A}_{22}$ is
\begin{equation}\label{eq:coefficient-matrix-HSWME-modify-A22}
    \tilde{A}_{22}(W) =
    \begin{pmatrix}
         {u_m} & \frac{\alpha_1}{3} \\
        \alpha_1 &  {u_m} & \frac{2\alpha_1}{5} \\
        & \frac{2}{3}\alpha_1 &  {u_m} & \frac{3\alpha_1}{7} \\
        & & & & \ddots & \ddots & \ddots \\
        & & & & & \frac{N-1}{2N-3}\alpha_1 &  {u_m} & \frac{N}{2N+1}\alpha_1 \\
        \cdots & 0 & -\frac{N+1}{2N+3}\alpha_1 & 0 & -\frac{N+1}{2N+3}\alpha_1 & 0 & \frac{2N+1}{(2N-1)(2N+3)}\alpha_1 &  {u_m}
    \end{pmatrix}.
\end{equation}

Next, we keep $\tilde{A}_{11}$ in \eqref{eq:coefficient-matrix-HSWME-A11} unchanged and write down $\tilde{A}_{21}$ based on the rotational invariance constraint given in Theorem \ref{thm:rotational-invariance-block-matrix}.
We can also write the coefficient matrices in the original order of variables:
\begin{equation}\label{eq:coefficient-matrix-A-general-closure}
    A = 
    \begin{aligned}
        \begin{pmatrix}
        0 & 1 & 0 \\
        - {u_m}^2 - \frac{\alpha_1^2}{3} + gh & 2 {u_m} & 0 & \frac{2\alpha_1}{3} & 0 \\
        - {u_m} {v_m} - \frac{\alpha_1\beta_1}{3} & v &  {u_m} & \frac{\beta_1}{3} & \frac{\alpha_1}{3} \\
        -2 {u_m}\alpha_1 & 2\alpha_1 & 0 &  {u_m} & 0 & \frac{3\alpha_1}{5} & 0 & \cdots & 0 & 0 \\
        -( {u_m}\beta_1+v\alpha_1) & \beta_1 & \alpha_1 & 0 &  {u_m} & \frac{\beta_1}{5} & \frac{2\alpha_1}{5} & \cdots & 0 & 0 \\
        - \frac{2}{3}\alpha_1^2 & 0 & 0 & \frac{\alpha_1}{3} & 0 &  {u_m} & 0 &  \cdots & 0 & 0 \\
        - \frac{2}{3}\alpha_1\beta_1 & 0 & 0 & -\frac{\beta_1}{3} & \frac{2\alpha_1}{3} & 0 &  {u_m} & \cdots & 0 & 0 \\
         &  &  &  &  & \ddots & \ddots & \ddots & \frac{(N+1)\alpha_1}{2N+1} & 0 \\
         &  &  &  &  & \ddots & \ddots & \ddots & \frac{\beta_1}{2N+1} & \frac{N\alpha_1}{2N+1} \\
         \cdots & \cdots & 0 & 0 & 0 & 0 & \frac{(N-1)\alpha_1}{2N-1} & 0 &  {u_m} & 0 \\
         \cdots & \cdots & \frac{(N+1)\beta_1}{2N+3} & -\frac{(N+1)\alpha_1}{2N+3} & 0 & 0 & \frac{(2N^2-N-4)\beta_1}{(2N-1)(N+1)} & \frac{(2N+1)\alpha_1}{(2N-1)(2N+3)} & 0 &  {u_m}
        \end{pmatrix}
    \end{aligned}   
\end{equation}
in the $x$ direction
and
\begin{equation}\label{eq:coefficient-matrix-B-general-closure}
    B = 
    \begin{aligned}
        \begin{pmatrix}
        0 & 0 & 1 \\
        - {u_m} {v_m} - \frac{\alpha_1\beta_1}{3} &  {v_m} &  {u_m} & \frac{\beta_1}{3} & \frac{\alpha_1}{3} \\
        - {v_m}^2 - \frac{\beta_1^2}{3} + gh & 0 & 2 {v_m} & 0 & \frac{2\beta_1}{3} \\
        -( {u_m}\beta_1+ {v_m}\alpha_1) & \beta_1 & \alpha_1 &  {v_m} & 0 & \frac{2\beta_1}{5} & \frac{\alpha_1}{5} \\
        -2 {v_m}\beta_1 & 0 & 2\beta_1 & 0 &  {v_m} & 0 & \frac{3\beta_1}{5} \\
        - \frac{2}{3}\alpha_1\beta_1 & 0 & 0 & \frac{2\beta_1}{3} & -\frac{\alpha_1}{3} &  {v_m} & 0 & \cdots & 0 & 0 \\
        - \frac{2}{3}\beta_1^2 & 0 & 0 & 0 & \frac{\beta_1}{3} & 0 &  {v_m} &  \cdots & 0 & 0 \\
         &  &  &  &  & \ddots & \ddots & \ddots & \frac{N}{2N+1}\beta_1 & \frac{1}{2N+1}\alpha_1 \\
         &  &  &  &  & \ddots & \ddots & \ddots & 0 & \frac{N+1}{2N+1}\beta_1 \\
         \cdots & \cdots & -\frac{(N+1)\beta_1}{2N+3} & \frac{(N+1)\alpha_1}{2N+3} & 0 & 0 & \frac{(2N+1)\beta_1}{(2N-1)(2N+3)} & \frac{(2N^2-N-4)\alpha_1}{(2N-1)(N+1)} &  {v_m} & 0 \\
         \cdots & \cdots & 0 & 0 & 0 & 0 & 0 & \frac{(N-1)\beta_1}{2N-1} & 0 &  {v_m}
        \end{pmatrix}
    \end{aligned}   
\end{equation}
in the $y$ direction.

Since $P_{N+1}'(\xi)$ in Corollary \ref{cor:characteristic-polynomial-A-11} and $P_{N+2}'(\xi)$ both have real and distinct roots and there are no common roots in the case of $\alpha_1\ne0$, we have that the coefficient matrix \eqref{eq:coefficient-matrix-A-general-closure} in the $x$ direction has real and distinct roots and thus we have that the matrix is real diagonalizable. Moreover, the analysis of the model in the case of $\alpha_1=0$ is similar to the one in section \ref{subsubsec:alpha1=0}. Combining with the rotational invariance in Theorem \ref{thm:general-closure-relation-rotational-invariance}, we have the following result: 
\begin{thm}[hyperbolicity of an  {examplore} general moment closure]
    The model in 2D with the coefficient matrices given by \eqref{eq:coefficient-matrix-A-general-closure} in $x$ direction and \eqref{eq:coefficient-matrix-B-general-closure} in $y$ direction is hyperbolic if $\alpha_1\ne0$ or $\alpha_1=\beta_1=0$ and weakly hyperbolic if $\alpha_1=0$ and $\beta_1\ne0$.
\end{thm}

 {
\section{Numerical simulations}\label{sec:numerical}

In this section, we present numerical simulations to compare the performance of different models. Here, we consider the following models:
\begin{enumerate}
    \item the classical shallow water equations (SWE);
    \item the hyperbolic shallow water moment equation (HSWME) in \eqref{eq:HSWME} with the coefficient matrices given in Lemma \ref{lem:coefficient-matrices-HSWME};
    \item the $\beta$-HSWME model with the coefficient matrices given in \eqref{eq:coefficient-matrix-beta-x} and \eqref{eq:coefficient-matrix-beta-y};
    \item our new example HSWME model with the coefficient matrices given by \eqref{eq:coefficient-matrix-A-general-closure} and \eqref{eq:coefficient-matrix-B-general-closure}
    \item our new globally HSWME model with the coefficient matrices given by \eqref{eq:coefficient-matrix-HSWME-globally-hyperbolic-x} and \eqref{eq:coefficient-matrix-HSWME-globally-hyperbolic-y}.
\end{enumerate}
Note that {the four types} of SWME models with the order of moment $N=1$ are the same.

\begin{exam}[radial dam break problem]
{We consider the radial dam break problem, see e.g. \cite{verbiest2022,verbiest2023hyperbolic}.} The computational domain is $\Omega = [0,1]^2$. We choose the initial condition for the water height:
\begin{equation}
    h(x,y,0) = \left\{
    \begin{aligned}
    1.5, & \quad \sqrt{(x-0.5)^2 + (y-0.5)^2} \le 0.2,\\
    1, & \quad \textrm{otherwise}.\\
    \end{aligned}
    \right.
\end{equation}
The initial velocity profiles are linear functions in vertical direction:
\begin{equation}
    u(x,y,\zeta,0) = v(x,y,\zeta,0) = \frac{1}{4}\zeta.
\end{equation}
The friction coefficient $\lambda=0.1$ and the slip length $\nu=0.1$. The gravitational constant is set to $g=1$. 
\end{exam}

We use the first-order Roe method with the local Lax-Friedrichs approximation to the absolute value of the intermediate matrix (see the details in section 3.4 in \cite{castro2017well}). The Euler forward method is applied for the time discretization. The spatial resolutions are $N_x = N_y = 1024$. The CFL number is taken to be 1. The simulation is run up to $t=1$.

In Figure \ref{fig:dam-break-t0d1}, we show the profiles of the water height $h$, the average velocity in $x$-direction $u_m$, and the moments $\alpha_1$ and $\alpha_2$ at $t=0.1$ along $y=0.5$. Here, we only show the profiles for $0\le x\le 0.5$ due to symmetry. From Figure \ref{fig:dam-break-t0d1}(a), we can see all the models give similar results for the water height $h$. From Figure \ref{fig:dam-break-t0d1}(b), {the four types} of SWME models give similar results in the average velocity $u_m$ while the SWE model gives a lower velocity. For the higher order moment, the HSMWE model with $N=1$ gives different prediction on $\alpha_1$ near the extreme points comparing to other types of SWME with $N=2$ in Figure \ref{fig:dam-break-t0d1}(c). The $\beta$-HSMWE model with $N=2$ gives different profiles on $\alpha_2$ near the extreme points from the other models as shown in  Figure \ref{fig:dam-break-t0d1}(d).

{To show the rotational invariance of the model, we plot the solution of the water height $h$ in 1D cut along different angles ($y=0.5$, $x=0.5$, and $y=x$), and the surface plot in Figure \ref{fig:dam-break-t0d1-rotation}. This demonstrates Theorem \ref{thm:rotation-invariance-HSWME} that the model is rotational invariant.
}

\begin{exam}[smooth wave problem]
{We consider the smooth wave problem, see e.g. \cite{kowalski2017moment}. The computational domain is $\Omega = [0,1]^2$. We choose the initial condition for the water height:
\begin{equation}
    h(x,y,0) = 1 + \exp(3\cos(2\pi(x+y+0.5))-4).
\end{equation}
The initial velocity profiles are linear functions in the vertical direction:
\begin{equation}
    u(x,y,\zeta,0) = v(x,y,\zeta,0) = \frac{1}{4}\zeta.
\end{equation}
The friction coefficient $\lambda=0.1$ and the slip length $\nu=0.1$. The gravitational constant is set to $g=1$. The periodic boundary conditions are applied in both the $x$ and $y$ directions.}
\end{exam}

{In Figure \ref{fig:smooth-wave-t1}, we compare the profiles of the water height is $h$, the average velocity in the $x$-direction is $u_m$, and the moments are $\alpha_1$ and $\alpha_2$ at $t=1$ along $y=0.5$. From Figure \ref{fig:smooth-wave-t1}(a), we see that the SWE, the HSWME with $N=1$, and the other four SWME with $N=2$ give different profiles of the  water height. Notably, all SWME models show the same wave propagation speeds in the simulation results, while the SWE model gives the speeds lacking behind, which is consistent with the eigenvalue analysis. From Figure \ref{fig:smooth-wave-t1}(b), the SWME models with $N=2$ all give a similar average velocity profiles while the SWME model with $N=1$ and the SWE model gives a smaller velocity. For the higher order moments, we observe that the SWME models with $N=2$ predict different profiles of $\alpha_1$ from the SWME model with $N=1$, while the profiles of the higher order moments $\alpha_2$ are almost the same for different SWME models.}

In summary, {the four types} of SWME models give very close predictions for the water height $h$ and the average velocity $u_m$ and slightly different predictions for the higher order moments. The similar numerical results are also observed in the 1D models in \cite{koellermeier2020analysis} for the HSWME model and the $\beta$-HSWME model.

Moreover, the weakly hyperbolicity for partial state does not cause any trouble in the simulation. The numerical results are stable and the solution profiles are smooth. It remains to be the future work to investigate the performance of the HSWME model in more complicated problems and compare with the solution of the full model of the incompressible NS equations.
\begin{figure}
    \centering
    \subfigure[the water height $h$]
    {
    \begin{minipage}[b]{0.465\textwidth}
    \includegraphics[width=1\textwidth]{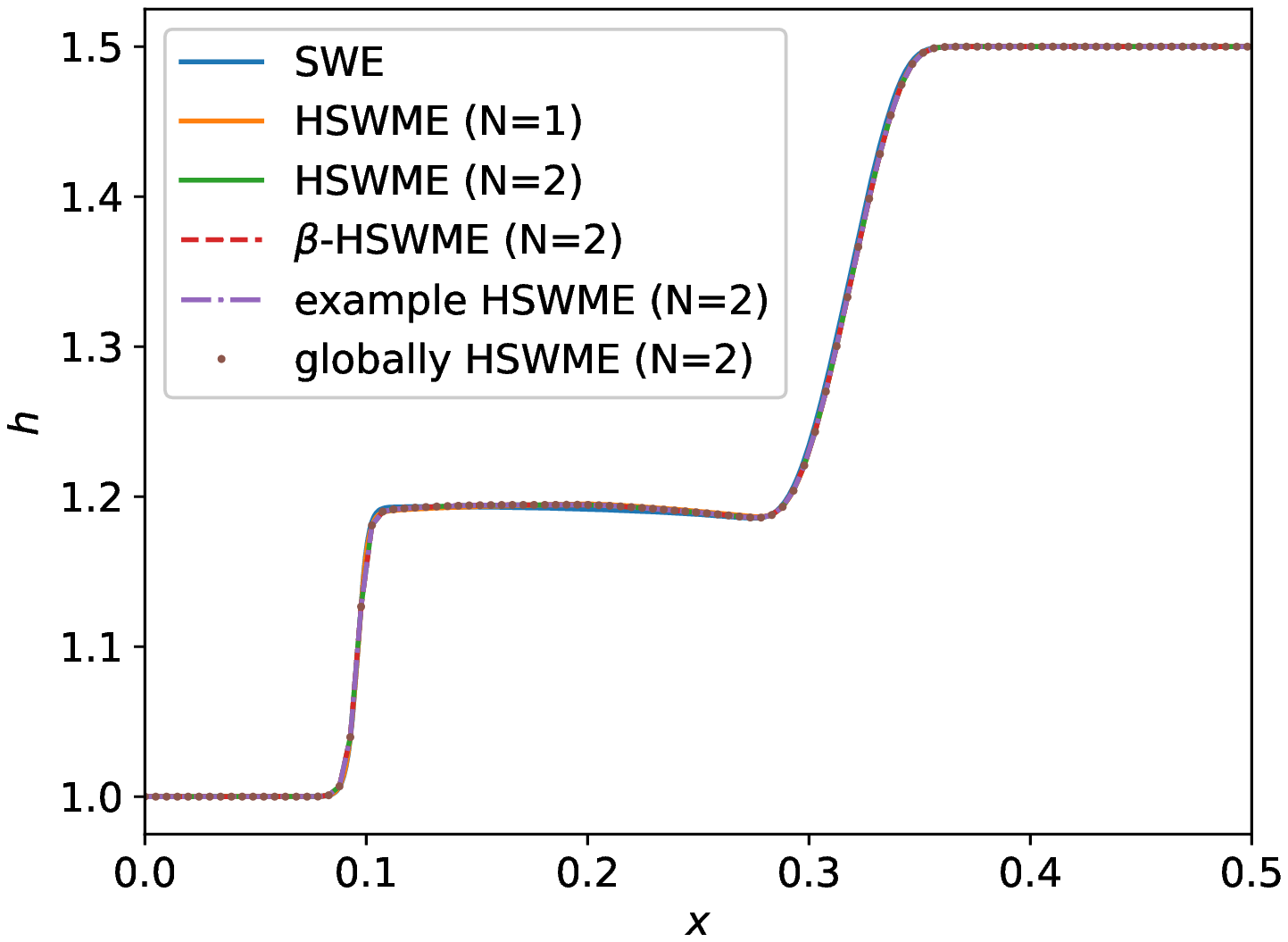}
    \end{minipage}
    }
    \subfigure[the average velocity in $x$-direction $u_m$]
    {
    \begin{minipage}[b]{0.46\textwidth}    
    \includegraphics[width=1\textwidth]{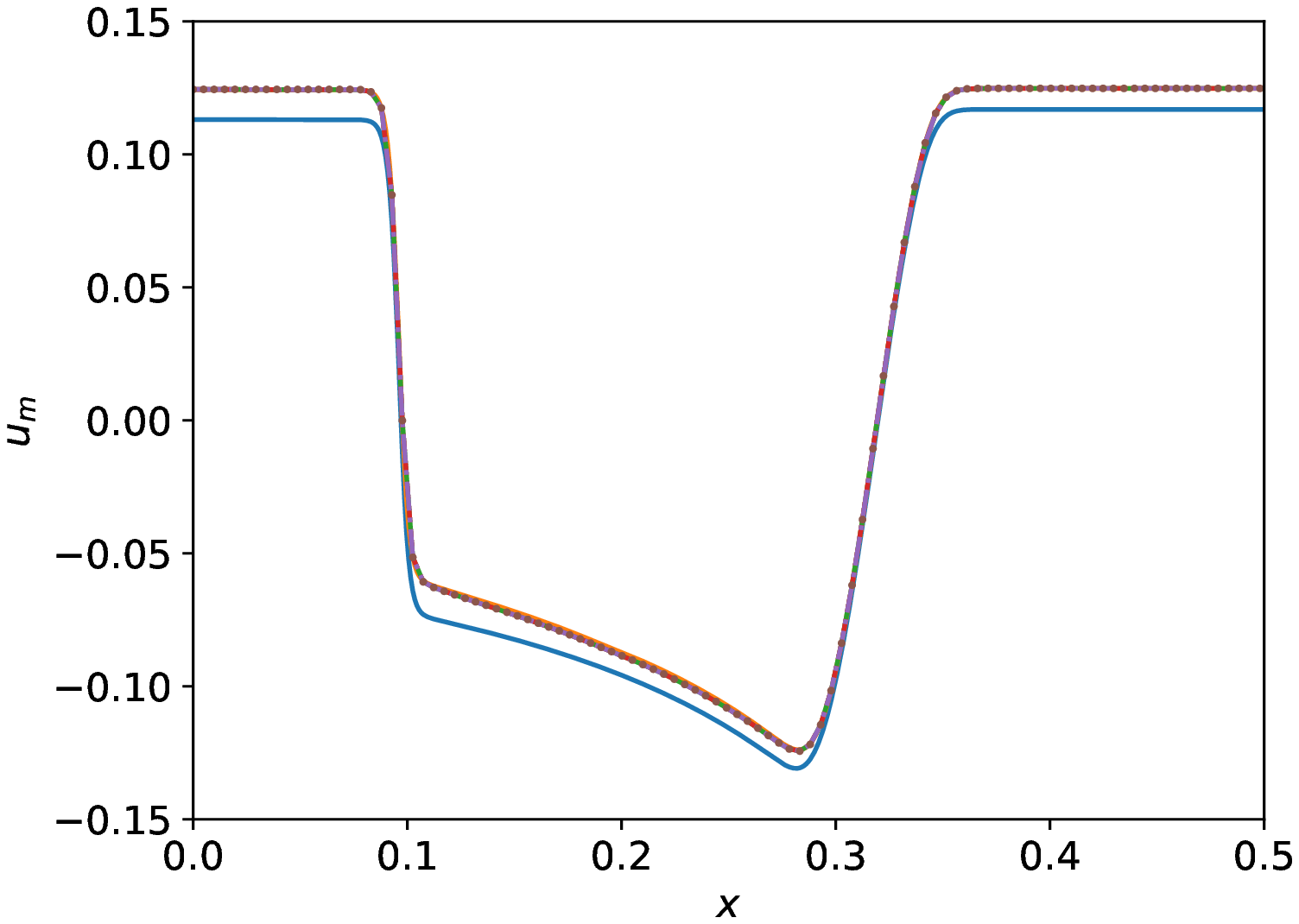}
    \end{minipage}
    }
    \bigskip
    \subfigure[the moment $\alpha_1$ in $x$-direction]{
    \begin{minipage}[b]{0.46\textwidth}
    \includegraphics[width=1\textwidth]{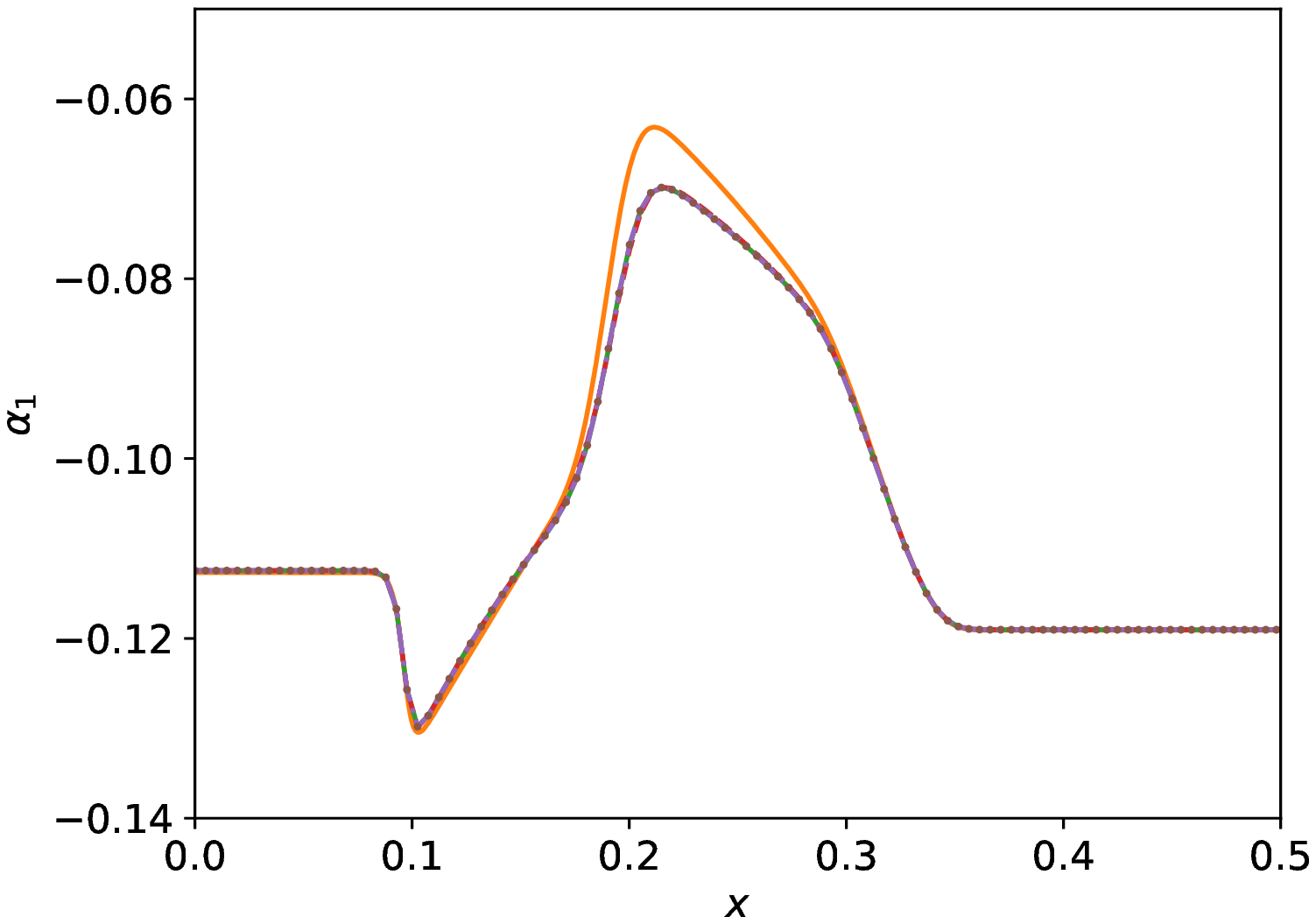}
    \end{minipage}
    }
    \subfigure[the moment $\alpha_2$ in $x$-direction]{
    \begin{minipage}[b]{0.46\textwidth}    
    \includegraphics[width=1\textwidth]{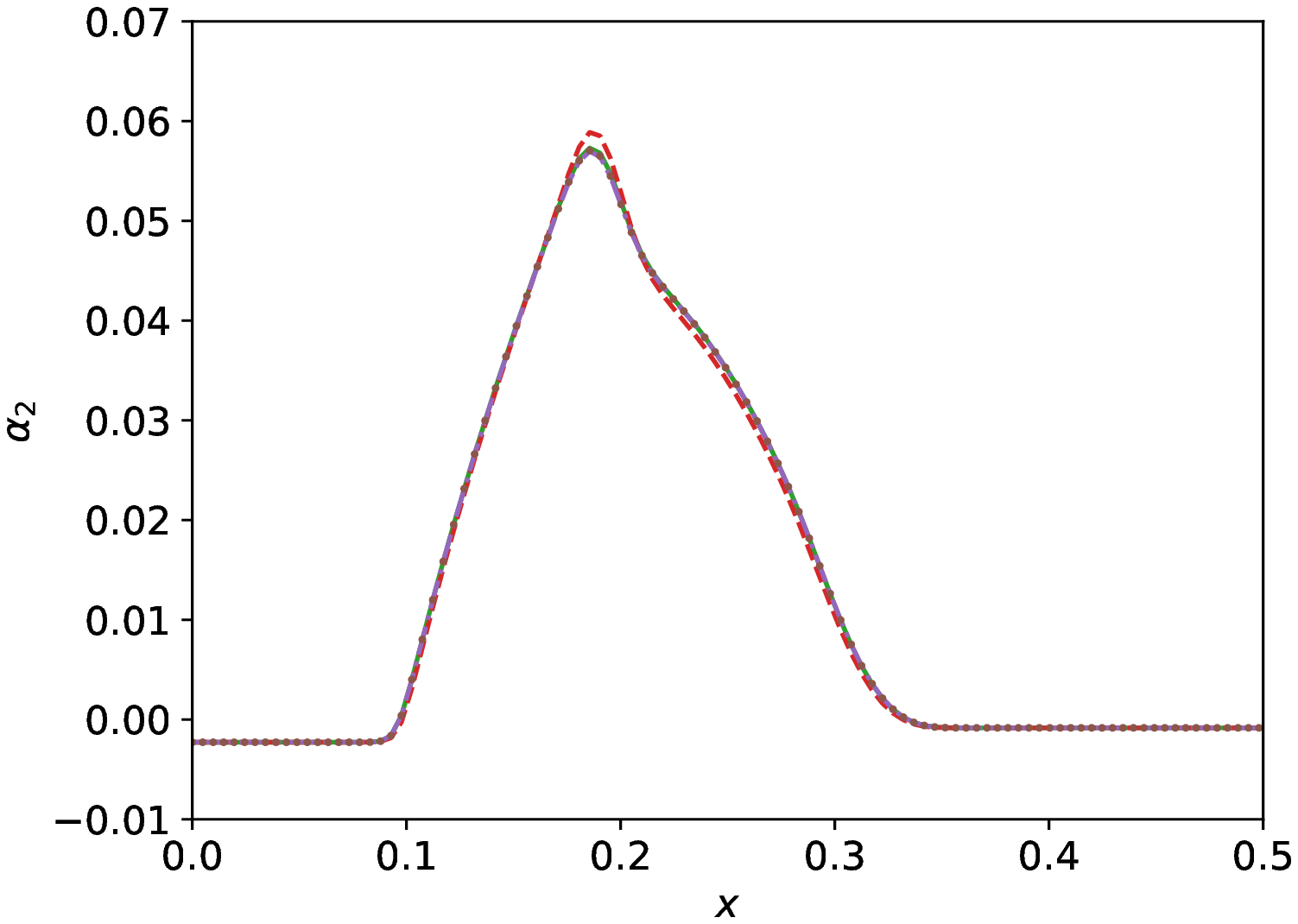}
    \end{minipage}
    }        
    \caption{Simulation of a radial dam break at $t=0.1$. (1) Blue solid line denotes the shallow water equation (SWE); (2) orange solid line denotes the hyperbolic shallow water moment equation (HSWME) with $N=1$; (3) green solid line denotes the HSWME with $N=2$; (4) red dashed line denotes the $\beta$-HSWME with $N=2$; (5) purple dashed dot line denotes the example HSWME with $N=2$; (6)  brown dot line denotes the globally HSWME with $N=2$.}
    \label{fig:dam-break-t0d1}
\end{figure}

\begin{figure}
    \centering
    \subfigure[the water height in 1D cut along different angles]    
    {
    \begin{minipage}[b]{0.45\textwidth}
    \includegraphics[width=1\textwidth]{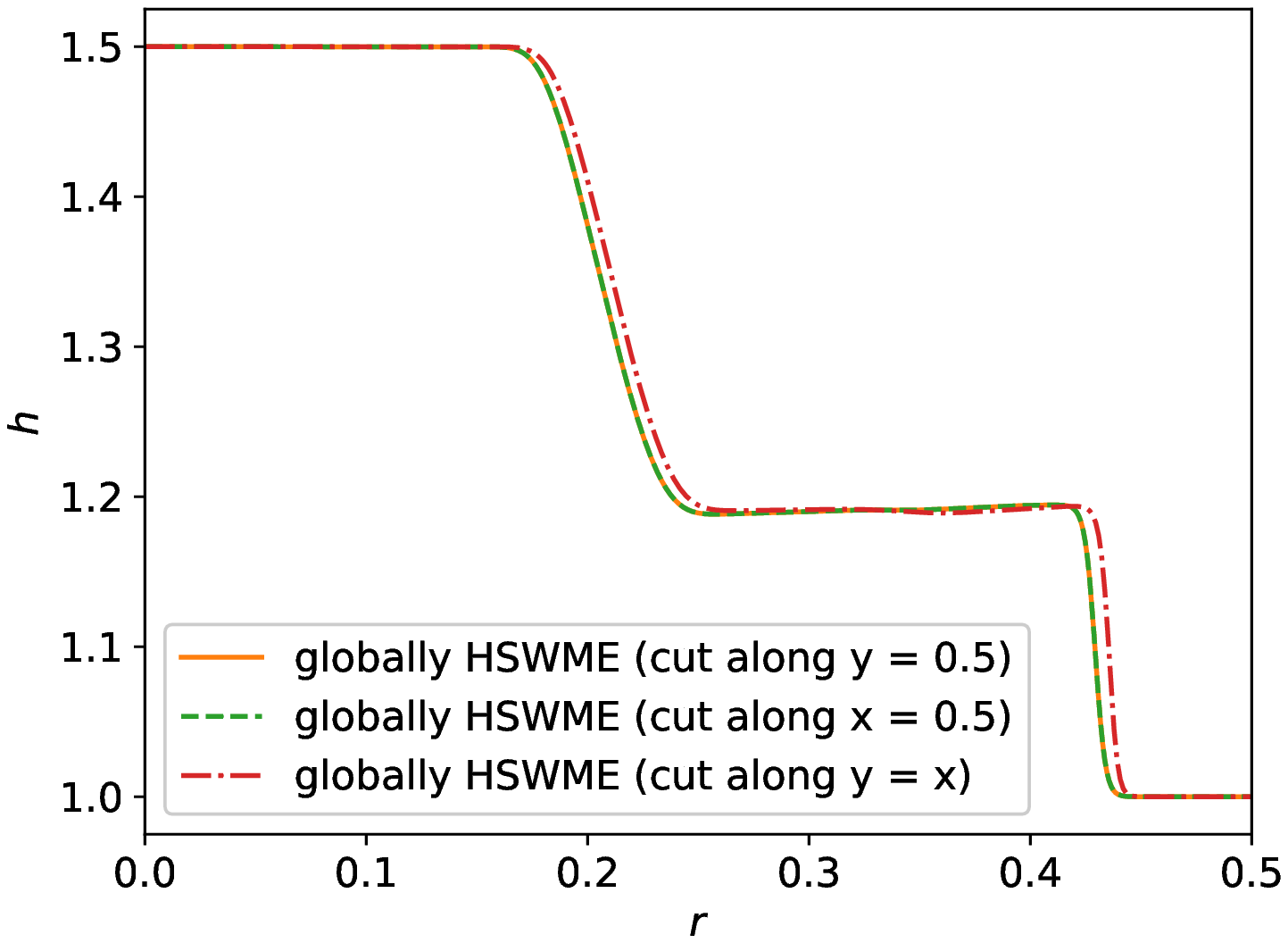}
    \end{minipage}
    }
    \subfigure[the water height in 3D plot]
    {
    \begin{minipage}[b]{0.45\textwidth}
    \includegraphics[width=1\textwidth]{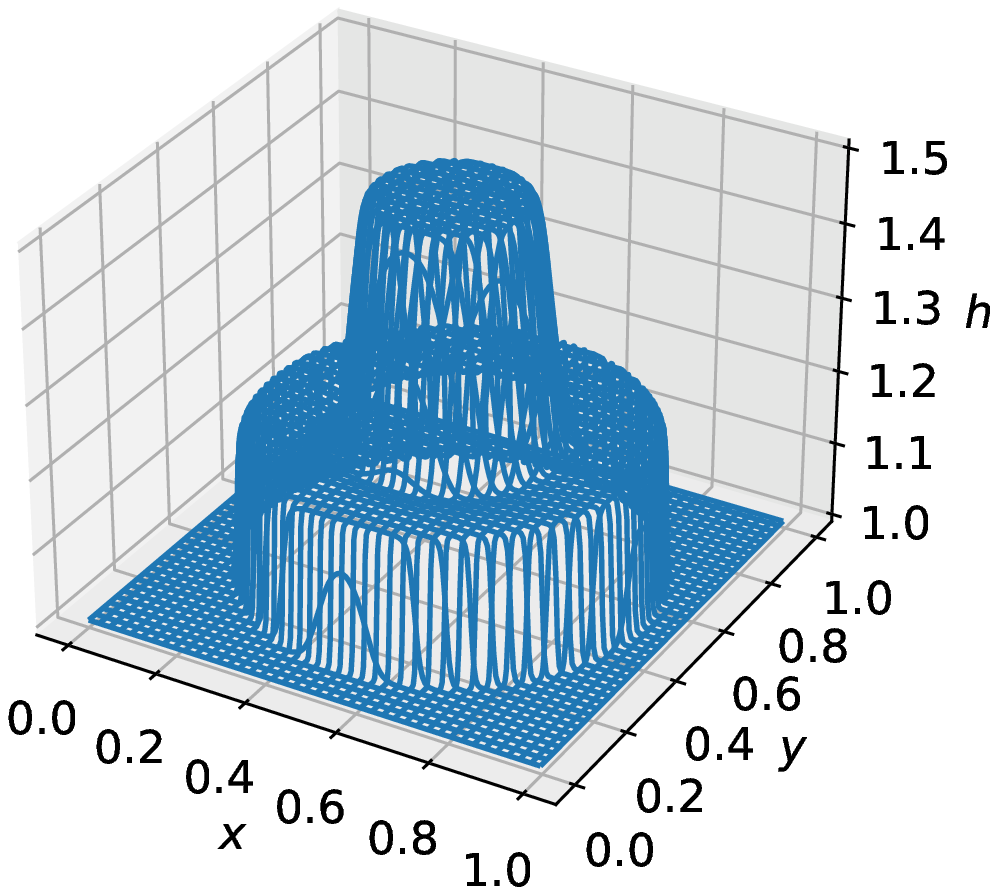}
    \end{minipage}
    }
    \caption{{Simulation of a radial dam break at $t=0.1$ using globally HSWME. Left: the water height in 1D cut along $y=0.5$, $x=0.5$, and $y=x$. Here $r=\sqrt{(x-0.5)^2 + (y-0.5)^2}$. Right: the water height in 3D plot.}}
    \label{fig:dam-break-t0d1-rotation}
\end{figure}

\begin{figure}
    \centering
    \subfigure[the water height $h$]
    {
    \begin{minipage}[b]{0.465\textwidth}
    \includegraphics[width=1\textwidth]{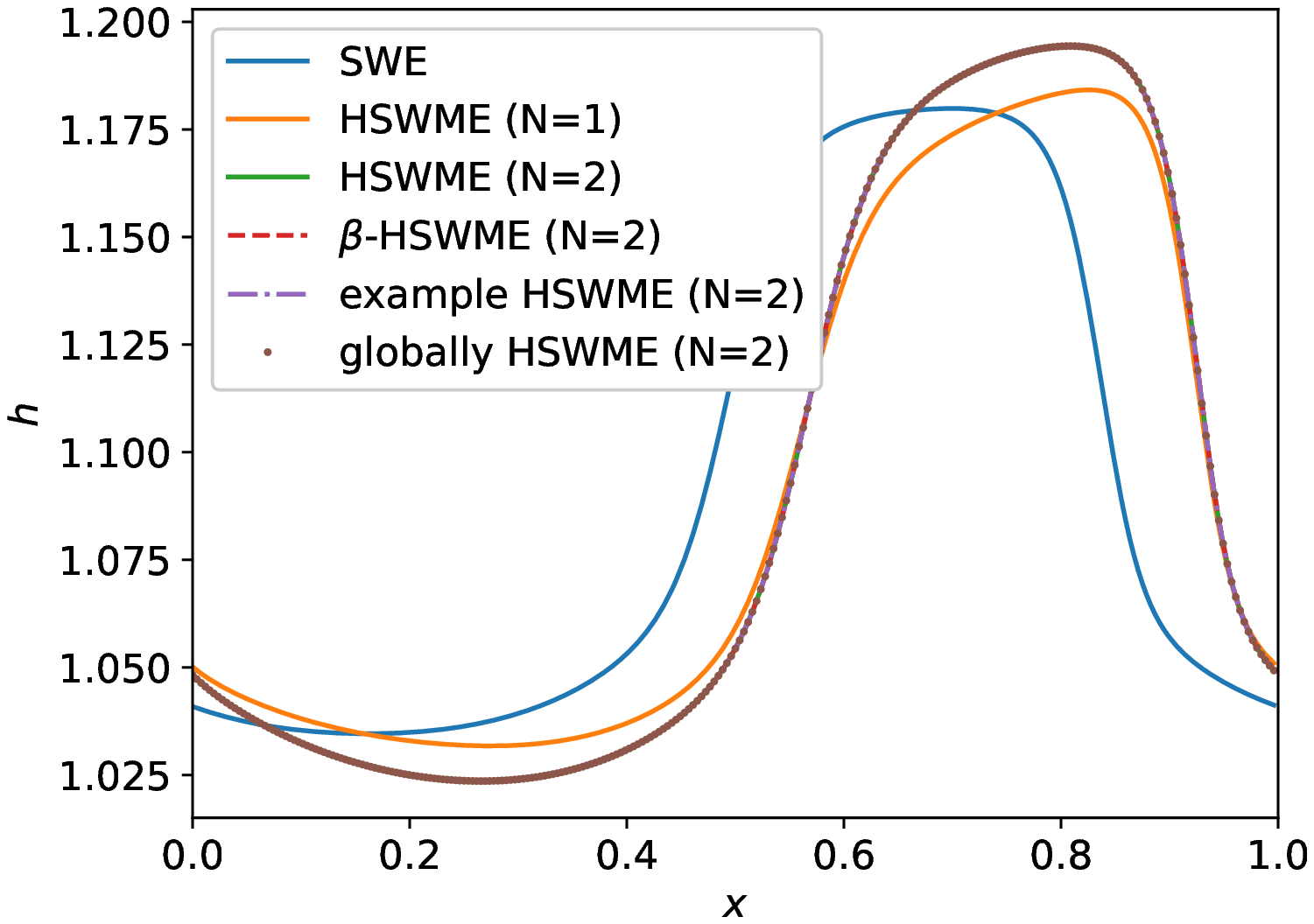}
    \end{minipage}
    }
    \subfigure[the average velocity in $x$-direction $u_m$]
    {
    \begin{minipage}[b]{0.46\textwidth}    
    \includegraphics[width=1\textwidth]{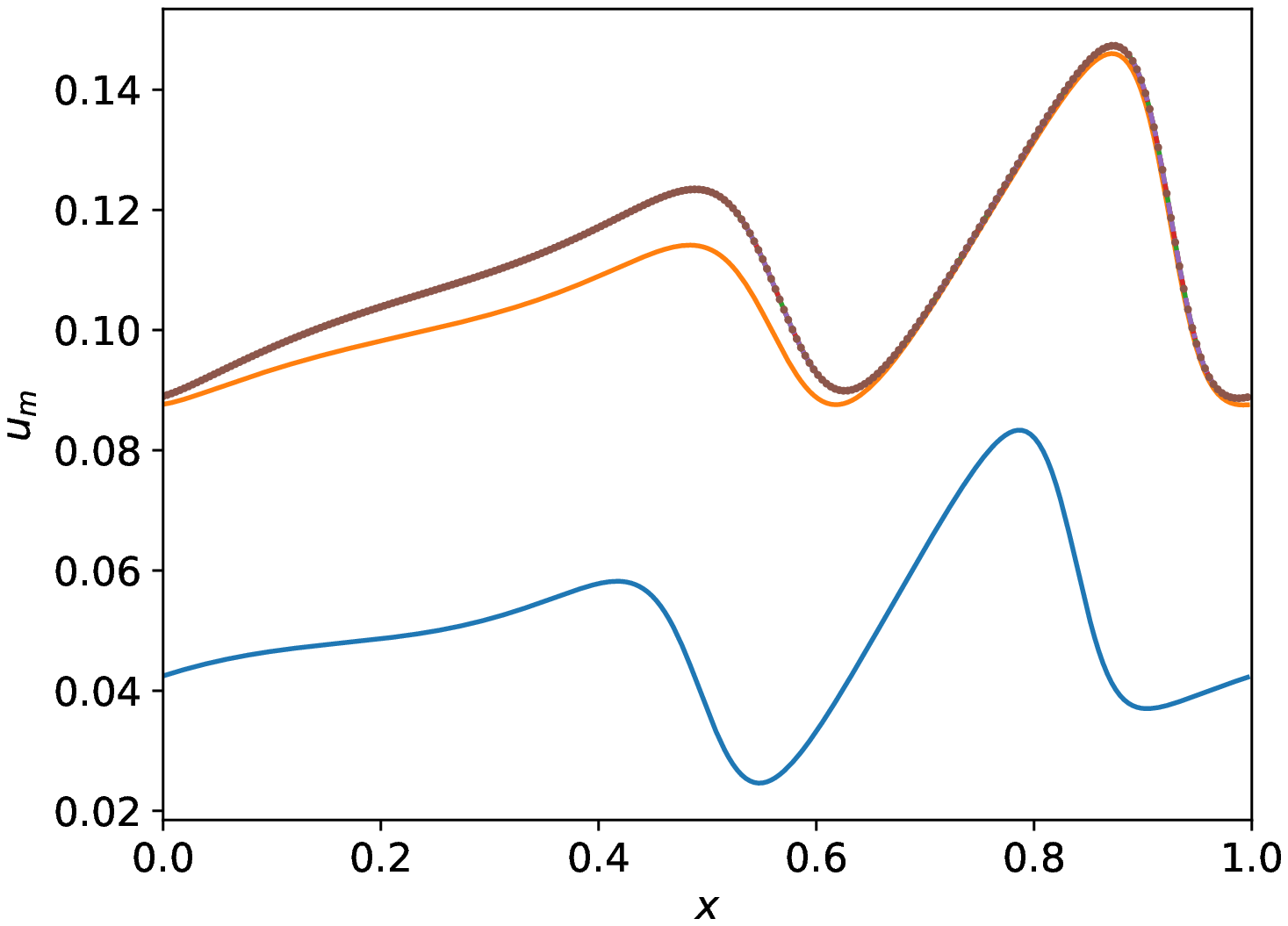}
    \end{minipage}
    }
    \bigskip
    \subfigure[the moment $\alpha_1$ in $x$-direction]{
    \begin{minipage}[b]{0.46\textwidth}
    \includegraphics[width=1\textwidth]{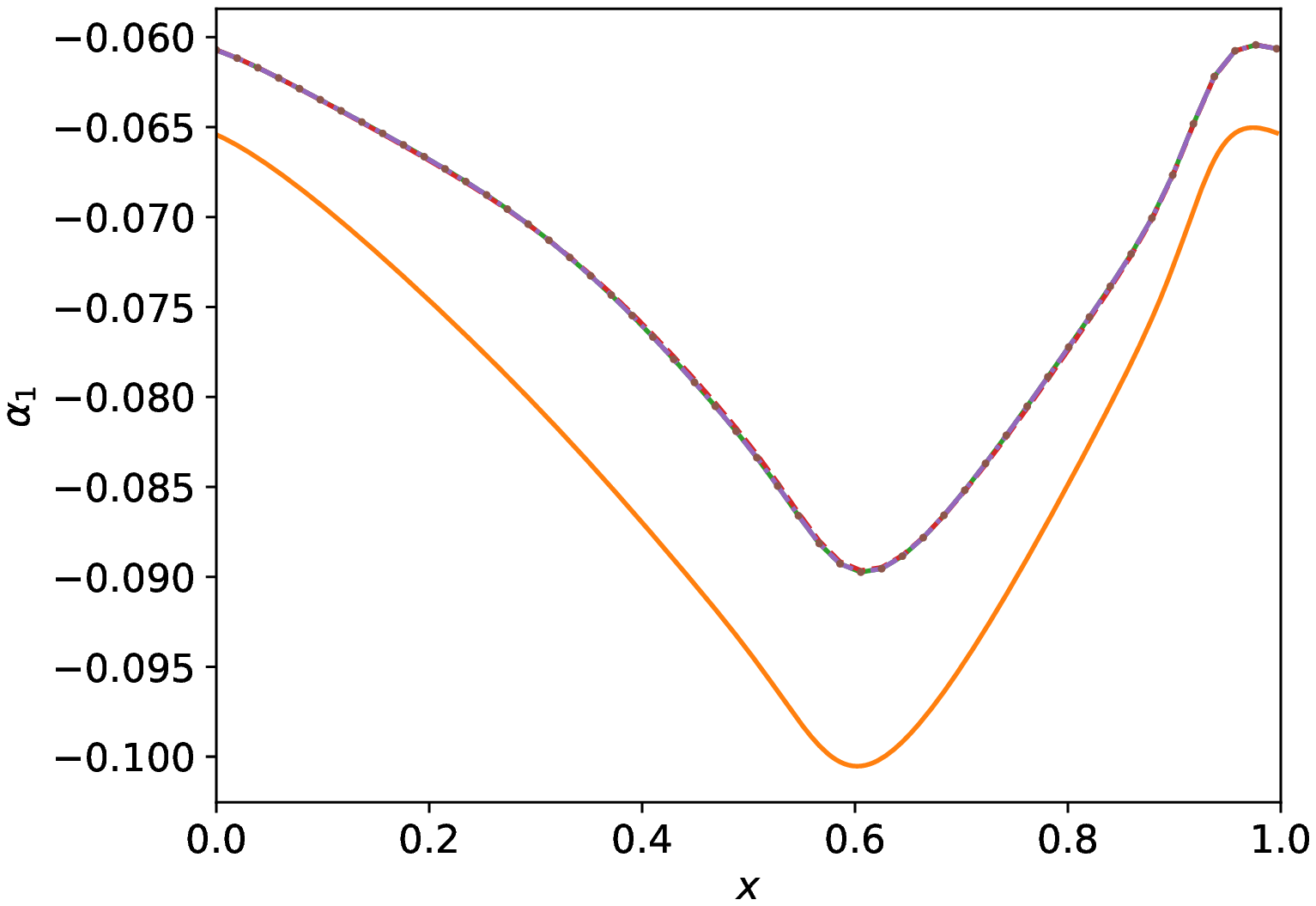}
    \end{minipage}
    }
    \subfigure[the moment $\alpha_2$ in $x$-direction]{
    \begin{minipage}[b]{0.46\textwidth}    
    \includegraphics[width=1\textwidth]{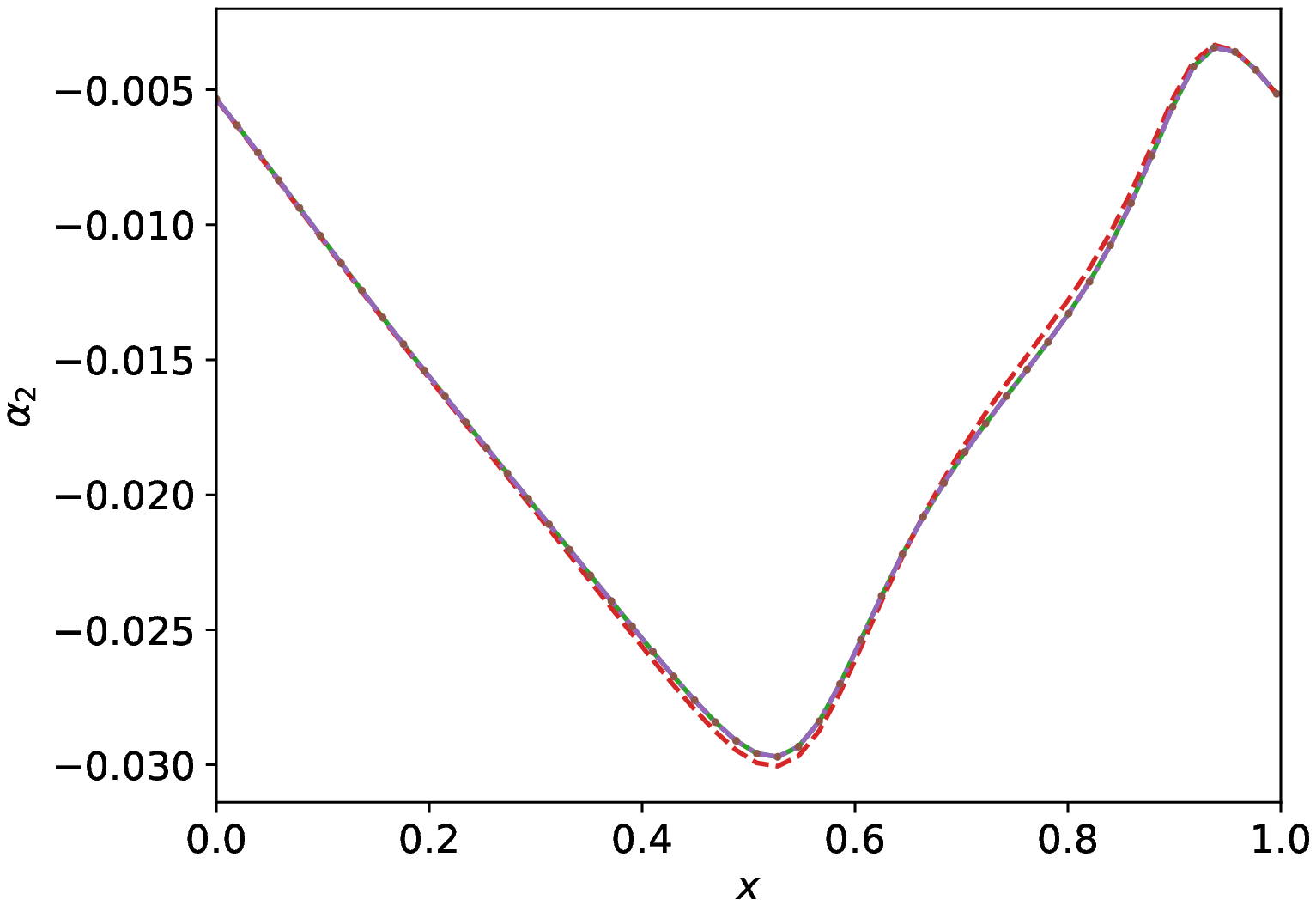}
    \end{minipage}
    }        
    \caption{{Simulation of the smooth wave at $t=1$. (1) Blue solid line denotes the shallow water equation (SWE); (2) orange solid line denotes the hyperbolic shallow water moment equation (HSWME) with $N=1$; (3) green solid line denotes the HSWME with $N=2$; (4) red dashed line denotes the $\beta$-HSWME with $N=2$; (5) purple dashed dot line denotes the example HSWME with $N=2$; (6)  brown dot line denotes the globally HSWME with $N=2$.}}
    \label{fig:smooth-wave-t1}
\end{figure}

}

\section{Concluding remarks}\label{sec:conclusion}

In this work, we investigate the rotational invariance and hyperbolicity of the shallow water moment equations in 2D.  {We find that the moment model in 2D is hyperbolic in most cases and weakly hyperbolic in a degenerate edge case. With a simple modification to the coefficient matrices, we fix this weakly hyperbolicity and propose a new global hyperbolic model.} We establish a general closure relation such that the resulting shallow water moment equations are hyperbolic and rotationally invariant. Moreover, we propose a new class of shallow water moment equations, which is a generalization of the HSWME model. {We conduct numerical simulations for these SWME models and compare with the SWE model. It shows that the model generates different solution profiles with different numbers of moments.} {We also mention that the new class of shallow water moment equations is only proven to be hyperbolic up to $N = 100$. Proving that the model is hyperbolic for an arbitrary number of moments remains an open question, i.e., $N>100$.}

In future work, it remains to be seen whether the derived shallow water moment models are good approximations of the original incompressible NS equations both numerically and theoretically. 
Another interesting direction is to apply data-driven methods to learn the closure relations while preserving the rotational invariance and hyperbolicity, as we have done in \cite{huang2020learning,huang2021gradient,huang2023machine2,huang2023machine3}.
{In addition, the analysis in the current paper could not be easily generalized to the Shallow Water Linearized Moment Equations in \cite{koellermeier2022steady} due to the different structure of the Jacobian matrix of the model.}
These topics are the subject of our ongoing work.

\section*{Acknowledgments}

The team would like to thank the anonymous reviewers for their invaluable comments and questions which improved the quality of this paper.
The team would like to thank Michigan State University and Texas Tech University for the support of this work, through the ongoing support of the PIs and students.  We would further like to thank Keith Promislow for his helpful discussions and insight during the development of this effort.  Finally, PI Christlieb would like to thank the Office of Naval Research for the support of this work under grant number N00014-18-1-2552.

\appendix
\section{Proof of Proposition \ref{prop:rotational-invariance-system-unchanged}}\label{app:proof-rotation-invariance-system-unchanged} 

We first introduce the following lemma, which will be used in the proof of Proposition \ref{prop:rotational-invariance-system-unchanged}.
\begin{lem}\label{lem:rotation-invariance-property-2}
    If the matrices $A(U)$ and $B(U)$ satisfy the rotational invariance property
    \begin{equation}\label{eq:rotation-invariance-property-1}
        \cos\theta \, A(U) + \sin\theta \, B(U) = T^{-1} A(T U) T, \quad \forall \,\, 0\le \theta < 2\pi,
    \end{equation}    
    then they also satisfy
    \begin{equation}\label{eq:rotation-invariance-property-2}
        -\sin\theta \, A(U) + \cos\theta \, B(U) = T^{-1} B(T U) T, \quad \forall \,\, 0\le \theta < 2\pi.
    \end{equation}    
    Here $T = T(\theta)$ is the rotation matrix defined in \eqref{eq:rotation-matrix-T}.
\end{lem}
\begin{proof}
    Taking $\theta$ to be $-\theta$ in \eqref{eq:rotation-invariance-property-1}, we have
    \begin{equation}\label{eq:rotation-invariance-property-1-neg}
        \cos\theta \, A(U) - \sin\theta \, B(U) = T(\theta) A(T^{-1}(\theta) U) T^{-1}(\theta),
    \end{equation}
    where we use the fact that $T^{-1}(\theta) = T(-\theta)$. 
    Taking $U$ to be $T(\theta)U$ in \eqref{eq:rotation-invariance-property-1-neg}, we have
    \begin{equation}\label{eq:rotation-invariance-property-1-neg-2}
        \cos\theta \, A(T(\theta)U) - \sin\theta \, B(T(\theta)U) = T(\theta) A(U) T^{-1}(\theta).
    \end{equation}
    Multiplying both sides of \eqref{eq:rotation-invariance-property-1-neg-2} by $T^{-1}(\theta)$ from the left and by $T(\theta)$ from the right, we obtain
    \begin{equation}\label{eq:rotation-invariance-property-1-neg-3}
        \cos\theta \, T^{-1}(\theta) \, A(T(\theta)U)  \, T(\theta) - \sin\theta \, T^{-1}(\theta) \, B(T(\theta)U) \, T(\theta) = A(U).
    \end{equation}
    Substituing \eqref{eq:rotation-invariance-property-1} to the first term on the left-hand side of \eqref{eq:rotation-invariance-property-1-neg-3}, we obtain
    \begin{equation*}
        \cos\theta \, (\cos\theta \, A(U) + \sin\theta \, B(U)) - \sin\theta \, T^{-1}(\theta) \, B(T(\theta)U) \, T(\theta) = A(U),
    \end{equation*}
    which is equivalent to
    \begin{equation*}
        -\sin^2\theta \, A(U) + \cos\theta \sin\theta \, B(U) = \sin\theta \, T^{-1}(\theta) \, B(T(\theta)U) \, T(\theta).
    \end{equation*}
    Under the assumption that $\sin\theta \ne 0$, we can divide both sides of the above equation by $\sin\theta$ and obtain \eqref{eq:rotation-invariance-property-2}.

    Therefore, it suffices to prove the case of $\sin\theta = 0$, i.e., $\theta = 0$ or $\theta = \pi$. It is obvious that \eqref{eq:rotation-invariance-property-2} holds for $\theta = 0$. Next, we only need to prove the case of $\theta = \pi$, in which \eqref{eq:rotation-invariance-property-2} is reduced to
    \begin{equation*}
        -B(U) = T^{-1}(\pi) \, B(T(\pi)U) \, T(\pi).
    \end{equation*}

    Taking $\theta$ to be $\pi$ in \eqref{eq:rotation-invariance-property-1}, we have
    \begin{equation}\label{eq:rotation-invariance-property-1-pi}
        -A(U) = T^{-1}(\pi)  \, A(T(\pi)U)  \, T(\pi).
    \end{equation}
    Taking $\theta$ to be $\frac{\pi}{2}$ in \eqref{eq:rotation-invariance-property-1}, we have
    \begin{equation}\label{eq:rotation-invariance-property-1-pi-2}
        B(U) = T^{-1}\brac{\frac{\pi}{2}} \, A\brac{T\brac{\frac{\pi}{2}}U} \, T\brac{\frac{\pi}{2}}.
    \end{equation}
    Therefore, we compute
    \begin{equation*}
    \begin{aligned}
        T^{-1}(\pi) \, B(T(\pi)U) \, T(\pi) \stackrel{\text{\eqref{eq:rotation-invariance-property-1-pi-2}}}{=}{}& T^{-1}(\pi) \, T^{-1}\brac{\frac{\pi}{2}} \, A\brac{T\brac{\frac{\pi}{2}}T(\pi)U} \, T\brac{\frac{\pi}{2}} \, T(\pi) \\
        ={}& T^{-1}\brac{\frac{\pi}{2}} \, \brac{T^{-1}(\pi) \, A\brac{T(\pi)T\brac{\frac{\pi}{2}}U} \, T(\pi)} \, T(\frac{\pi}{2}) \\
        \stackrel{\text{\eqref{eq:rotation-invariance-property-1-pi}}}{=}{}& T^{-1}\brac{\frac{\pi}{2}} \brac{- A\brac{T\brac{\frac{\pi}{2}}U}} T\brac{\frac{\pi}{2}} \\
        \stackrel{\text{\eqref{eq:rotation-invariance-property-1-pi-2}}}{=}{}& -B(U).
    \end{aligned}
    \end{equation*}
    Here in the second equality, we have used the fact that $T(\alpha)T(\beta)=T(\beta)T(\alpha)$ for any $\alpha,\beta\in\mathbb{R}$. This completes the proof.
\end{proof}

Next, we show the proof of Proposition \ref{prop:rotational-invariance-system-unchanged}:
\begin{proof}[Proof of Proposition \ref{prop:rotational-invariance-system-unchanged}]
    Suppose that we have a Cartesian coordinate system $(x,y)$ and another Cartesian coordinate system $(x',y')$ in which the origin is fixed and the axes are rotated by an angle $\theta$ counterclockwise with respect to the original coordinate system. Suppose that an unknown function in the original system satisfying the governing equation \eqref{eq:general-first-order-system} is denoted by $U = U(t,X)$ with $X=(x,y)^T$. Then this function in the new coordinate system is represented by
    \begin{equation}\label{eq:rotated-unknown-function}
        U'(t,X') = T U(t, \, T_2^{-1}X'),
    \end{equation}
    where $T=T(\theta)$ and $T_2=T_2(\theta)$ are defined in \eqref{eq:rotation-matrix-T} and \eqref{eq:rotation-matrix-size-2-2} and $X'=(x',y')^T$. We can also write \eqref{eq:rotated-unknown-function} in an explicit form:
    \begin{equation*}
        U'(t,x',y') = T U(t,  \, \cos\theta \, x' - \sin\theta \, y',  \,  \sin\theta \, x' + \cos\theta \, y').
    \end{equation*}

    Next, we will prove that $U'=U'(t,X')$ also satisfies \eqref{eq:general-first-order-system}.
    By chain rule, we have
    \begin{equation*}
    \begin{aligned}
        \partial_t U'(t,X') &= T \partial_t U(t,T_2^{-1}X'), \\
        \partial_{x'} U'(t,X') &= T \brac{\cos\theta \, \partial_x U(t,T_2^{-1}X') + \sin\theta \, \partial_y U(t,T_2^{-1}X')}, \\
        \partial_{y'} U'(t,X') &= T \brac{-\sin\theta \, \partial_x U(t,T_2^{-1}X') + \cos\theta \, \partial_y U(t,T_2^{-1}X')}.
    \end{aligned}
    \end{equation*}
    Now we compute the governing equation in the new coordinate system:
    \begin{equation*}
    \begin{aligned}
        & \partial_t U' + A(U')\partial_{x'} U' + B(U')\partial_{y'} U' \\
        ={}& T \partial_t U(t,T_2^{-1}X') + A(U') T \brac{\cos\theta \, \partial_x U(t,T_2^{-1}X') + \sin\theta \, \partial_y U(t,T_2^{-1}X')} \\
        & + B(U') T \brac{-\sin\theta \, \partial_x U(t,T_2^{-1}X') + \cos\theta \, \partial_y U(t,T_2^{-1}X')} \\
        ={}& T \partial_t U + T \brac{T^{-1} A(U') T} \brac{\cos\theta \, \partial_x U + \sin\theta \, \partial_y U} \\
        & + T \brac{T^{-1} B(U') T} \brac{-\sin\theta \, \partial_x U + \cos\theta \, \partial_y U} \\        
        ={}& T \partial_t U + T \brac{\cos\theta A(U) + \sin\theta B(U)} \brac{\cos\theta \partial_x U + \sin\theta \partial_y U} \\
        & + T \brac{-\sin\theta A(U) + \cos\theta B(U)} \brac{-\sin\theta \partial_x U + \cos\theta \partial_y U} \\
        ={}& T \partial_t U + T A(U) \partial_x U + T B(U) \partial_y U \\
        ={}& T \brac{\partial_t U(t,T_2^{-1}X') + A(U) \partial_x U(t,T_2^{-1}X') + B(U) \partial_y U(t,T_2^{-1}X')} \\
        ={}& T S(U,T_2^{-1}X') \\
        ={}& S(U',X').
    \end{aligned}
    \end{equation*}
    Here we have used the rotation invariance property in Definition \ref{def:rotation-invariance} and Lemma \ref{lem:rotation-invariance-property-2}. This completes the proof.
\end{proof}

\section{Proof of Proposition \ref{prop:identity}}\label{appendix:proof-equality}

\begin{proof}
    We prove the equalities in Proposition \ref{prop:identity} by direct calculations:
\begin{enumerate}
\item
\begin{equation*}
\textrm{LHS} = \ct \, (\cos\theta  \, u + \sin\theta  \, v) - \st \, (-\sin\theta  \, u + \cos\theta  \, v) = u = \textrm{RHS}.
\end{equation*}

\item
\begin{equation*}
\textrm{LHS} = \st \, (\cos\theta  \, u + \sin\theta  \, v) + \ct \, (-\sin\theta  \, u + \cos\theta  \, v) = v = \textrm{RHS}.
\end{equation*}

\item 
\begin{equation*}
\begin{aligned} 
\textrm{LHS} ={}& 
\cos\theta \, (\cos\theta  \, u + \sin\theta  \, v)^2 - \sin\theta \, (\cos\theta  \, u + \sin\theta  \, v)(-\sin\theta  \, u + \cos\theta  \, v) \\
={}& (\cos\theta  \, u + \sin\theta  \, v)\brac{\cos\theta \, (\cos\theta  \, u + \sin\theta  \, v) - \sin\theta \, (-\sin\theta  \, u + \cos\theta  \, v)} \\
={}& u(\cos\theta  \, u + \sin\theta  \, v) = \textrm{RHS}.
\end{aligned}
\end{equation*}

\item
\begin{equation*}
\begin{aligned} 
\textrm{LHS} ={}& 
\cos\theta(\cos\theta  \, u + \sin\theta  \, v)^2 - \sin\theta(\cos\theta  \, u + \sin\theta  \, v)(-\sin\theta  \, u + \cos\theta \,  v) \\
={}& (\cos\theta  \, u + \sin\theta  \, v)\brac{\cos\theta(\cos\theta \,  u + \sin\theta  \, v) - \sin\theta(-\sin\theta  \, u + \cos\theta  \, v)} \\
={}& u(\cos\theta  \, u + \sin\theta \,  v) = \textrm{RHS}.
\end{aligned}
\end{equation*}

\item
\begin{equation*}
\begin{aligned}
\textrm{LHS}
={}& 2 \cos\theta  \, (\cos^2\theta  \, u\alpha + \cos\theta \, \sin\theta \, (u\beta+v\alpha) + \sin^2\theta  \,  v\beta) \\
& - \sin\theta \, \brac{\cos^2\theta \,  u\beta + \cos\theta \, \sin\theta \, (v\beta-u\alpha) - \sin^2\theta  \, v\alpha} \\
& - \sin\theta \, \brac{\cos^2\theta v\alpha + \cos\theta \, \sin\theta \, (-u\alpha+v\beta) - \sin^2\theta u\beta} \\
={}& u\alpha\cos\theta \, (2\cos^2\theta+\sin^2\theta+\sin^2\theta) + v\alpha\sin\theta \, (2\cos^2\theta+\sin^2\theta-\cos^2\theta) \\
& + u\beta\sin\theta \, (2\cos^2\theta-\cos^2\theta+\sin^2\theta) + v\beta\cos\theta(2\sin^2\theta-\sin^2\theta-\sin^2\theta) \\
={}& 2u\alpha\cos\theta + (u\beta+v\alpha)\sin\theta \\
={}& \textrm{RHS}
\end{aligned}
\end{equation*}

\item
\begin{equation*}
\begin{aligned}
\textrm{LHS}
={}& 2 \sin\theta  \, (\cos^2\theta \,  u\alpha + \cos\theta \, \sin\theta \, (u\beta+v\alpha) + \sin^2\theta v\beta) \\
& + \cos\theta \, \brac{\cos^2\theta  \, u\beta + \cos\theta \, \sin\theta \, (v\beta-u\alpha) - \sin^2\theta  \, v\alpha} \\
& + \cos\theta \, \brac{\cos^2\theta  \, v\alpha + \cos\theta \, \sin\theta \, (-u\alpha+v\beta) - \sin^2\theta  \, u\beta} \\
={}& u\alpha\sin\theta \, (2\cos^2\theta \, -\cos^2\theta-\cos^2\theta) + v\alpha\cos\theta \, (2\sin^2\theta-\sin^2\theta+\cos^2\theta) \\
& + u\beta\cos\theta(2\sin^2\theta+\cos^2\theta-\sin^2\theta) + v\beta\sin\theta(2\sin^2\theta+\cos^2\theta+\cos^2\theta) \\
={}&
(u\beta+v\alpha)\cos\theta + 2v\beta\sin\theta = \textrm{RHS}
\end{aligned}
\end{equation*}

\end{enumerate}
\end{proof}

\section{Proof of Lemma \ref{lem:rotational-invariance-block-matrix-linear}}\label{app:proof-lem-rotational-invariance-block-matrix-linear}

\subsection{Proof of the first case}

For the first case, $A(V)$ only has two non-zero entries in the first column:
\begin{equation}
    A(V) = 
    \begin{pmatrix}
        a_{11}(V) & 0 \\
        a_{21}(V) & 0
    \end{pmatrix}.
\end{equation}
From the necessary condition given in Lemma \ref{lem:rotational-invariance-block-matrix-necessary-condition}, we have that $B(V)$ only has two non-zero entries in the second column:
\begin{equation}
    B(V) = 
    \begin{pmatrix}
        0 & b_{12}(V) \\
        0 & b_{22}(V)
    \end{pmatrix}.
\end{equation}
Now we compute
\begin{equation*}
\begin{aligned}
    T_2^{-1} A(T_2 V) T_2
    &=
    \begin{pmatrix}
        \cos\theta & -\sin\theta \\
        \sin\theta & \cos\theta
    \end{pmatrix}
    \begin{pmatrix}
        a_{11}(T_2 V) & 0 \\
        a_{21}(T_2 V) & 0
    \end{pmatrix}
    \begin{pmatrix}
        \cos\theta & \sin\theta \\
        -\sin\theta & \cos\theta
    \end{pmatrix}
    \\
    &=
    \begin{pmatrix}
        \cos\theta (\cos\theta a_{11}(T_2 V)-\sin\theta a_{21}(T_2 V)) & \sin\theta (\cos\theta a_{11}(T_2 V)-\sin\theta a_{21}(T_2 V)) \\
        \cos\theta (\sin\theta a_{11}(T_2 V)+\cos\theta a_{21}(T_2 V)) & \sin\theta (\sin\theta a_{11}(T_2 V)+\cos\theta a_{21}(T_2 V))
    \end{pmatrix}
\end{aligned}
\end{equation*}
which should be equal to
\begin{equation}
    \cos\theta A(V) + \sin\theta B(V) = 
    \begin{pmatrix}
        \cos\theta a_{11}(V) & \sin\theta b_{12}(V) \\
        \cos\theta a_{21}(V) & \sin\theta b_{22}(V)
    \end{pmatrix}.
\end{equation}
Then we have
\begin{equation}
\begin{aligned}
    \cos\theta (\cos\theta \, a_{11}(T_2V)-\sin\theta  \, a_{21}(T_2V)) &= \cos\theta  \, a_{11}(V), \\ 
    \sin\theta (\cos\theta  \, a_{11}(T_2V)-\sin\theta  \, a_{21}(T_2V)) &= \sin\theta  \, b_{12}(V), \\
    \cos\theta (\sin\theta  \, a_{11}(T_2V)+\cos\theta  \, a_{21}(T_2V)) &= \cos\theta  \, a_{21}(V), \\
    \sin\theta (\sin\theta  \, a_{11}(T_2V)+\cos\theta  \, a_{21}(T_2V)) &= \sin\theta  \, b_{22}(V).
\end{aligned}
\end{equation}
The above equations can be further simplified as
\begin{equation}\label{eq:simplified-constraint-case-1}
    \begin{aligned}
        \cos\theta a_{11}(T_2V)-\sin\theta a_{21}(T_2V) &=  a_{11}(V), \\ 
        \sin\theta a_{11}(T_2V)+\cos\theta a_{21}(T_2V) &=  a_{21}(V),
    \end{aligned}
\end{equation}
and
\begin{equation}
    \begin{aligned}
        a_{11}(V) &= b_{12}(V), \\
        a_{21}(V) &= b_{22}(V).
    \end{aligned}
\end{equation}

Since $a_{11}(V)$ and $a_{21}(V)$ are linear functions of $V$, we have
\begin{equation}\label{eq:linear-function-of-U-case-1}
    \begin{aligned}
        a_{11}(V) &= c_1 p + c_2 q, \\
        a_{21}(V) &= c_3 p + c_4 q,
    \end{aligned}
\end{equation}
where $c_1, c_2, c_3, c_4$ are constants.
Plugging \eqref{eq:linear-function-of-U-case-1} into \eqref{eq:simplified-constraint-case-1}, we have
\begin{equation*}
\begin{aligned}    
    \cos\theta (c_1 (\cos\theta p + \sin\theta q) + c_2 (-\sin\theta p + \cos\theta q))-\sin\theta (c_3 (\cos\theta p + \sin\theta q) 
    &+ c_4 (-\sin\theta p + \cos\theta q)) \\
    &=  c_1 p + c_2 q \\
    \sin\theta (c_1 (\cos\theta p + \sin\theta q) + c_2 (-\sin\theta p + \cos\theta q))+\cos\theta (c_3 (\cos\theta p + \sin\theta q) &+ c_4 (-\sin\theta p + \cos\theta q) \\
    &=  c_3 p + c_4 q    
\end{aligned}
\end{equation*}
which can be simplified as
\begin{equation*}
    \begin{aligned}        
    (c_1\cos^2\theta - (c_2+c_3)\cos\theta\sin\theta + c_4\sin^2\theta) p + (c_2\cos^2\theta + (c_1-c_4)\cos\theta\sin\theta - c_3\sin^2\theta) q &= c_1 p + c_2 q \\
    (c_3\cos^2\theta + (c_1-c_4)\cos\theta\sin\theta - c_2\sin^2\theta) p + (c_4\cos^2\theta - (c_2+c_3)\cos\theta\sin\theta + c_1\sin^2\theta) q &= c_3 p + c_4 q
    \end{aligned}
\end{equation*}
This implies
\begin{equation}
    \begin{aligned}
        c_1\cos^2\theta - (c_2+c_3)\cos\theta\sin\theta + c_4\sin^2\theta &= c_1 \\
        c_2\cos^2\theta + (c_1-c_4)\cos\theta\sin\theta - c_3\sin^2\theta &= c_2
    \end{aligned}
\end{equation}
Then we derive the following relations:
\begin{equation}
    c_1 = c_4, \quad c_2 = -c_3
\end{equation}
Therefore, we have
\begin{equation}
    \begin{aligned}
        a_{11}(V) &= c_1 p - c_2 q \\
        a_{21}(V) &= c_2 p + c_1 q
    \end{aligned}
\end{equation}
Thus, the matrices $A(V)$ and $B(V)$ are
\begin{equation}
    A(V) = 
    \begin{pmatrix}
        c_1 p + c_2 q & 0\\
        -c_2 p + c_1 q & 0
    \end{pmatrix}
    =
    c_1
    \begin{pmatrix}
        p & 0\\
        q & 0
    \end{pmatrix}
    +
    c_2
    \begin{pmatrix}
        q & 0 \\
        -p & 0
    \end{pmatrix},
\end{equation}
and
\begin{equation}
    B(V) = 
    \begin{pmatrix}
        0 & c_1 p + c_2 q \\
        0 & -c_2 p + c_1 q
    \end{pmatrix}
    =    
    c_1
    \begin{pmatrix}
        0 & p \\
        0 & q
    \end{pmatrix}
    +
    c_2
    \begin{pmatrix}
        0 & q \\
        0 & -p
    \end{pmatrix}.
\end{equation}

\subsection{Proof of the second case}

In this case, we assume that $A(V)$ is a diagonal matrix
\begin{equation*}
    A(V) = 
    \begin{pmatrix}
        a_{11}(V) & 0\\
        0 & a_{22}(V)
    \end{pmatrix}.
\end{equation*}
From Lemma \ref{lem:rotational-invariance-block-matrix-necessary-condition}, we have $B(V)$ is also a diagonal matrix
\begin{equation*}
    B(V) =
    \begin{pmatrix}
        b_{11}(V) & 0\\
        0 & b_{22}(V)
    \end{pmatrix}.
\end{equation*}
Then we compute
\begin{equation*}
    \begin{aligned}        
    T_2^{-1}A(T_2V)T_2 &=
    \begin{pmatrix}
        \cos\theta & -\sin\theta \\
        \sin\theta & \cos\theta
    \end{pmatrix}
    \begin{pmatrix}
        a_{11}(T_2V) & 0\\
        0 & a_{22}(T_2V)
    \end{pmatrix}
    \begin{pmatrix}
        \cos\theta & \sin\theta \\
        -\sin\theta & \cos\theta
    \end{pmatrix}
    \\
    &= 
    \begin{pmatrix}
        \cos^2\theta a_{11}(T_2V) + \sin^2\theta a_{22}(T_2V) & \cos\theta\sin\theta(a_{11}(T_2V) - a_{22}(T_2V)) \\
        \cos\theta\sin\theta(a_{11}(T_2V) - a_{22}(T_2V)) & \sin^2\theta a_{11}(T_2V) + \cos^2\theta a_{22}(T_2V)
    \end{pmatrix},
    \end{aligned}
\end{equation*}
which should be equal to
\begin{equation*}
    \cos\theta A(V) + \sin\theta B(V) =
    \begin{pmatrix}
        \cos\theta a_{11}(V) + \sin\theta b_{11}(V) & 0\\
        0 & \cos\theta a_{22}(V) + \sin\theta b_{22}(V)
    \end{pmatrix}.
\end{equation*}
Then we have
\begin{equation}
\begin{aligned}
    a_{11}(T_2V) &= a_{22}(T_2V), \\
    \cos^2\theta a_{11}(T_2V) + \sin^2\theta a_{22}(T_2V) &= \cos\theta a_{11}(V) + \sin\theta b_{11}(V), \\
    \sin^2\theta a_{11}(T_2V) + \cos^2\theta a_{22}(T_2V) &= \cos\theta a_{22}(V) + \sin\theta b_{22}(V).
\end{aligned}
\end{equation}
This is reduced to
\begin{equation}
\begin{aligned}
    a_{11}(V) &= a_{22}(V), \\
    b_{11}(V) &= b_{22}(V), \\
    a_{11}(T_2V) &= \cos\theta a_{11}(V) + \sin\theta b_{11}(V).
\end{aligned}    
\end{equation}

Next, we assume the linear functions
\begin{equation}
\begin{aligned}
    a_{11}(V) = a_{22}(V) &= c_1 p + c_2 q, \\
    b_{11}(V) = b_{22}(V) &= c_3 p + c_4 q,
\end{aligned}    
\end{equation}
where $c_1, c_2, c_3, c_4$ are constants. Then we have
\begin{equation}
\begin{aligned}
    c_1(\cos\theta p + \sin\theta q) + c_2(-\sin\theta p + \cos\theta q) = \cos\theta (c_1 p + c_2 q) + \sin\theta (c_3 p + c_4 q),
\end{aligned}    
\end{equation}
which means that
\begin{equation}
    c_1 = c_4, \quad c_2 = -c_3.
\end{equation}
Therefore, we have
\begin{equation}
    A(V) = 
    \begin{pmatrix}
        c_1 p + c_2 q & 0\\
        0 & c_1 p + c_2 q
    \end{pmatrix}
    = c_1
    \begin{pmatrix}
        p & 0\\
        0 & p
    \end{pmatrix}
    + c_2
    \begin{pmatrix}
        q & 0\\
        0 & q
    \end{pmatrix},
\end{equation}
and
\begin{equation}
    B(V) =
    \begin{pmatrix}
        c_3 p + c_4 q & 0\\
        0 & c_3 p + c_4 q
    \end{pmatrix}
    =
    -c_2
    \begin{pmatrix}
        p & 0\\
        0 & p
    \end{pmatrix}
    + c_1
    \begin{pmatrix}
        q & 0\\
        0 & q
    \end{pmatrix}.
\end{equation}

\subsection{Proof of the third case}

In the last case, we assume that $A(V)$ only has non-zero entries in the second column:
\begin{equation}
    A(V) =
    \begin{pmatrix}
        0 & a_{12}(V) \\
        0 & a_{22}(V) \\
    \end{pmatrix},
\end{equation}
and from Lemma \ref{lem:rotational-invariance-block-matrix-necessary-condition}, we have
\begin{equation}
    B(V) = 
    \begin{pmatrix}
        b_{11}(V) & 0\\
        b_{21}(V) & 0
    \end{pmatrix}.
\end{equation}
We compute
\begin{equation*}
\begin{aligned}    
    T_2^{-1} A(T_2V) T_2 &=
    \begin{pmatrix}
        \cos\theta & -\sin\theta \\
        \sin\theta & \cos\theta
    \end{pmatrix}
    \begin{pmatrix}
        0 & a_{12}(T_2V)\\
        0 & a_{22}(T_2V)\\
    \end{pmatrix}
    \begin{pmatrix}
        \cos\theta & \sin\theta \\
        -\sin\theta & \cos\theta
    \end{pmatrix}
    \\
    & =
    \begin{pmatrix}
        0 & \cos\theta a_{12}(T_2V) - \sin\theta a_{22}(T_2V)\\
        0 & \sin\theta a_{12}(T_2V) + \cos\theta a_{22}(T_2V)        
    \end{pmatrix}
    \begin{pmatrix}
        \cos\theta & \sin\theta \\
        -\sin\theta & \cos\theta
    \end{pmatrix}
    \\
    &=
    \begin{pmatrix}
        -\sin\theta(\cos\theta a_{12}(T_2V) - \sin\theta a_{22}(T_2V)) & \cos\theta(\cos\theta a_{12}(T_2V) - \sin\theta a_{22}(T_2V))\\
        -\sin\theta(\sin\theta a_{12}(T_2V) + \cos\theta a_{22}(T_2V)) & \cos\theta(\sin\theta a_{12}(T_2V) + \cos\theta a_{22}(T_2V))
    \end{pmatrix}    
\end{aligned}    
\end{equation*}
which should be equal to
\begin{equation}
    \cos\theta A(V) + \sin\theta B(V) =
    \begin{pmatrix}
        \sin\theta b_{11}(V) & \cos\theta a_{12}(V) \\
        \sin\theta b_{21}(V) & \cos\theta a_{22}(V)
    \end{pmatrix}.
\end{equation}

Therefore, we have
\begin{equation}
\begin{aligned}
    -(\cos\theta a_{12}(T_2V) - \sin\theta a_{22}(T_2V)) &= b_{11}(V) \\
    \cos\theta a_{12}(T_2V) - \sin\theta a_{22}(T_2V) &= a_{12}(V) \\
    -(\sin\theta a_{12}(T_2V) + \cos\theta a_{22}(T_2V)) &= b_{21}(V) \\
    \sin\theta a_{12}(T_2V) + \cos\theta a_{22}(T_2V) &= a_{22}(V)
\end{aligned}    
\end{equation}
which is reduced to
\begin{equation}
\begin{aligned}
    b_{11}(V) &= -a_{12}(V) \\
    b_{21}(V) &= -a_{22}(V) \\
    \cos\theta a_{12}(T_2V) - \sin\theta a_{22}(T_2V) &= a_{12}(V) \\
    \sin\theta a_{12}(T_2V) + \cos\theta a_{22}(T_2V) &= a_{22}(V)    
\end{aligned}    
\end{equation}
Next, we assume the linear function
\begin{equation}
\begin{aligned}
    a_{12}(V) &= c_1 p + c_2 q, \\
    a_{22}(V) &= c_3 p + c_4 q,
\end{aligned}
\end{equation}
where $c_1, c_2, c_3, c_4$ are constants.
Then we have
\begin{equation*}
\begin{aligned}
    \cos\theta (c_1 (\cos\theta p + \sin\theta q) + c_2 (-\sin\theta p + \cos\theta q)) - \sin\theta (c_3 (\cos\theta p + \sin\theta q) &+ c_4 (-\sin\theta p + \cos\theta q)) \\
    &= c_1 p + c_2 q, \\
    \sin\theta (c_1 (\cos\theta p + \sin\theta q) + c_2 (-\sin\theta p + \cos\theta q)) + \cos\theta (c_3 (\cos\theta p + \sin\theta q) &+ c_4 (-\sin\theta p + \cos\theta q)) \\
    &= c_3 p + c_4 q,
\end{aligned}
\end{equation*}
from which we solve out 
\begin{equation}
    c_1 = c_4, \quad c_2 = -c_3.
\end{equation}

Therefore, we have
\begin{equation}
    A(V) = 
    \begin{pmatrix}
        0 & c_1 p + c_2 q \\
        0 & -c_2 p + c_1 q
    \end{pmatrix}
    = c_1
    \begin{pmatrix}
        0 & p \\
        0 & q
    \end{pmatrix}
    + c_2
    \begin{pmatrix}
        0 & q \\
        0 & -p
    \end{pmatrix},
\end{equation}
and
\begin{equation}
    B(V) =
    \begin{pmatrix}
        -c_1 p - c_2 q & 0 \\
        c_2 p - c_1 q & 0
    \end{pmatrix}
    = c_1
    \begin{pmatrix}
        -p & 0 \\
        -q & 0
    \end{pmatrix}
    + c_2
    \begin{pmatrix}
        -q & 0 \\
        p & 0
    \end{pmatrix}.
\end{equation}

\end{document}